\documentclass{article}

\usepackage{microtype}
\usepackage{graphicx}
\usepackage{subcaption}
\usepackage{booktabs} 

\usepackage{hyperref}

\usepackage[accepted]{icml2026}

\usepackage{amsmath}
\usepackage{amssymb}
\usepackage{mathtools}
\usepackage{amsthm}
\usepackage{bm}
\usepackage{booktabs}
\usepackage{multirow}
\usepackage{pifont}
\usepackage{caption}
\usepackage{enumitem}

\usepackage[capitalize,noabbrev]{cleveref}

\theoremstyle{plain}
\newtheorem{theorem}{Theorem}[section]
\newtheorem{proposition}[theorem]{Proposition}
\newtheorem{lemma}[theorem]{Lemma}

\theoremstyle{definition}

\newtheorem{assumption}[theorem]{Assumption}
\theoremstyle{remark}

\usepackage[textsize=tiny]{todonotes}

\icmltitlerunning{Proximal-Based Generative Modeling for Bayesian Inverse Problems}

\begin{document}
	
	\twocolumn[
	\icmltitle{Proximal-Based Generative Modeling for Bayesian Inverse Problems}
	
	
	
	\icmlsetsymbol{equal}{*}
	
	\begin{icmlauthorlist}
		\icmlauthor{Boyang Zhang}{yyy1,asy1}
		\icmlauthor{Zhiguo Wang}{yyy2}
		\icmlauthor{Ya-Feng Liu}{lab1,yyy3}
	\end{icmlauthorlist}
	
	\icmlaffiliation{yyy1}{School of Advanced Interdisciplinary Sciences, University of Chinese Academy of Sciences, Beijing, China}
	\icmlaffiliation{yyy2}{School of Mathematics, Sichuan University, Chengdu, China}
	\icmlaffiliation{yyy3}{School of Mathematical Sciences, Beijing University of Posts and Telecommunications, Beijing, China}
	\icmlaffiliation{asy1}{Academy of Mathematics and Systems Science, Chinese Academy of Sciences, Beijing, China}
	\icmlaffiliation{lab1}{Ministry of Education Key Laboratory of Mathematics and Information Networks, Beijing, China}
	
	\icmlcorrespondingauthor{Ya-Feng Liu}{yafengliu@bupt.edu.cn}
	
	\icmlkeywords{Generative Modeling, Optimization, Proximal Mapping, Diffusion Models}
	
	\vskip 0.3in
	]
	
	
	
	\printAffiliationsAndNotice{}  
	
	\begin{abstract}
		Score-based diffusion models demonstrate superior performance in generative tasks but encounter fundamental bottlenecks in inverse problems due to the analytical intractability of the time-dependent likelihood score. To bridge this gap, we propose a novel proximal-based generative modeling (PGM) framework that rigorously circumvents explicit likelihood evaluation. Our framework is built upon a theoretical equivalence between Gaussian convolution in diffusion processes and Moreau-Yosida regularization in nonsmooth optimization.  This enables a new sampling mechanism driven by the proposed Moreau score, which admits a closed-form expression via proximal operators. Moreover, we introduce Moreau score matching to learn the proximal operators that rely solely on samples drawn from the prior distribution. Theoretically, PGM eliminates the early-stopping bias inherent in the score-based diffusion model and achieves non-asymptotic convergence. Experiments demonstrate that PGM significantly surpasses state-of-the-art methods in reconstruction quality and sampling time.

	\end{abstract}
	
	\section{Introduction}
	
	\begin{figure*}[t]
		\centering
		\includegraphics[width=0.98\linewidth]{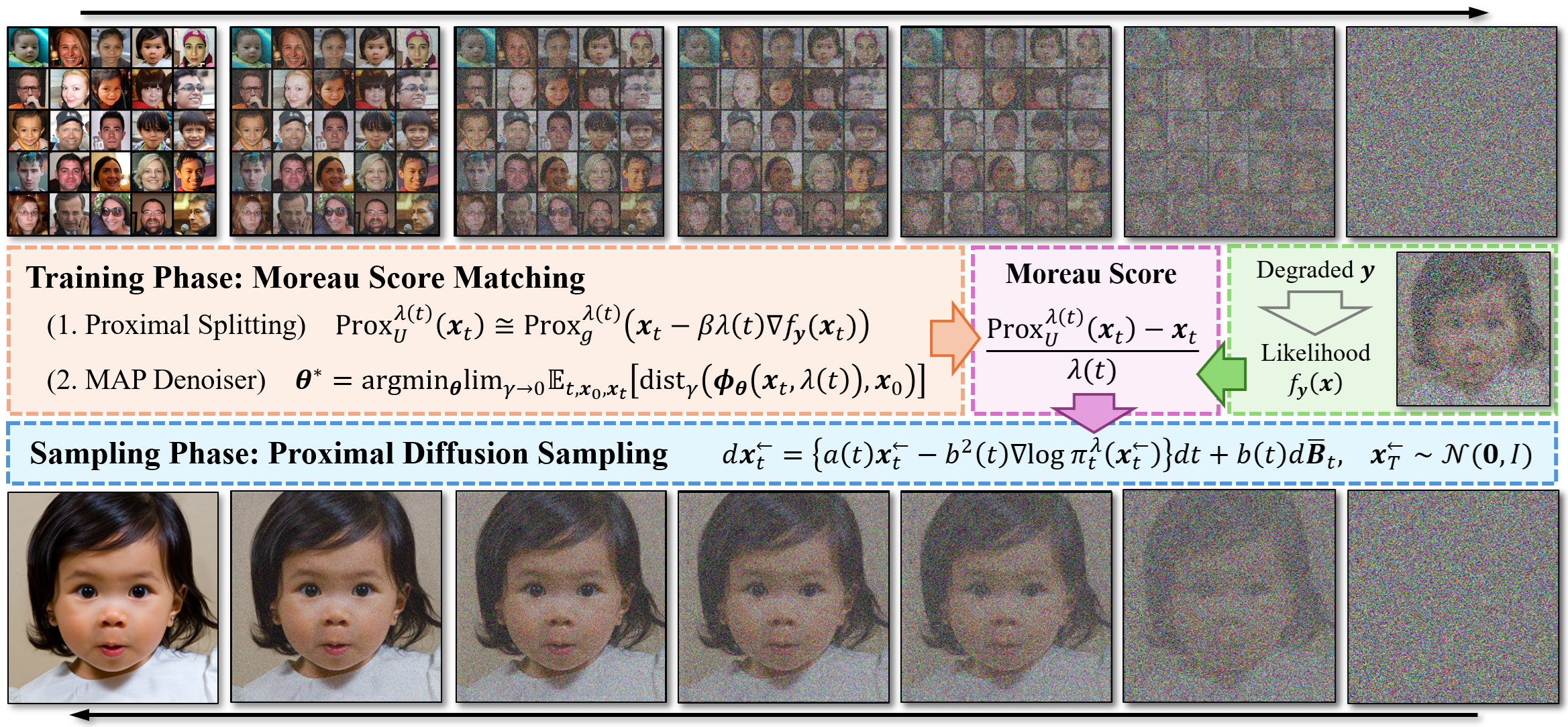}
		\caption{\textbf{A sketch map for PGM.} Training phase: a proximal splitting is applied to provide an approximation of the Moreau score, where a network is trained to learn the proximal operator in an unsupervised manner. Sampling phase: the traditional score function is replaced by the Moreau score, which admits an explicit, smooth, and asymptotically equivalent formulation via proximal operators. }
		\label{fig:sketch_map}
		\vspace{-5pt}
	\end{figure*}

	Inverse problems entail the estimation of an unknown signal of interest, $\bm{x} \in \mathbb{R}^d$, from observed measurement data, $\bm{y} \in \mathbb{R}^m$. These problems are ubiquitous in signal processing \cite{tropp2010computational,ma2018sparse,hu2020deep,zhang2023physics}, computational imaging \cite{jin2017deep,bertero2021introduction,daras2024survey}, and communications \cite{soltani2019deep,yu2022role,liu2024survey}. From a Bayesian perspective, solving an inverse problem is equivalent to sampling from a posterior distribution of the composite form \cite{mou2022efficient,renaud2025stability}:
	\begin{equation}
		\label{eq:composite}
		\pi(\bm{x}) \propto \exp\left\{-U(\bm{x})\right\}~\text{with}~U(\bm{x}) = \beta f_{\bm{y}}(\bm{x}) + g(\bm{x}),
	\end{equation}
	where $\beta$ is a hyperparameter, $f_{\bm{y}}(\bm{x})$ represents the data fidelity derived from the likelihood $p(\bm{y} |\bm{x})$ and $g(\bm{x})$ represents the prior derived from $p(\bm{x})$ which may be nonsmooth.
	
	To address the Bayesian inverse problem, two primary methodological paradigms have been developed: optimization-based methods \cite{haber2000optimization,afonso2010augmented,ye2019optimization} and sampling-based methods \cite{stuart2010inverse,knapik2011bayesian,ehrhardt2024proximal,janati2025bridging,habring2026diffusion}. Optimization-based methods seek to reconstruct the unknown signal by identifying the most likely one under the posterior distribution, which leads to the maximum a posteriori (MAP) estimator. Although optimization-based methods provide a single point estimate, they do not capture the full structure of the posterior distribution, and thus may fail to reflect uncertainty. In contrast, sampling-based methods aim to generate a collection of samples from the posterior distribution, enabling principled uncertainty quantification.
	
	Sampling from such complex, high-dimensional distributions remains a fundamental challenge \cite{welling2011bayesian,xu2018global,song2020score}. Traditional Markov chain Monte Carlo (MCMC) methods, such as the unadjusted Langevin algorithm (ULA) \cite{durmus2017nonasymptotic} and the Metropolis-adjusted Langevin algorithm (MALA) \cite{roberts1996exponential}, offer theoretical guarantees for smooth distributions \cite{ma2019sampling}. While recent advances in proximal \cite{pereyra2016proximal,salim2019stochastic,pillai2024optimal} and mirror Langevin dynamics \cite{hsieh2018mirrored,ahn2021efficient} have extended these capabilities to nonsmooth settings, they rely on explicit knowledge of the potential functions. In the literature, the traditional approach of using hand-crafted, explicit priors is slowly being replaced by rich, learned priors \cite{kingma2013auto,goodfellow2014generative,ho2020denoising,chung2022diffusion}. Despite their strong ability to approximate complex prior distributions, these learned implicit priors make it hard to generate samples from the posterior distribution.
	
	Recent years have witnessed considerable interest in leveraging diffusion models as structural priors for solving inverse problems \cite{chung2022diffusion,song2023solving,mardani2023variational,alkhouri2024sitcom,chang2025provable,zhang2026gradient,habring2026diffusion,yuan2026improving}. The emergence of score-based diffusion models has marked a paradigm shift, demonstrating the ability to learn generative models directly from data without requiring an explicit analytical form of the target density \cite{ho2020denoising,song2020score}. By estimating the score function via denoising score matching and simulating reverse-time stochastic differential equations (SDEs), these methods have achieved state-of-the-art performance in image generation \cite{dhariwal2021diffusion,rombach2022high,song2023solving,tolooshams2025equireg}. Despite this success, applying diffusion models to Bayesian inverse problems remains challenging in practice, due to the difficulty of accurately estimating the time-dependent posterior score $\nabla \log p_t(\bm{x}_t |\bm{y})$ \cite{meng2022diffusion,kawar2022denoising}. As a direct extension, conditional diffusion models explicitly train a score network architecture $\bm{s}_{\bm{\theta}}(\bm{x}_t,t,\bm{y})$ to learn the conditional score $\nabla_{\bm{x}} \log p_t(\bm{x}_t|\bm{y})$ with paired training data \cite{saharia2022palette,saharia2022image}.
	
	To enhance the generalizability and reduce training costs, zero-shot methods modify the sampling trajectory to enforce measurement consistency, without requiring any retraining or fine-tuning of the score network parameters. Specifically, by Bayes' rule,
	\begin{equation}
		\label{eqn_pos_g}
		\nabla_{\bm{x}} \log p_t(\bm{x}_t|\bm{y}) = \underbrace{\nabla_{\bm{x}} \log p_t(\bm{x}_t)}_{\text{prior score}} + \underbrace{\nabla_{\bm{x}} \log p(\bm{y}|\bm{x}_t)}_{\text{log-likelihood}},
	\end{equation}
	where the time-dependent log-likelihood  $\nabla_{\bm{x}} \log p(\bm{y}| \bm{x}_t)$ generally lacks a closed-form expression. Numerous studies have focused on approximating this intractable term \cite{song2023pseudoinverse,yu2023freedom,li2025efficient}. For linear inverse problems, subspace and null-space projection is applied to modify the sampling trajectory by forcing the intermediate sample to directly satisfy the measurement consistency \cite{lugmayr2022repaint,kawar2022denoising,wang2022zero,zirvi2025diffusion}. For nonlinear inverse problems, a series of approximations on $p(\bm{x}_0| \bm{x}_t)$ is applied to obtain the time-dependent log-likelihood \cite{chung2022diffusion,song2023pseudoinverse,boys2023tweedie,wu2024diffusion,xu2025rethinking}. However, these approximations require expensive iterative computations during sampling and lack theoretical convergence guarantees.
	
	For Bayesian inverse problems, the primary objective is to efficiently and accurately sample from the posterior distribution $\pi(\bm x)$. This goal does not necessarily require using the time-dependent posterior score in the reverse-time SDE. A natural and fundamental question therefore arises:
	
	\textit{Can we construct an alternative score function that replaces the time-dependent posterior score while still guaranteeing correct sampling from the posterior distribution?}
	
	Motivated by proximal-based sampling methods, we develop a new generative modeling paradigm, a novel proximal-based generative modeling (PGM) framework, that bypasses the computation of the intractable time-dependent likelihood term.
	
	Our contributions are threefold:
	
	\begin{itemize}[topsep=0pt, itemsep=0pt]
		\item  \textbf{A novel framework of bridging optimization and diffusion.} We establish a rigorous theoretical connection between the posterior score and the proposed Moreau score. Leveraging this relationship, we introduce a proximal-based generative modeling and derive a new reverse-time SDE driven explicitly by the Moreau score, which admits a closed-form expression via proximal operators and can be evaluated efficiently. 
		\item \textbf{Unsupervised Moreau score matching.} We introduce Moreau score matching, a training objective that learns the Moreau score for unknown priors in a completely unsupervised manner. Once the proximal operator is learned, it can be used as a pre-trained operator for any sampling processes or further fine-tuning. This allows us to perform posterior sampling via proximal diffusion sampling, unifying score-based generative modeling with proximal Langevin dynamics. 
		\item \textbf{Improved theoretical and experimental results.} We extend classical convergence results for diffusion models to the context of solving inverse problems. Notably, our convergence guarantees are stronger, as they successfully eliminate the early-stopping error based on the properties of the Moreau approximation. Extensive experiments demonstrate that PGM consistently outperforms state-of-the-art methods in generation quality, memory consumption, and sampling time.
	\end{itemize}
	
	\section{Background}
	\label{sec:background}
	
	\subsection{Score-Based Diffusion Models}
	
	We begin with a brief review of score-based diffusion models \cite{song2020score}. The forward process progressively perturbs data into noise via Gaussian kernels, a transformation described by the SDE of the following form:
	\begin{equation}
		\label{eqn_forward}
		d\bm{x}_t^{\rightarrow} = a(t)\bm{x}_t^{\rightarrow}dt + b(t)d\bm{B}_t ,
	\end{equation}
	where $\bm{B}_t$ denotes a standard Wiener process. The goal of diffusion models is to learn the corresponding reverse SDE, which takes the form
	\begin{equation}
		\label{eqn_reverse}
		d\bm{x}_t^{\leftarrow} = \left\{a(t)\bm{x}_t^{\leftarrow}-b^2(t)\nabla \log p_t(\bm{x}_t^{\leftarrow})\right\}dt + b(t)d\bar{\bm{B}}_t ,
	\end{equation}
	where $d\bar{\bm{B}}_t$ is a reverse-time Wiener process and $\nabla \log p_t(\bm{x}_t^{\leftarrow})$ is the score function. In practice, the score is approximated by a neural network $\bm{s}_{\bm{\theta}}(\bm{x}, t)$ with parameters $\bm{\theta}$, trained via denoising score matching \cite{vincent2011connection}: 
	\begin{equation}
		\label{eq:denoise}
		\begin{aligned}
			\bm{\theta}^{*}=
			\mathop{\mathop{\mathop{\arg\min}}}\limits_{\bm{\theta}} \mathbb{E}
			\left[ \| \bm{s}_{\bm{\theta}}(\bm{x}_t,t)-\nabla \log p_t(\bm{x}_t|\bm{x}_0) \|_2^2 \right],
		\end{aligned}
	\end{equation}
	where the expectation is taken over $t\sim \mathcal{U}[0, T]$, $\bm x_t\sim p(\bm x_t|\bm x_0)$, and $\bm x_0\sim p_0(\bm{x}_0)$. Given a trained score model, the reverse-time SDE is simulated with standard numerical solvers (e.g., Euler--Maruyama) to generate samples.

	\subsection{Diffusion Models for Inverse Problems}
	
	Consider a given measurement $\bm{y}\in\mathbb{R}^m$ of the form
	\begin{equation}
		\label{eqn_inverse}
		\bm{y} = \mathcal{A}(\bm{x})+\bm{\xi},
	\end{equation}
	where $\mathcal{A}:\mathbb{R}^d\to\mathbb{R}^m$ denotes a forward measurement operator, $\bm{x}\in\mathbb{R}^d$ is an unknown signal, and $\bm{\xi}\sim\mathcal{N}(\bm{0},\sigma^2_{\xi}I)$ is the additive Gaussian noise. 
	
	If $\mathcal{A}$ is a linear operator, it can be represented by a measurement matrix $A \in \mathbb{R}^{m \times d}$. The forward model becomes $\bm{y} = A\bm{x}+\bm{\xi}$. If $\mathcal{A}$ is a nonlinear function, the Jacobian $\nabla \mathcal{A}(\bm{x})$ depends on the state $\bm{x}$. In nearly all practical scenarios, $m \ll d$ (fewer measurements than pixels) or $\mathcal{A}$ has a nontrivial null-space. The problem is severely ill-posed, meaning the pseudo-inverse solution is either nonunique or wildly unstable to noise. Therefore, a prior is usually needed to constrain the solution.
	
	Diffusion models can be adapted to solve such inverse problems by replacing the unconditional score in \eqref{eqn_reverse} with the posterior score $\nabla \log p_t(\bm{x}_t |\bm{y})$. For notational convenience, we denote $\pi_t(\bm{x}_t) :=p_t(\bm{x}_t|\bm{y})$ as the time-dependent conditional distribution. As shown in \eqref{eqn_pos_g}, the posterior score  $\nabla \log\pi_t(\bm{x}_t)$ decomposes into a time-dependent prior score $\nabla \log p_t(\bm{x}_t)$ that can be efficiently estimated via denoising score matching \eqref{eq:denoise} and a time-dependent likelihood term
	\begin{equation}
		\label{eqn_like}
		\nabla_{\bm{x}} \log p(\bm{y}|\bm{x}_t)=\int_{\mathbb{R}^d} p(\bm{y}|\bm{x}_0)p(\bm{x}_0|\bm{x}_t) d\bm{x}_0,
	\end{equation}
	whose evaluation constitutes the main technical challenge.
	
	Numerous studies have sought to approximate this intractable time-dependent likelihood term \eqref{eqn_like}. For Gaussian linear inverse problems, diffusion posterior sampling (DPS) \cite{chung2022diffusion} adopts the likelihood approximation $p(\bm{y}|\bm{x}_t)\approx\mathcal{N}(A\hat{\bm{x}}_0(\bm{x}_t), \sigma_{\xi}^2 I)$, where $\hat{\bm{x}}_0(\bm{x}_t)$ denotes a posterior mean estimator. However, this Dirac delta approximation of $p(\bm{x}_0|\bm{x}_t)$ fundamentally biases the sampling distribution, pulling it away from the true Bayesian posterior. To obtain a better approximation, pseudoinverse-guided diffusion models ($\Pi$GDM) \cite{song2023pseudoinverse} and Tweedie moment projected diffusion (TMPD) \cite{boys2023tweedie} leverage higher order information and obtain statistically principled approximations. This term acts as a force vector, gently pushing the unconditional SDE updates toward the measurement manifold based on the residual error.
	
	However, these DPS-family approaches require backpropagation through a pre-trained model, making it computationally demanding and memory intensive. Numerous more advanced techniques such as plug-and-play methods, have since emerged, which do not require backpropagation \cite{wang2022zero,zhu2023denoising,mardani2023variational,wu2024diffusion,dou2025hybrid,zhang2025improving}. A detailed discussion on recent works for solving inverse problems with diffusion models is provided in Appendix \ref{ap:limit}.
	
	\subsection{Proximal Unadjusted Langevin Algorithm}	
	
	Sampling algorithms play an increasingly important role in tackling Bayesian inverse problems \eqref{eq:composite}. Among them, proximal methods provide a principled way to handle nonsmooth posterior potentials. In this subsection, we review the proximal unadjusted Langevin algorithm (P-ULA) and its connection to Moreau–Yosida regularization.
	
	As the potential in \eqref{eq:composite} has a nonsmooth term $g$, we exploit smoothing techniques from optimization theory to tackle this challenge. For a proper, closed, convex function $g: \mathbb{R}^d \rightarrow \mathbb{R} \cup \{+\infty\}$, its proximal operator is defined as:
	\begin{equation}
		\label{eq:proximal}
		\text{Prox}^{\lambda}_{g}(\bm{x}) = \mathop{\mathop{\arg\min}}_{\bm{u}\in\mathbb{R}^d} \left\{ g(\bm{u}) + \frac{1}{2\lambda} \|\bm{u} - \bm{x}\|^2 \right\},
	\end{equation}
	where $\lambda > 0$ is a proximal parameter. The Moreau-Yosida regularization (also known as Moreau envelope) \cite{durmus2022proximal} of $g$ is defined as:
	\begin{equation}
		\label{eq:moreau}
		g^\lambda(\bm{x}) = \min_{\bm{u}\in\mathbb{R}^d} \left\{ g(\bm{u}) + \frac{1}{2\lambda} \|\bm{u} - \bm{x}\|^2 \right\}.
	\end{equation}
	It can be shown that the gradient of $g^\lambda$ is connected to the proximal operator via the relation 
	\begin{equation}
		\label{eqn_grad_lam}
		\nabla g^\lambda(\bm{x}) = \lambda^{-1}(\bm{x} - \text{Prox}^{\lambda}_{g}(\bm{x})).
	\end{equation}
	This relation enables the optimization of nonsmooth functions with gradient-based methods, provided access to their proximal operator. Further properties of the Moreau envelope are detailed in Appendix \ref{ap:moreau}.

	To handle a potentially nonsmooth term $g$, one can consider its Moreau envelope $g^\lambda$ and the surrogate distribution
	\begin{equation}
		\nonumber 
		p_0^{\lambda}(\bm{x}) \propto \sup_{\bm{u}\in\mathbb{R}^d} \left\{p_0(\bm{u}) \exp\left\{-\frac{\|\bm{u}-\bm{x}\|^2}{2\lambda}\right\} \right\},
	\end{equation}
	which is smooth for any $\lambda>0$ and converges to the original target $p_0$ as $\lambda \to 0$. A standard sampling method for such smooth surrogates is the P-ULA \cite{pereyra2016proximal}, which iterates as follows:
	\begin{equation}
		\label{eq:pula}
		\bm{x}_{k+1}=\bm{x}_k+\delta\,\nabla\log p_0^{\lambda}(\bm{x}_k)+\sqrt{2\delta}\,\bm{\xi}_k,
	\end{equation}
	where $\delta>0$ is the stepsize and $\bm{\xi}_k\sim\mathcal{N}(\bm{0},I)$.
	
	Although this approach provides a principled mechanism for sampling from the target distribution $\pi(\bm x)$, it requires evaluating $\nabla \log p_0^{\lambda}$, i.e., explicit access to the target density. In many practical applications, however, one only has access to samples from the target distribution of P-ULA. Notably, from the perspective of score-based modeling, \eqref{eq:pula} can be viewed as employing an alternative Moreau score function $\nabla\log p_0^{\lambda}(\bm{x}_k)$ to replace the original score $\nabla\log p_0(\bm{x}_k)$.
	
	\section{Proximal-Based Generative Modeling}
	\label{sec:method}
	
	In this section, we introduce a novel proximal-based generative modeling framework for inverse problems. Our core idea is to substitute the traditional score function in the reverse-time SDE with the Moreau score, thereby admitting an explicit, smooth, and asymptotically equivalent formulation.
	
	\subsection{Motivation: Moreau--Yosida--Gaussian Equivalence}
	To motivate our approach, we start with the VE-SDE forward process defined in \eqref{eqn_forward}, which yields
	\begin{equation}
		\label{eqn_xx0}
		\bm{x}_t = \bm{x}_0 + \sqrt{\lambda(t)}\,\bm{\xi}_t,\quad \bm{\xi}_t \sim \mathcal{N}(\bm{0}, I),
	\end{equation}
	where $\bm x_0\sim \pi$ and $\pi$ is defined in (\ref{eq:composite}). The corresponding marginal density is
	\begin{equation}
		\label{eq:gauss}
		\pi_t(\bm{x}_t) \propto \int_{\mathbb{R}^d} \pi(\bm{x}_0) \exp\left\{-\frac{\|\bm{x}_0-\bm{x}_t\|^2}{2\lambda(t)}\right\}\, d\bm{x}_0,
	\end{equation}
	which means a Gaussian convolution of the target density.
	
	In parallel, we define the Moreau approximation of $\pi$ as
	\begin{equation}
		\label{eq:moreau_approx}
		\pi_t^{\lambda}(\bm{x}_t) \propto \sup_{\bm{x}_0\in \mathbb{R}^d} \left\{\pi(\bm{x}_0) \exp\left\{-\frac{\|\bm{x}_0-\bm{x}_t\|^2}{2\lambda(t)}\right\}\right\}.
	\end{equation}
	The following lemma demonstrates that, for quadratic potentials, these two constructions are equivalent.
	\begin{lemma}
		\label{eq_lemma_c1}
		Let $\pi(\bm{x})\propto \exp\{-U(\bm{x})\}$, where $U$ is a positive semidefinite quadratic function. Then for any $\lambda(t)>0$, we have
		\begin{equation}
			\pi_t^{\lambda}(\bm{x}_t)\propto \pi_t(\bm{x}_t).
		\end{equation}
	\end{lemma}
	Lemma \ref{eq_lemma_c1} provides an optimization perspective on Gaussian noising: in the quadratic case, adding Gaussian noise (convolution) is equivalent, up to a normalization constant, to applying Moreau--Yosida regularization.
	
	Inspired by this equivalence, we replace the traditional Stein score $\nabla \log \pi_t(\bm{x}_t)$ in the reverse-time dynamics with the Moreau score $\nabla \log \pi_t^{\lambda}(\bm{x}_t)$. This yields the reverse SDE:
	\begin{equation}
		\label{eqn_reverse1}
		d\bm{x}_t^{\leftarrow} = \big\{-b^2(t)\nabla \log \pi_t^\lambda(\bm{x}_t^{\leftarrow})\big\}\,dt + b(t)\,d\bar{\bm{B}}_t .
	\end{equation}
	Using \eqref{eq:composite} and \eqref{eqn_grad_lam}, the Moreau score admits the closed form
	\begin{equation}
		\label{eq:score_prox1}
		\nabla \log \pi_t^\lambda(\bm{x}_t)=\frac{\text{Prox}_{U}^{\lambda(t)}(\bm{x}_t)-\bm{x}_t}{\lambda(t)},
	\end{equation}
	where the proximal operator is defined by
	\begin{equation}
		\label{eq:prox_def}
		\begin{aligned}
			&\text{Prox}_{U}^{\lambda(t)}(\bm{x}_t)
			=\mathop{\mathop{\arg\min}}_{\bm{x}_0\in\mathbb{R}^d}\Big\{U(\bm{x}_0)+\frac{\|\bm{x}_0-\bm{x}_t\|^2}{2\lambda(t)}\Big\}\\
			&=\mathop{\mathop{\arg\min}}_{\bm{x}_0\in\mathbb{R}^d}\Big\{\beta f_{\bm{y}}(\bm x_0)+g(\bm x_0)+\frac{\|\bm{x}_0-\bm{x}_t\|^2}{2\lambda(t)}\Big\}.
		\end{aligned}
	\end{equation}
	This constructs an alternative score function that replaces the time-dependent posterior score while still guaranteeing correct sampling from the posterior distribution. This provides the key intuition behind our approach and establishes a direct connection to score-based diffusion models.
	
	Under the inverse model \eqref{eqn_inverse}, we have $f_{\bm{y}}(\bm{x})=-\log p(\bm{y}|\bm{x})$ and $g(\bm{x})=-\log p(\bm{x})$. Then the objective in \eqref{eq:prox_def} can be reformulated as
	\begin{align}
		\nonumber
		\mathrm{Prox}_{U}^{\lambda(t)}(\bm{x}_t)
		&=
		\mathop{\mathop{\arg\max}}_{\bm{x}_0\in\mathbb{R}^d}\ \left\{\log p(\bm{y}|\bm{x}_0)+\log p(\bm{x}_0| \bm{x}_t)\right\}\\
		\label{eqn_prox_U1}&=\mathop{\arg\max}_{\bm{x}_0\in\mathbb{R}^d}\ \log p(\bm x_0| \bm x_t, \bm{y}),
	\end{align}
	where the last equality follows from the conditional independence of $\bm y$ and $\bm x_t$ given $\bm x_0$.
	Substituting \eqref{eqn_prox_U1} into \eqref{eq:score_prox1} reveals a key advantage: in contrast to the conditional score defined in \eqref{eqn_pos_g}, the proposed Moreau score avoids evaluating the time-dependent likelihood term $\log p(\bm{y}|\bm{x}_t)$.
	
	We now quantify the discrepancy between the traditional score $\nabla\log \pi_t(\bm{x}_t)$ and the Moreau score $\nabla\log \pi_t^\lambda(\bm{x}_t)$.
	\begin{proposition}
		\normalfont
		\label{pro:score_error}
		Let $\bm{x}_t$ be generated by \eqref{eqn_xx0}, and let $\bm x_0$ and $\bm y$ satisfy the inverse model \eqref{eqn_inverse}.
		
		(1) If the posterior $p(\bm{x}_0|\bm{x}_t,\bm y)$ is symmetric unimodal, then
		\begin{equation}
			\nonumber
			\nabla \log \pi_t(\bm{x}_t) = \nabla \log \pi_t^{\lambda}(\bm{x}_t).
		\end{equation}
		(2) If $f_{\bm{y}}$ is convex, then
		\begin{equation}
			\nonumber
			\big\|\nabla \log \pi_t(\bm{x}_t)-\nabla \log \pi_t^\lambda(\bm{x}_t)\big\|
			\leq \sqrt{\frac{2d}{\lambda(t)}}.
		\end{equation}
		In particular, if $f_{\bm{y}}$ is $m$-strongly convex, then
		\begin{equation}
			\nonumber
			\big\|\nabla \log \pi_t(\bm{x}_t)-\nabla \log \pi_t^\lambda(\bm{x}_t)\big\|
			\leq \sqrt{\frac{2d}{\beta m \lambda^2(t) + \lambda(t)}}.
		\end{equation}
		Further, the discrepancy vanishes as $\beta\to \infty$.
	\end{proposition}   
	Proposition \ref{pro:score_error} states that the Moreau score yields a computationally tractable and well controlled approximation to the traditional score. A generalized discussion of the Moreau-Yosida-Gaussian equivalence and the properties of the Moreau score is provided in Appendix \ref{ap:equivalent}.

	\subsection{Moreau Score Matching}
	
	To sample from the target distribution $\pi$ via the reverse SDE \eqref{eqn_reverse1}, the Moreau score \eqref{eq:score_prox1} is required. In this subsection, we introduce a Moreau score matching procedure to compute this quantity.
	
	Recall from \eqref{eq:composite} that $U(\bm x) = \beta f_{\bm y}(\bm x) + g(\bm x)$, where $f_{\bm y}(\bm x)$ is derived from the likelihood and is assumed known for any given measurement $\bm{y}$. A standard proximal splitting for $\text{Prox}_{U}^{\lambda(t)}(\bm{x}_t)$ gives \cite{condat2023proximal}:
	\begin{equation}
		\label{eqn_prox_U}
		\text{Prox}_{U}^{\lambda(t)}(\bm{x}_t)\approx \text{Prox}_{g}^{\lambda(t)}\!\left(\bm{x}_t-\beta \lambda(t)\nabla f_{\bm{y}}(\bm{x}_t)\right).
	\end{equation}
	Since $g$ is unknown, the proximal operator $\text{Prox}_{g}^{\lambda(t)}$ is unavailable, which forces us to learn this operator from data in an unsupervised mechanism. Notably, this only requires samples generated from the prior $p(\bm x_0)\propto \exp\left\{-g(\bm x_0)\right\}$, rather than the posterior $\pi$. A detailed characterization of the proximal splitting is provided in Appendix \ref{ap:score}.
	
	We now turn to the computation of the proximal operator appearing in the Moreau score \eqref{eq:score_prox1}. As shown in (\ref{eqn_xx0}), the transition kernel is
	\begin{equation}
		\label{eqn_trans_x0t}
		p(\bm x_t|\bm{x}_0)\propto \exp\left\{-\frac{\|\bm x_0-\bm x_t\|^2}{2\lambda(t)}\right\}.
	\end{equation}
	Substituting (\ref{eqn_trans_x0t}) and prior $p(\bm x_0)$ into the definition of the proximal operator in \eqref{eq:proximal} yields
	\begin{equation}
		\text{Prox}_{g}^{\lambda(t)}(\bm x_t)=\mathop{\mathop{\arg\max}}_{\bm x_0\in\mathbb{R}^d}\ p(\bm x_0|\bm x_t).
	\end{equation}
	This implies that the proximal operator is equivalent to the maximum a posteriori (MAP) estimator of $\bm{x}_0$ given Gaussian corrupted observation $\bm x_t$ \cite{fang2023s}.
	
	From the viewpoint of Bayesian decision theory, the MAP estimator is also obtained by minimizing the Bayes risk under sharply concentrated loss functions. Concretely, let $\bm{\phi}_{\bm{\theta}}(\bm{x}_t,\lambda(t))$, which is parameterized by $\bm{\theta}$, be an estimator of $\bm{x}_0$. For any $\zeta > 0$, we measure the discrepancy between $\bm x_0$ and its estimator using the loss function:
	\begin{equation}
		\nonumber
		\text{dist}_{\zeta}(\bm{x},\bm{x}^\prime) = 1-\mathcal{N}(\bm{x}-\bm{x}^\prime;\bm{0},\zeta^2I).
	\end{equation}
	We proceed under the following regularity conditions:
	\begin{assumption}[Properties of potential]
		\normalfont
		\label{assumption1}
		The function $f_{\bm{y}}$ is proper, closed, $L$-smooth, and $m$-strongly convex with $m \geq 0$. The function $g$ is convex with compact domain $\mathcal{X}$. 
	\end{assumption}
	We now present the optimality result, identifying the MAP estimator as the minimizer of a specific Bayes risk.
	\begin{proposition}
		\normalfont
		\label{pro:moreau_score}
		Under Assumption \ref{assumption1}, let $\bm x_t$ be generated by (\ref{eqn_xx0}). Consider the parameter $\bm{\theta}^*$ that minimizes the limiting Bayes risk:
		\begin{equation}
			\label{eq:moreau_score_match}
			\bm{\theta}^* = \mathop{\mathop{\arg\min}}_{\bm{\theta}} \lim_{\zeta \to 0} \mathbb{E}\left[\text{dist}_{\zeta}\left(\bm{\phi}_{\bm{\theta}}(\bm{x}_t,\lambda(t)),\bm{x}_0\right)\right],
		\end{equation}
		where the expectation is taken over $t\sim \mathcal{U}[0, T]$, $\bm x_0\sim \exp\left\{-g(\bm{x}_0)\right\}$, and $\bm x_t\sim p(\bm x_t|\bm x_0)$. Then, we have
		\begin{equation*}
			\bm{\phi}_{\bm{\theta}^*}(\bm{x}_t,\lambda(t)) =\mathop{\mathop{\arg\max}}_{\bm x_0\in\mathbb{R}^d}\ p(\bm x_0|\bm x_t)=\text{Prox}_{g}^{\lambda(t)}(\bm x_t).
		\end{equation*}
	\end{proposition}
	A detailed proof of Proposition \ref{pro:moreau_score} is provided in Appendix \ref{ap:prox_match}. Notably, 
	Proposition \ref{pro:moreau_score} indicates that the proximal operator $\text{Prox}_{g}^{\lambda(t)}(\bm x_t)$ can be approximated with $\bm{\phi}_{\bm{\theta}^*}(\bm{x}_t,\lambda(t))$ by solving the optimization problem (\ref{eq:moreau_score_match}). 
	
	Substituting this result back into the splitting \eqref{eqn_prox_U}, we obtain an explicit, tractable formulation for the Moreau score:
	\begin{equation}
		\nonumber
		\nabla \log \pi_t^\lambda(\bm{x}_t)\approx \frac{\bm{\phi}_{\bm{\theta}^*}\left(\bm{x}_t-\beta \lambda(t)\nabla f_{\bm{y}}(\bm{x}_t),\lambda(t)\right)-\bm{x}_t}{\lambda(t)}.
	\end{equation}
	This enables the execution of reverse dynamics \eqref{eqn_reverse1} to effectively generate samples from the target posterior $\pi(\bm x) \propto \exp\{-U(\bm{x})\}$. After training, the learned proximal operator $\bm{\phi}_{\bm{\theta}^*}(\bm{x}_t,\lambda(t))$ can be used as a pre-trained operator for any sampling processes or further fine-tuning, following the zero-shot posterior sampling paradigm like score-based diffusion models. Further, it can also serve as a plug-and-play prior, replacing classical mathematical denoisers in optimization and sampling literature.
	
	In Section \ref{sec:theory}, we will discuss the error between the traditional score and the Moreau score.

	\subsection{Generalized Proximal Diffusion Sampling}
	The above Moreau score \eqref{eq:moreau_approx} is based on the VE-SDE forward process (\ref{eqn_xx0}). To sample from \eqref{eq:composite}, we extend the Moreau score to a general forward process \eqref{eqn_forward}, and then provide the proximal diffusion sampling process. Consider
	\begin{equation}
		\bm{x}_t = \mu(t)\bm{x}_0+\sigma(t)\bm{\xi}_t,\quad \bm{\xi}_t \sim \mathcal{N}(\bm{0},I),
	\end{equation}
	where 
	\begin{equation}
		\nonumber
		\begin{aligned}
			&\mu(t)=\exp\left\{\int_{0}^{t}a(s)ds\right\},\\
			&\sigma^2(t)=\int_{0}^{t} b^2(s)\exp\left\{2\int_{s}^{t}a(r)dr\right\} ds.
		\end{aligned}
	\end{equation} 
	We consider the reverse-time model
	\begin{equation}
		\label{eqn_reverse2}
		d\bm{x}_t^{\leftarrow} = \left\{a(t)\bm{x}_t^{\leftarrow}-b^2(t)\nabla \log \pi_t^\lambda(\bm{x}_t^{\leftarrow})\right\}dt + b(t)d\bar{\bm{B}}_t ,
	\end{equation}
	where the generalized surrogate distribution is defined by
	\begin{equation}
		\nonumber
		\pi_t^{\lambda}(\bm{x}_t) \propto \sup_{\bm{x}_0\in \mathbb{R}^d} \left\{\pi(\bm{x}_0) \exp\left\{-\frac{\|\bm{x}_0-\bm{x}_t/\mu(t)\|^2}{2\lambda(t)}\right\}\right\},
	\end{equation}
	and the time-varying parameter 
	\begin{equation}
		\nonumber
		\lambda(t)=\sigma^2(t)/\mu^2(t) = \int_{0}^{t} b^2(s)/\mu^2(s) ds.
	\end{equation}
	Then the generalized Moreau score function is given by
	\begin{equation}
		\label{eq:score_prox}
		\nabla\log \pi_t^{\lambda}(\bm{x}_t)=\frac{\mu(t)\text{Prox}_{U}^{\lambda(t)}\left(\bm{x}_t/\mu(t)\right)-\bm{x}_t}{\sigma^2(t)}.
	\end{equation}
	Compared with the original Moreau score in \eqref{eq:score_prox1}, the generalized expression in \eqref{eq:score_prox} involves a rescaling of the proximal input by $1/\mu(t)$  and substitutes the denominator with $\sigma^2(t)$.

	\begin{figure*}[t]
		\centering
		\begin{equation}
			\tag{RP}
			\label{eq:reverse}
			d\bm{x}_t^{\leftarrow} = \left\{\left[a(t)+\frac{b^2(t)}{\sigma^2(t)}\right]\bm{x}_t^{\leftarrow}-\frac{b^2(t)\mu(t)}{\sigma^2(t)}\bm{\phi}_{\bm{\theta}^*}\left(\frac{\bm{x}_t^{\leftarrow}}{\mu(t)}-\beta \lambda(t)\nabla f_{\bm{y}}\left(\frac{\bm{x}_t^{\leftarrow}}{\mu(t)}\right),\lambda(t)\right)\right\}dt + b(t)d\bar{\bm{B}}_t.
		\end{equation}
		\vspace{-10pt}
	\end{figure*} 
	Following the splitting \eqref{eqn_prox_U} and Proposition \ref{pro:moreau_score}, we use a proximal network $\bm{\phi}_{\bm{\theta}}(\bm{x}_t,\lambda(t))$ to approximate the real proximal $\text{Prox}_{g}^{\lambda(t)}$. The network parameter $\bm{\theta}$ is learned by solving the optimization problem \eqref{eq:moreau_score_match} via stochastic gradient descent, using samples drawn from $t\sim \mathcal{U}[0,T],~\bm{x}_0\sim \exp\left\{-g(\bm x_0)\right\},~\text{and}~ \bm{x}_t|\bm{x}_0\sim \mathcal{N}(\bm{0},\lambda(t)I)$.

	Substituting (\ref{eq:score_prox}) into \eqref{eqn_reverse2}, the corresponding reverse process is given by (\ref{eq:reverse}) at the top of the next page. To numerically simulate this process, we employ an exponential interpolation discretization \cite{zhang2022fast}, which leads to the following proximal diffusion sampling iteration:
	\begin{equation}
		\label{eq:sample}
		\bar{\bm{x}}_{k+1} = \alpha_{1,k}\bar{\bm{x}}_{k}+ \alpha_{2,k} \bm{P}_k + \alpha_{3,k} \bm{\xi}_k,
	\end{equation}
	where $\bm{\xi}_k\sim\mathcal{N}(\bm{0},I)$, $\tau_k=T-t_k$, 
	\begin{equation}
		\label{eq:Pk}
		\bm{P}_k = \bm{\phi}_{\bm{\theta}^*}\left(\frac{\bar{\bm{x}}_{k}}{\mu(\tau_k)}-\beta \lambda(\tau_k)\nabla f\left(\frac{\bar{\bm{x}}_{k}}{\mu(\tau_k)}\right),\lambda(\tau_k)\right),
	\end{equation}
	and the coefficients
	\begin{align}
		\nonumber  \alpha_{1,k} &= \frac{\lambda(\tau_{k+1})\mu(\tau_{k+1})}{\lambda(\tau_{k})\mu(\tau_{k})}, 
		\alpha_{2,k} = \mu(\tau_{k+1})\left(1-\frac{\lambda(\tau_{k+1})}{\lambda(\tau_{k})}\right) \\
		\label{eq:coeff}	\alpha_{3,k} &= \mu(\tau_{k+1})\sqrt{\lambda(\tau_{k+1})}\sqrt{1-\frac{\lambda(\tau_{k+1})}{\lambda(\tau_{k})}}.
	\end{align}
	
	\begin{algorithm}[t]
		\caption{Proximal-Based Generative Modeling}
		\label{alg:proximal_sample}
		\begin{algorithmic}[1]
			\REQUIRE Realizations $\{\bm{x}^{(i)}\}_{i=1}^{N_s}\sim \exp\left\{-g\right\}$, objective $f_{\bm y}$, inverse temperature $\beta$, proximal schedule $\lambda(t)$, diffusion schedule $\mu(t)$, learning schedule $\{\zeta_{\text{Epoch}}\}$.
			\ENSURE Sample $\bm{x}\sim \exp\left\{-\left[\beta f_{\bm y}(\bm{x})+g(\bm{x})\right]\right\}$.
			\STATE \textbf{(Training Phase: Moreau Score Matching)}
			\STATE Initialize proximal network $\bm{\phi}_{\bm{\theta}}(\bm{x},\lambda)$.
			\FOR{Epoch $=1$ to MaxEpoch}
			\STATE Sample $\bm{x}_0 \sim \exp\left\{-g(\bm{x})\right\}, t \in \mathcal{U}[0,T]$.
			\STATE Generate $\bm{x}_t|\bm{x}_0\sim \mathcal{N}(\bm{0},\lambda(t)I)$. 
			\STATE Take stochastic gradient descent on \eqref{eq:moreau_score_match}.
			\ENDFOR
			\STATE \textbf{(Sampling Phase: Proximal Diffusion Sampling)}
			\STATE Initialize $\bar{\bm{x}}_{0}\sim p_T$.
			\FOR{$k = 0$ to $K$}
			\STATE Calculate $\bm{P}_k$ with \eqref{eq:Pk} and coefficients with \eqref{eq:coeff}.
			\STATE Update $\bar{\bm{x}}_{k+1}$ with \eqref{eq:sample}.
			\ENDFOR
			\STATE \textbf{Return} Sample $\bm{x}=\bar{\bm{x}}_{K+1}$.
		\end{algorithmic}
	\end{algorithm}
	
	A detailed analysis of the sampling process and discretization error is provided in Appendix \ref{ap:theorem}. Finally, the complete PGM framework is summarized in Algorithm \ref{alg:proximal_sample}.

	\section{Theoretical Results}
	\label{sec:theory}
	
	In this section, we present quantitative bound on the Wasserstein-1 distance between the target distribution and the generated distribution induced by the proposed PGM. We begin by stating our main assumptions.
	\begin{assumption}
		\normalfont 
		\label{assumption2}
		We assume the following conditions hold:
		\begin{itemize}[topsep=0pt, itemsep=0pt]
			\item [1.] (Bounded schedule) The diffusion schedule $\mu(t)\in \left[0,1\right]$ is nonincreasing and $ m_{\mu}\leq|\nabla \log \mu(t)|\leq M_{\mu}\leq 1/2$. Moreover, $\lambda(t)\in[0,\bar{\lambda}]$ is nondecreasing and $1\leq m_{\lambda}\leq|\nabla \log  \lambda(t)|\leq M_{\lambda}$.
			\item [2.] (Bounded stepsize) $ \delta:= \sup_{k} \left\{t_{k+1}-t_k\right\} < \infty$. 
		\end{itemize}
	\end{assumption}
	The first condition controls the growth of the coefficients in \eqref{eq:coeff} and is satisfied by practical SDEs (e.g., VE-SDE and VP-SDE). The second condition provides a uniform upper bound on the stepsize and the relation $T\leq K\delta$, which helps control the discretization error. These assumptions are standard for analyzing the convergence of diffusion models \cite{lee2022convergence,li2024accelerating,xu2026polynomial}.
	
	Next, we first bound the discrepancy between the approximated Moreau score and the true score function.
	\begin{proposition}
		\label{pro:score_estimate}
		Under Assumptions \ref{assumption1}, consider the diffusion process \eqref{eqn_forward}. Denote 
		\begin{equation}
			\bm{P}_t = \bm{\phi}_{\bm{\theta}^*}\left(\frac{\bm{x}_{t}}{\mu(t)}-\beta \lambda(t)\nabla f_{\bm{y}}(\frac{\bm{x}_{t}}{\mu(t)}), \lambda(t)\right).
		\end{equation}
		Then there exists $M\geq0$ such that
		\begin{equation}
			\left\|\nabla \log \pi_t(\bm{x}_t)-\frac{\mu(t)\bm{P}_t-\bm{x}_t}{\sigma^2(t)}\right\|\leq  \frac{M\left(1 + \|\bm{x}_t\|\right)}{\sigma^2(t)}.
		\end{equation}
	\end{proposition}
	
	Proposition \ref{pro:score_estimate} provides an upper bound on the score estimation error, which consists of two components: the proximal splitting error and the gap between the true score and the Moreau score. Both of them can be bounded with standard optimization techniques and the discrepancy from Proposition \ref{pro:score_error}, respectively. 
	
	Let $\text{Law}(\bar{\bm{x}}_k) = \pi_{\infty} R_k$ denote the distribution of $\bar{\bm{x}}_k$ in (\ref{eq:sample}), where $\bar{\bm{x}}_0 \sim \pi_{\infty} := \lim_{t\to\infty} \text{Law}(\bm{x}_t^{\rightarrow})$ and $R_k$ is the transition kernel associated with $p(\bar{\bm{x}}_{k}|\bar{\bm{x}}_{0})$. Define the diameter of the constraint set as $\text{diam}(\mathcal{X})=\sup\left\{\|\bm{x}-\bm{x}^{\prime}\|:\bm{x},\bm{x}^{\prime}\in\mathcal{X} \right\}$. We then present the following non-asymptotic convergence result.
	\begin{theorem}
		\normalfont 
		\label{thm:dis_conv}
		Under Assumption \ref{assumption1} and \ref{assumption2}, for sufficiently large $T$, there exist $D_1,D_2,D_3$ such that 
		\begin{equation}
			\nonumber
			\mathcal{W}_1\left(\text{Law}(\bar{\bm{x}}_K),\pi \right)\leq D_1 \exp\left[-m_{\mu} T\right] + D_2  M +D_3 \delta^{1/2},
		\end{equation}
		where 
		\begin{equation}
			\nonumber
			\begin{aligned}
				&D_1 = W_e(\sqrt{d}+\mathrm{diam}(\mathcal{X})), \\
				&D_2 = \left(1 + W_e\lambda^{-1}\left(\frac{\operatorname{diam}^2(\mathcal{X})M_{\lambda}}{m_{\lambda}-M_{\mu}}\right)\right)W_{\lambda}, \\
				&D_3 = \left(1 + W_e\lambda^{-1}\left(\frac{\operatorname{diam}^2(\mathcal{X})M_{\lambda}}{m_{\lambda}-M_{\mu}}\right)\right)\left(W_f + W_{\mu}\right),
			\end{aligned}
		\end{equation}
		with some bounded constants $W_e,\ W_{\lambda},\ W_{\mu},\ W_f$.
	\end{theorem}
	
	\begin{figure}[t]
		\centering
		\includegraphics[width=0.98\linewidth]{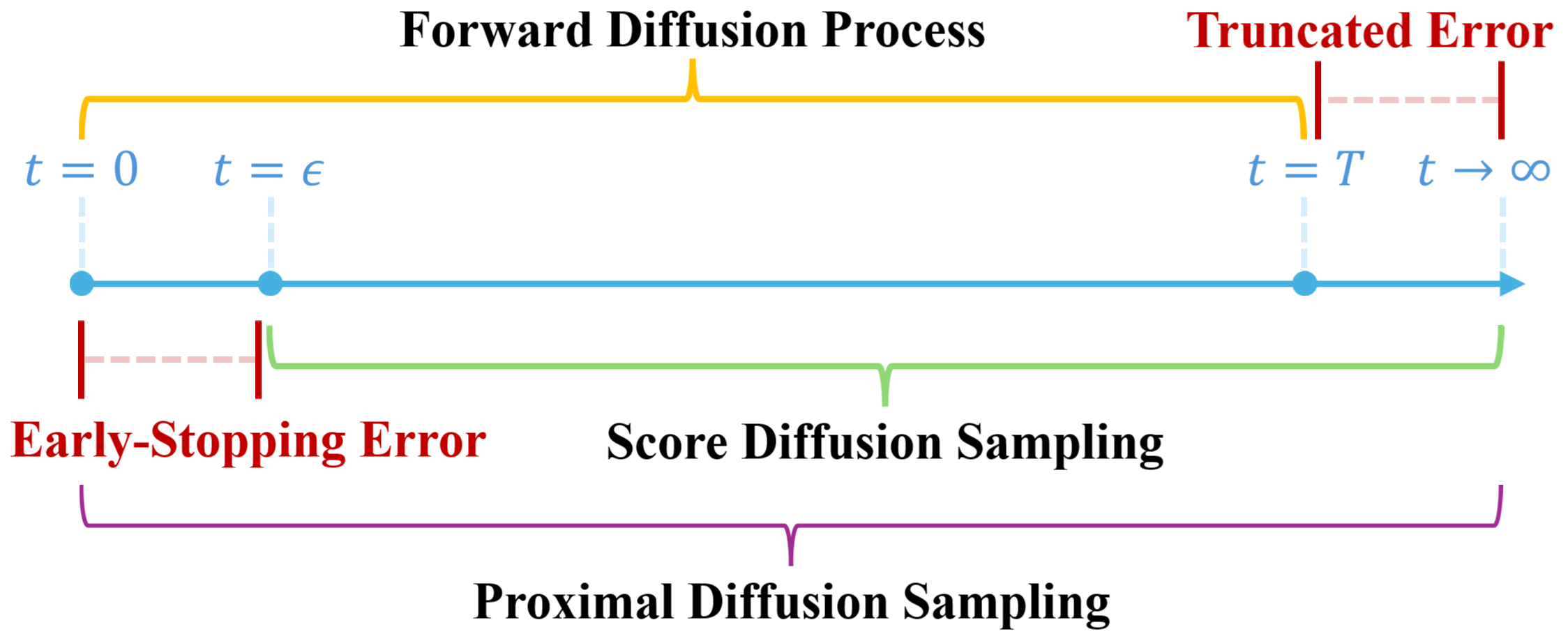}
		\caption{Sampling error decomposition.}
		\label{fig:error}
		\vspace{-5pt}
	\end{figure}
	
	It is worth noting that the error $\mathcal{W}_1\left(\text{Law}(\bar{\bm{x}}_K),\pi \right)$  vanishes as $T \to \infty$, $M \to 0$, and $\delta \to 0$. Consequently, Theorem \ref{thm:dis_conv} extends classical convergence results for diffusion models  \cite{de2022convergence,khalafi2024constrained,li2024towards} to the context of solving inverse problems. Notably, our convergence guarantees are stronger, as they eliminate the early-stopping error of order $\mathcal{O}(\sqrt{\epsilon})$, which corresponds to the deviation between the green coverage region $t\in[\epsilon,\infty)$ and the blue arrow region $t\in[0,\infty)$ in Figure \ref{fig:error}. It typically arises from partitioning the sampling interval into $[0,\epsilon]$ and $[\epsilon,T]$ to address nonsmoothness near $t=0$. As $t$ approaches zero, the true score $\nabla \log \pi_t(\bm{x}_t)$ diverges, creating a singularity that must be avoided. In our framework, this issue is naturally resolved by replacing the true score with the well-defined Moreau score $\nabla \log \pi_t^{\lambda}(\bm{x}_t)$.
	
	Theorem \ref{thm:dis_conv} provides the convergence of Algorithm \ref{alg:proximal_sample}, building a theoretical basis for the superior performance in some scenarios. However, the theoretical results are established mainly under convexity-type assumptions for analytical tractability, whereas the real-world problems involve more complex and generally nonconvex settings in practical scenarios. In Section \ref{sec:experiments}, we execute experiments to validate our theoretical results with synthetic data and show practical performance with real-world datasets when assumptions are violated. In future work, we will extend our PGM framework and theoretical results to more general nonconvex settings.
	
	\begin{figure*}[t]
		\centering
		\includegraphics[width=0.99\linewidth]{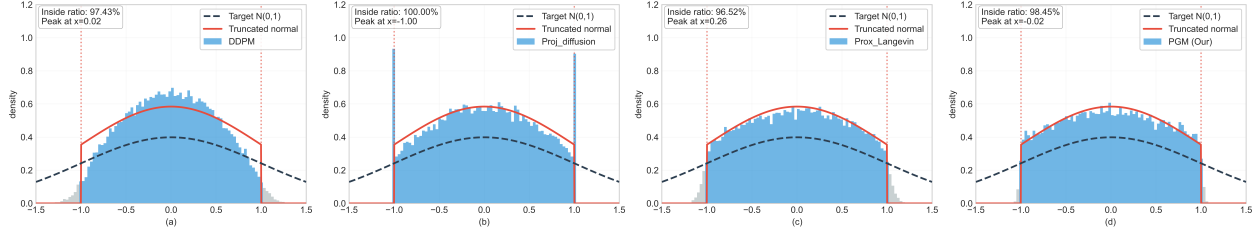}
		\caption{\textbf{Sampling from truncated normal distribution.} Score-based methods (a) DDPM and (b) Projected diffusion model fail to handle constraint. Proximal-based methods (c) proximal Langevin and (d) PGM (Our) perform better. PGM achieves better feasibility (inside-ratio$=98.45\%$) and optimality (peak at $x=-0.02$).}
		\label{fig:ve_constrained_sampling}
		\vspace{-10pt}
	\end{figure*}
	
	\section{Experiments}
	\label{sec:experiments}
	
	In this section, we validate the practical performance of PGM through two experiments. To highlight the significance of the theoretical results, we test a toy example on quadratic programming. Further, to validate practical performance of PGM with real-world datasets, we apply PGM to inverse problems across four image restoration tasks: super-resolution, random inpainting, Gaussian deblurring, and nonlinear deblurring. Models are trained on the FFHQ \cite{kazemi2014one}, CelebA-HQ \cite{liu2015deep}, and  LSUN-Bedroom \cite{yu2015lsun} datasets.
	
	Additional implementation details and extended results are provided in Appendix \ref{ap:experiment}. Our code is available at \href{https://github.com/boyangzhang2000/PGM}{https://github.com/boyangzhang2000/PGM}.

	\subsection{Toy Example: Quadratic Programming}
	
	Consider the constrained quadratic optimization problem:
	\begin{equation}
		\begin{aligned}
			\min_{\bm{x} \in \mathbb{R}^d} \quad & f(\bm{x}) = \frac{1}{2}\bm{x}^\top A \bm{x} + \bm{b}^\top \bm{x} \\
			\text{s.t.} \quad & \bm{x}\in\mathcal{X}\triangleq\left\{\bm{x}:\|x\|_2 \leq r\right\},
		\end{aligned}
	\end{equation}
	where ${A,\bm{b},r}$ are hyperparameters. This corresponds to sampling from $\pi(\bm{x}) \propto \exp\left\{-\beta f(\bm{x})\right\}\mathbb{I}_{\mathcal{X}}(\bm{x})$.
	
	To illustrate PGM’s superiority in handling constraints, we first visualize sampling from truncated Gaussian $\pi(\bm{x}) \propto\mathcal{N}(0,1)\mathbb{I}_{[-1,1]}$ where $A=I, \bm{b}=\bm{0}, r=1$. Under these settings, both traditional score and Moreau score \eqref{eq:score_prox} can be calculated explicitly. The methods for comparison include DDPM \cite{ho2020denoising} and projected diffusion model \cite{christopher2024constrained}, together with proximal-based methods such as proximal Langevin \cite{pereyra2016proximal} and PGM. Results are evaluated with feasibility, i.e., the proportion of samples satisfying the constraint (inside-ratio), and the optimality, i.e., the peak of the sampling distribution.
	
	As shown in Figure~\ref{fig:ve_constrained_sampling}, DDPM degrades in performance significantly due to the early-stopping error (see Fig.\ref{fig:error}). Projected diffusion model disrupts the structure of distribution and causes peaks at the boundary $\{-1,1\}$. Proximal Langevin has better convergent performance, yet suffers from long mixing time. In contrast, PGM recovers the target fast and achieves better feasibility and optimality.
	
	To validate the effect of Moreau score matching, we randomly choose ${A,\bm{b},r}$ and generate realizations uniformly from $\mathcal{X}$ for training. In sampling processes, we vary the number of iterations \(K\) and evaluate the feasibility rate. Based on the results in Table \ref{tab:constrained_quadratic1}, as \(K\) increases, the sample distribution rapidly converges towards the constraint set, and the feasibility rate ultimately approaches 1. This is consistent with the convergence result in Theorem \ref{thm:dis_conv}.
	
	\begin{table}[t]
		\captionof{table}{Feasibility evaluation with various \(K\).}
		\label{tab:constrained_quadratic1}
		\renewcommand{\arraystretch}{1.3}
		\resizebox{0.48\textwidth}{!}{
			\begin{tabular}{cccccc}
				\hline
				\multirow{2}{*}{Metric} & \multicolumn{5}{c}{10000 PGM samples, \(\beta=10\)}                \\ \cline{2-6} 
				& \(K=0\)      & \(K=1\)      & \(K=5\)      & \(K=10\)     & \(K=20\)     \\ \hline
				Feasibility                              & 0.00\%   & 7.43\%   & 97.43\%  & 99.61\%  & 99.73\%  \\ \hline
			\end{tabular}
		}
	\end{table}
	
	\begin{table}[t]
		\captionof{table}{Optimality and feasibility evaluation with various \(\beta\).}
		\label{tab:constrained_quadratic2}
		\renewcommand{\arraystretch}{1.3}
		\resizebox{0.48\textwidth}{!}{
			\begin{tabular}{cccccc}
				\hline
				& \multicolumn{5}{c}{10000 PGM samples, \(K=10\)}                                                        \\ \cline{2-6} 
				\multirow{-2}{*}{Metric} & \(\beta=0\)   & \(\beta=0.1\) & \(\beta=1\)   & \(\beta=2\)   & \(\beta=10\)                                 \\ \hline
				Optimality                          & 0.2365 & 0.2094 & 0.0701 & 0.0438 & 0.0164 \\
				Feasibility                         & 99.90\%  & 99.88\%  & 99.78\%  & 99.66\%  & 99.61\%                                  \\ \hline
			\end{tabular}
		}
		\vspace{-10pt}
	\end{table}
	
	Further, we test the influence of the inverse temperature \(\beta\) and calculate the optimality gap, i.e., the difference between the sample objective and the true optimum. Results in Table \ref{tab:constrained_quadratic2} indicate that as \(\beta\) increases, the target distribution $\pi(\bm{x})$ will concentrate around the optimal solution of the problem, which causes the samples to gradually approach the constraint boundary and converge to the optimizer.

	\subsection{Memory and Sampling Efficiency}

	\begin{table}[t]
		\captionof{table}{Memory and sampling usage of different methods.}
		\label{tab:cost}
		\renewcommand{\arraystretch}{1.3}
		\resizebox{0.48\textwidth}{!}{
			\begin{tabular}{cccccc}
				\hline
				Dataset                   & Method                                       & \# parameters  & Model size & Total memory & Sampling time \\ \hline
				& DPS                                          & 162,075,353 & 1953 MB    & 5369 MB      & 36 s           \\
				& DMPS                                         & 162,075,353 & 1953 MB    & 7168 MB      & 43 s           \\
				& PSLD                                         & 329,378,945 & 3969 MB    & 9485 MB      & 65 s           \\
				& ReSample                                     & 329,378,945 & 3969 MB    & 5009 MB      & 54 s           \\
				\multirow{-5}{*}{FFHQ} & PGM(Our)                                     & 221,673,987 & 2055 MB    & \textbf{4682} MB      & \textbf{4} s            \\ \hline
			\end{tabular}
		}
		\vspace{-5pt}
	\end{table}

	
	To demonstrate the efficiency of PGM, we monitored the memory usage during the solution of an inverse problem, and compared it with baseline methods. Table~\ref{tab:cost} reports the memory usage for the FFHQ dataset. Although the proposed PGM slightly increases the model size, it is significantly more efficient in terms of the memory usage because it avoids backpropagation through the posterior mean estimator. Moreover, PGM achieves a dramatic speedup in the sampling time, completing in 4s compared to 36s for DPS, representing a $9\times$ improvement in speed.
	

	\subsection{Quantitative Evaluation}

	We then present a comparative evaluation of human face reconstruction on the FFHQ and CelebA-HQ datasets. Results are compared with state-of-the-art diffusion-based methods including DPS \cite{chung2022diffusion}, manifold constrained gradients (MCG) \cite{chung2022improving}, plug-and-play using ADMM (ADMM-PnP) \cite{ahmad2020plug}, denoising diffusion destoration models (DDRM) \cite{kawar2022denoising}, diffusion model posterior sampling (DMPS) \cite{meng2022diffusion}, posterior sampling with latent diffusion (PSLD) \cite{rout2023solving}, and Resample \cite{song2023solving}. The pixel level error (e.g., PSNR) and structure error (including SSIM and LPIPS) are measured to show the comparison results. For fair comparison, all methods use similar network backbone and have similar memory usage.
	
	The quantitative results on the FFHQ dataset, as presented in Table~\ref{tab:ffhq1}, demonstrate the superior performance of PGM across all three linear image restoration tasks. This confirms its excellent generalizability and robustness to different types of image degradation problems. Further, PGM achieves a superior trade-off between numerical fidelity and perceptual similarity by simultaneously optimizing PSNR, SSIM, and LPIPS.

	\begin{table}[h]
		\captionof{table}{Quantitative evaluation of samples for FFHQ.}
		\label{tab:ffhq1}
		\renewcommand{\arraystretch}{1.4}
		\resizebox{0.48\textwidth}{!}{
			\begin{tabular}{cccccccccc}
				\hline
				FFHQ       & \multicolumn{3}{c}{Super resolution 4×} & \multicolumn{3}{c}{Inpainting (random   70\%)} & \multicolumn{3}{c}{Gaussian deblurring}          \\ \hline
				Method    & PSNR $\uparrow$     & SSIM $\uparrow$             & LPIPS $\downarrow$  & PSNR $\uparrow$    & SSIM $\uparrow$            & LPIPS $\downarrow$           & PSNR $\uparrow$          & SSIM $\uparrow$          & LPIPS $\downarrow$         \\ \hline
				DPS        & 28.47     & 0.793             & 0.175   & 32.32    & 0.897            & 0.106            & 27.70          & 0.774          & 0.169          \\
				MCG        & 23.74     & 0.673             & 0.223   & 24.89    & 0.731            & 0.178            & 25.33          & 0.668          & 0.371          \\
				ADMM-PnP   & 21.30     & 0.760             & 0.303   & 15.87    & 0.608            & 0.308            & 21.23          & 0.675          & 0.399          \\
				DDRM       & 27.51     & 0.753             & 0.257   & 24.97    & 0.680            & 0.287            & 26.51          & 0.702          & 0.299          \\
				DMPS       & 27.21     & 0.766             & 0.181   & 28.17    & 0.814            & 0.150            & 26.04          & 0.699          & 0.227          \\ \hline
				Latent-DPS & 24.65     & 0.609             & 0.344   & 27.08    & 0.727            & 0.270            & 25.98          & 0.704          & 0.258          \\
				PSLD-LDM   & 27.22     & 0.705             & 0.267   & 25.61    & 0.630            & 0.270            & 20.08          & 0.400          & 0.422          \\
				ReSample   & 28.90     & 0.804             & 0.164   & 31.34    & 0.890            & 0.099            & 28.73          & 0.794          & 0.201          \\ \hline
				PGM (Our)       & \textbf{29.82}    & \textbf{0.956}    & \textbf{0.158}   & \textbf{32.48}    & \textbf{0.962}   & \textbf{0.092}   & \textbf{30.67} & \textbf{0.959} & \textbf{0.148} \\ \hline
			\end{tabular}
		}
	\end{table}
	
	Results on the CelebA-HQ dataset in Table~\ref{tab:celeb1} further confirm the superiority of PGM. It achieves consistent state-of-the-art performance across all three restoration tasks, indicating that its advantages are not data-specific but stem from fundamental strengths in the proposed proximal-based generative modeling framework.

	\begin{table}[h]
		\captionof{table}{Quantitative evaluation of samples for CelebA-HQ.}
		\label{tab:celeb1}
		\renewcommand{\arraystretch}{1.4}
		\resizebox{0.48\textwidth}{!}{
			\begin{tabular}{cccccccccc}
				\hline
				CelebA-HQ  & \multicolumn{3}{c}{Super resolution 4×} & \multicolumn{3}{c}{Inpainting (random   70\%)} & \multicolumn{3}{c}{Gaussian deblurring}          \\ \hline
				Method    & PSNR $\uparrow$    & SSIM $\uparrow$             & LPIPS $\downarrow$   & PSNR $\uparrow$    & SSIM $\uparrow$            & LPIPS $\downarrow$           & PSNR $\uparrow$          & SSIM $\uparrow$          & LPIPS $\downarrow$         \\ \hline
				DPS        & 28.41    & 0.782             & 0.173    & 32.48    & 0.899            & 0.102            & 28.36          & 0.772          & 0.175          \\
				MCG        & 25.92    & 0.740             & 0.193    & 29.53    & 0.847            & 0.134            & 15.85          & 0.536          & 0.517          \\
				ADMM-PnP   & 21.08    & 0.631             & 0.304    & 15.40    & 0.342            & 0.627            & 20.98          & 0.602          & 0.289          \\
				DDRM       & 29.49    & 0.817             & 0.151    & 27.69    & 0.798            & 0.166            & 26.88          & 0.747          & 0.193          \\
				DMPS       & 28.48    & 0.811             & 0.147    & 28.84    & 0.826            & 0.175            & 26.45          & 0.726          & 0.206          \\ \hline
				Latent-DPS & 26.83    & 0.690             & 0.272    & 26.23    & 0.703            & 0.226            & 27.42          & 0.729          & 0.205          \\
				PSLD-LDM   & 27.61    & 0.704             & 0.209    & 27.07    & 0.689            & 0.260            & 24.21          & 0.548          & 0.323          \\
				ReSample   & 30.45    & 0.832             & 0.144    & 32.77    & 0.903            & 0.082            & 30.69          & 0.832          & 0.148          \\ \hline
				PGM (Our)       & \textbf{30.84}    & \textbf{0.951}    & \textbf{0.131}    & \textbf{33.41}    & \textbf{0.960}   & \textbf{0.076}   & \textbf{31.52} & \textbf{0.957} & \textbf{0.129} \\ \hline
			\end{tabular}
		}
	\end{table}
	
	Consistent with previous results, the SSIM scores of PGM on the FFHQ and CelebA-HQ datasets are exceptionally high, i.e., all of the scores are above 0.95, approaching the theoretical maximum of 1.0. This indicates that the restored images are virtually indistinguishable from the ground truth in terms of structural information—a crucial measure for high-quality image restoration.
	
	\subsection{Visualizing Confirmation}
	To evaluate the cross-prior generalization capability of PGM, we execute additional tests on the LSUN-Bedroom datasets. Qualitative results on LSUN-Bedroom are displayed in Figure \ref{fig:lsun1}, visually confirming the model's capability to generate high-fidelity and natural-looking images.

	\begin{figure}[t]
		\centering
		\includegraphics[width=0.98\linewidth]{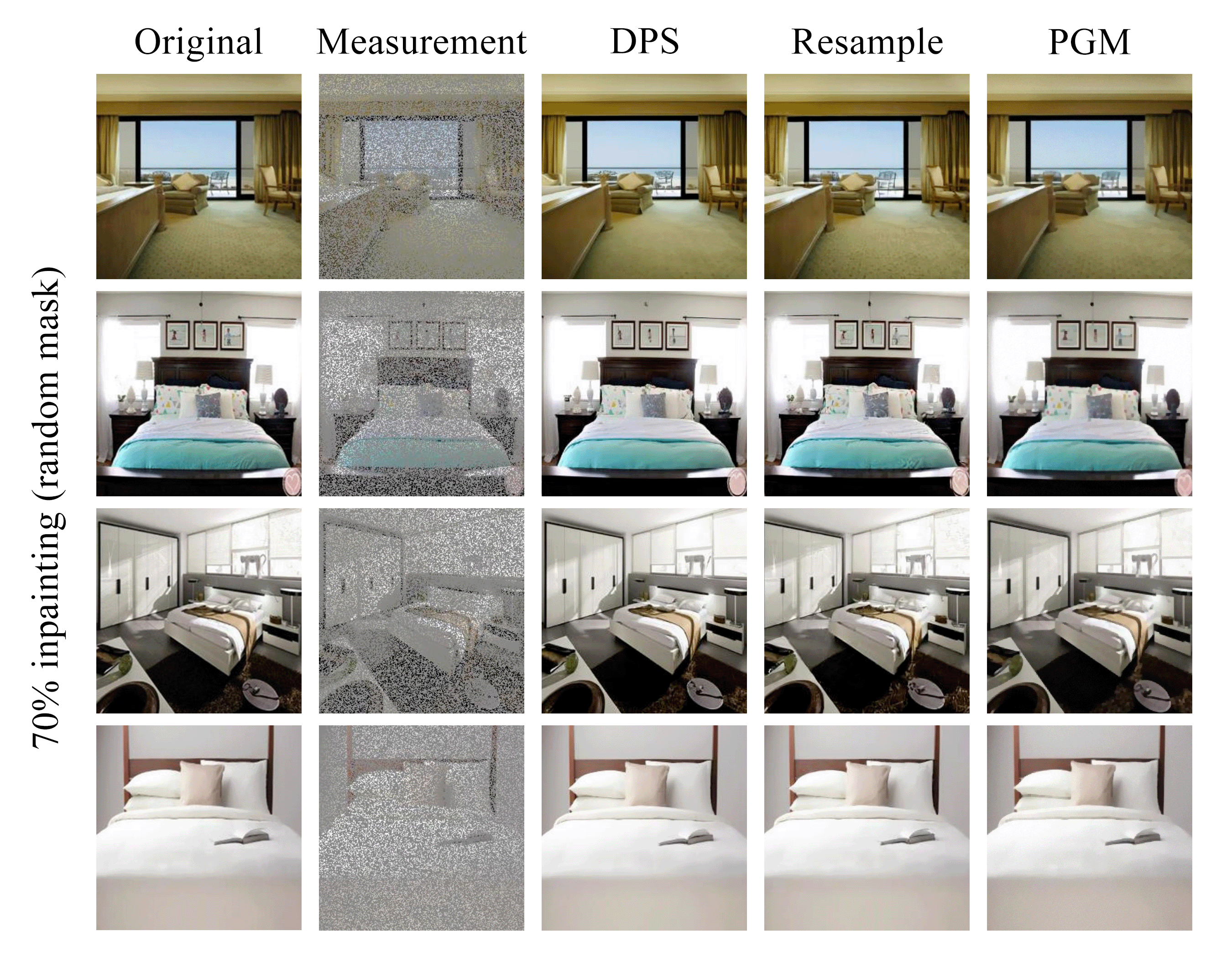}
		\caption{Visual samples for LSUN-Bedroom. }
		\label{fig:lsun1}
		\vspace{-10pt}
	\end{figure}

	\section{Conclusion}
	\label{sec:conclusion}
	
	In this paper, we have introduced PGM, a novel framework that bridges proximal operators and generative modeling to address fundamental challenges in Bayesian inverse problems. Built upon a theoretical equivalence between Gaussian convolution and Moreau-Yosida regularization, PGM enables a new pre-trained proximal diffusion sampling mechanism driven by the proposed Moreau score. We provide non-asymptotic convergence guarantees under convexity and demonstrate state-of-the-art empirical performance across multiple image restoration tasks. In summary, PGM not only advances the theoretical understanding of generative modeling but also offers a practical, efficient, and highly effective tool for tackling inverse problems in imaging and beyond.

	\clearpage
	
	\section*{Acknowledgements}
	
	The work of Boyang Zhang and Ya-Feng Liu was supported in part by the National Key R\&D Program of China under Grant 2024YFA1014203 and the National Natural Science Foundation of China (NSFC) under Grant 12371314. The work of Zhiguo Wang was supported in part by the NSFC under Grant 12571503 and the Open Foundation of the State Key Laboratory of Mathematical Sciences under Grant SLMS-2025-KFKT-TD-02.
	
	\section*{Impact Statement}
	
	This paper presents work whose goal is to advance the field of Machine
	Learning. There are many potential societal consequences of our work, none
	which we feel must be specifically highlighted here.
	
	%
	%

	\bibliography{camera-ready}

@inproceedings{welling2011bayesian,
	title={Bayesian Learning via Stochastic Gradient Langevin Dynamics},
	author={Welling, Max and Teh, Yee W},
	booktitle={Proceedings of the 28th International Conference on Machine Learning},
	pages={681--688},
	year={2011}
}

@article{song2020score,
	title={Score-Based Generative Modeling through Stochastic Differential Equations},
	author={Song, Yang and Sohl-Dickstein, Jascha and Kingma, Diederik P and Kumar, Abhishek and Ermon, Stefano and Poole, Ben},
	journal={arXiv preprint arXiv:2011.13456},
	year={2020}
}

@article{durmus2022proximal,
	title={A Proximal Markov Chain Monte Carlo Method for Bayesian Inference in Imaging Inverse Problems: When Langevin Meets Moreau},
	author={Durmus, Alain and Moulines, {\'E}ric and Pereyra, Marcelo},
	journal={SIAM Review},
	volume={64},
	number={4},
	pages={991--1028},
	year={2022},
	publisher={SIAM}
}

@article{xu2018global,
	title={Global Convergence of Langevin Dynamics Based Algorithms for Nonconvex Optimization},
	author={Xu, Pan and Chen, Jinghui and Zou, Difan and Gu, Quanquan},
	journal={Advances in Neural Information Processing Systems},
	volume={31},
	year={2018}
}

@article{durmus2017nonasymptotic,
	title={Nonasymptotic Convergence Analysis for the Unadjusted Langevin Algorithm},
	author={Durmus, Alain and Moulines, {\'E}ric},
	journal={The Annals of Applied Probability},
	volume={27},
	number={3},
	pages={1551--1587},
	year={2017}
}

@article{roberts1996exponential,
	title={Exponential Convergence of Langevin Distributions and Their Discrete Approximations},
	author={Roberts, Gareth O and Tweedie, Richard L},
	journal={Bernoulli},
	volume={2},
	number={4},
	pages={341--363},
	year={1996}
}

@article{pereyra2016proximal,
	title={Proximal Markov Chain Monte Carlo Algorithms},
	author={Pereyra, Marcelo},
	journal={Statistics and Computing},
	volume={26},
	number={4},
	pages={745--760},
	year={2016},
	publisher={Springer}
}

@article{hsieh2018mirrored,
	title={Mirrored Langevin Dynamics},
	author={Hsieh, Ya-Ping and Kavis, Ali and Rolland, Paul and Cevher, Volkan},
	journal={Advances in Neural Information Processing Systems},
	volume={31},
	year={2018}
}

@article{ho2020denoising,
	title={Denoising Diffusion Probabilistic Models},
	author={Ho, Jonathan and Jain, Ajay and Abbeel, Pieter},
	journal={Advances in Neural Information Processing Systems},
	volume={33},
	pages={6840--6851},
	year={2020}
}

@inproceedings{rombach2022high,
	title={High-Resolution Image Synthesis with Latent Diffusion Models},
	author={Rombach, Robin and Blattmann, Andreas and Lorenz, Dominik and Esser, Patrick and Ommer, Bj{\"o}rn},
	booktitle={Proceedings of the IEEE/CVF Conference on Computer Vision and Pattern Recognition},
	pages={10684--10695},
	year={2022}
}

@article{dhariwal2021diffusion,
	title={Diffusion Models Beat GANs on Image Synthesis},
	author={Dhariwal, Prafulla and Nichol, Alexander},
	journal={Advances in Neural Information Processing Systems},
	volume={34},
	pages={8780--8794},
	year={2021}
}

@article{zhang2026gradient,
	title={A Gradient Guided Diffusion Framework for Chance Constrained Programming},
	author={Zhang, Boyang and Wang, Zhiguo and Liu, Ya-Feng},
	journal={Advances in Neural Information Processing Systems},
	volume={38},
	pages={147692--147711},
	year={2026}
}

@article{christopher2024constrained,
	title={Constrained Synthesis with Projected Diffusion Models},
	author={Christopher, Jacob K and Baek, Stephen and Fioretto, Nando},
	journal={Advances in Neural Information Processing Systems},
	volume={37},
	pages={89307--89333},
	year={2024}
}

@inproceedings{chen2023restoration,
	title={Restoration-Degradation Beyond Linear Diffusions: A Non-Asymptotic Analysis for DDIM-Type Samplers},
	author={Chen, Sitan and Daras, Giannis and Dimakis, Alex},
	booktitle={International Conference on Machine Learning},
	pages={4462--4484},
	year={2023},
	organization={PMLR}
}

@article{li2023towards,
	title={Towards Faster Non-Asymptotic Convergence for Diffusion-Based Generative Models},
	author={Li, Gen and Wei, Yuting and Chen, Yuxin and Chi, Yuejie},
	journal={arXiv preprint arXiv:2306.09251},
	year={2023}
}

@article{lee2022convergence,
	title={Convergence for Score-Based Generative Modeling with Polynomial Complexity},
	author={Lee, Holden and Lu, Jianfeng and Tan, Yixin},
	journal={Advances in Neural Information Processing Systems},
	volume={35},
	pages={22870--22882},
	year={2022}
}

@article{chen2022sampling,
	title={Sampling Is as Easy as Learning the Score: Theory for Diffusion Models with Minimal Data Assumptions},
	author={Chen, Sitan and Chewi, Sinho and Li, Jerry and Li, Yuanzhi and Salim, Adil and Zhang, Anru R},
	journal={arXiv preprint arXiv:2209.11215},
	year={2022}
}

@article{kingma2013auto,
	title={Auto-Encoding Variational Bayes},
	author={Kingma, Diederik P and Welling, Max},
	journal={arXiv preprint arXiv:1312.6114},
	year={2013}
}

@article{razavi2019generating,
	title={Generating Diverse High-Fidelity Images with VQ-VAE-2},
	author={Razavi, Ali and Van den Oord, Aaron and Vinyals, Oriol},
	journal={Advances in Neural Information Processing Systems},
	volume={32},
	year={2019}
}

@inproceedings{daniel2021soft,
	title={Soft-IntroVAE: Analyzing and Improving the Introspective Variational Autoencoder},
	author={Daniel, Tal and Tamar, Aviv},
	booktitle={Proceedings of the IEEE/CVF Conference on Computer Vision and Pattern Recognition},
	pages={4391--4400},
	year={2021}
}

@article{solera2024beta,
	title={$\beta$-Variational Autoencoders and Transformers for Reduced-Order Modelling of Fluid Flows},
	author={Solera-Rico, Alberto and Sanmiguel Vila, Carlos and G{\'o}mez-L{\'o}pez, Miguel and Wang, Yuning and Almashjary, Abdulrahman and Dawson, Scott TM and Vinuesa, Ricardo},
	journal={Nature Communications},
	volume={15},
	number={1},
	pages={1361},
	year={2024},
	publisher={Nature Publishing Group UK London}
}

@article{fang2025deeponet,
	title={A DeepONet Joint Neural Tangent Kernel Hybrid Framework for Physics-Informed Inverse Source Problems and Robust Image Reconstruction},
	author={Fang, Yuhao and Wang, Zijian and Lu, Yao and Zhang, Ye and Li, Chun},
	journal={arXiv preprint arXiv:2511.00338},
	year={2025}
}

@book{rockafellar1998variational,
	title={Variational Analysis},
	author={Rockafellar, R Tyrrell and Wets, Roger JB},
	year={1998},
	publisher={Springer}
}

@article{efron2011tweedie,
	title={Tweedie’s Formula and Selection Bias},
	author={Efron, Bradley},
	journal={Journal of the American Statistical Association},
	volume={106},
	number={496},
	pages={1602--1614},
	year={2011},
	publisher={Taylor \& Francis}
}

@article{stein1981estimation,
	title={Estimation of the Mean of a Multivariate Normal Distribution},
	author={Stein, Charles M},
	journal={The Annals of Statistics},
	pages={1135--1151},
	year={1981},
	publisher={JSTOR}
}

@article{del2022backward,
	title={Backward It{\^o}--Ventzell and Stochastic Interpolation Formulae},
	author={Del Moral, Pierre and Singh, Sumeetpal Sidhu},
	journal={Stochastic Processes and their Applications},
	volume={154},
	pages={197--250},
	year={2022},
	publisher={Elsevier}
}

@article{ito1944109,
	title={109. Stochastic Integral},
	author={It{\^o}, Kiyosi},
	journal={Proceedings of the Imperial Academy},
	volume={20},
	number={8},
	pages={519--524},
	year={1944},
	publisher={The Japan Academy}
}

@inproceedings{fontaine2021convergence,
	title={Convergence Rates and Approximation Results for SGD and Its Continuous-Time Counterpart},
	author={Fontaine, Xavier and De Bortoli, Valentin and Durmus, Alain},
	booktitle={Conference on Learning Theory},
	pages={1965--2058},
	year={2021},
	organization={PMLR}
}

@article{mou2022efficient,
	title={An Efficient Sampling Algorithm for Non-Smooth Composite Potentials},
	author={Mou, Wenlong and Flammarion, Nicolas and Wainwright, Martin J and Bartlett, Peter L},
	journal={Journal of Machine Learning Research},
	volume={23},
	number={233},
	pages={1--50},
	year={2022}
}

@article{ma2019sampling,
	title={Sampling Can Be Faster Than Optimization},
	author={Ma, Yi-An and Chen, Yuansi and Jin, Chi and Flammarion, Nicolas and Jordan, Michael I},
	journal={Proceedings of the National Academy of Sciences},
	volume={116},
	number={42},
	pages={20881--20885},
	year={2019},
	publisher={National Academy of Sciences}
}

@article{salim2019stochastic,
	title={Stochastic Proximal Langevin Algorithm: Potential Splitting and Nonasymptotic Rates},
	author={Salim, Adil and Kovalev, Dmitry and Richt{\'a}rik, Peter},
	journal={Advances in Neural Information Processing Systems},
	volume={32},
	year={2019}
}

@article{pillai2024optimal,
	title={Optimal Scaling for the Proximal Langevin Algorithm in High Dimensions},
	author={Pillai, Natesh S},
	journal={Journal of Machine Learning Research},
	volume={25},
	number={404},
	pages={1--32},
	year={2024}
}

@article{ahn2021efficient,
	title={Efficient Constrained Sampling via the Mirror-Langevin Algorithm},
	author={Ahn, Kwangjun and Chewi, Sinho},
	journal={Advances in Neural Information Processing Systems},
	volume={34},
	pages={28405--28418},
	year={2021}
}

@article{song2023solving,
	title={Solving Inverse Problems with Latent Diffusion Models via Hard Data Consistency},
	author={Song, Bowen and Kwon, Soo Min and Zhang, Zecheng and Hu, Xinyu and Qu, Qing and Shen, Liyue},
	journal={arXiv preprint arXiv:2307.08123},
	year={2023}
}

@article{khalafi2024constrained,
	title={Constrained Diffusion Models via Dual Training},
	author={Khalafi, Shervin and Ding, Dongsheng and Ribeiro, Alejandro},
	journal={Advances in Neural Information Processing Systems},
	volume={37},
	pages={26543--26576},
	year={2024}
}

@article{chung2022improving,
	title={Improving Diffusion Models for Inverse Problems Using Manifold Constraints},
	author={Chung, Hyungjin and Sim, Byeongsu and Ryu, Dohoon and Ye, Jong Chul},
	journal={Advances in Neural Information Processing Systems},
	volume={35},
	pages={25683--25696},
	year={2022}
}

@article{li2024accelerating,
	title={Accelerating Convergence of Score-Based Diffusion Models, Provably},
	author={Li, Gen and Huang, Yu and Efimov, Timofey and Wei, Yuting and Chi, Yuejie and Chen, Yuxin},
	journal={arXiv preprint arXiv:2403.03852},
	year={2024}
}

@article{de2022convergence,
	title={Convergence of Denoising Diffusion Models under the Manifold Hypothesis},
	author={De Bortoli, Valentin},
	journal={arXiv preprint arXiv:2208.05314},
	year={2022}
}

@article{fang2023s,
	title={What's in a Prior? Learned Proximal Networks for Inverse Problems},
	author={Fang, Zhenghan and Buchanan, Sam and Sulam, Jeremias},
	journal={arXiv preprint arXiv:2310.14344},
	year={2023}
}

@inproceedings{li2024towards,
	title={Towards Non-Asymptotic Convergence for Diffusion-Based Generative Models},
	author={Li, Gen and Wei, Yuting and Chen, Yuxin and Chi, Yuejie},
	year={2024},
	booktitle={The Twelfth International Conference on Learning Representations}
}

@article{chung2022diffusion,
	title={Diffusion Posterior Sampling for General Noisy Inverse Problems},
	author={Chung, Hyungjin and Kim, Jeongsol and McCann, Michael T and Klasky, Marc L and Ye, Jong Chul},
	journal={arXiv preprint arXiv:2209.14687},
	year={2022}
}

@inproceedings{song2023pseudoinverse,
	title={Pseudoinverse-Guided Diffusion Models for Inverse Problems},
	author={Song, Jiaming and Vahdat, Arash and Mardani, Morteza and Kautz, Jan},
	booktitle={International Conference on Learning Representations},
	year={2023}
}

@article{li2025efficient,
	title={Efficient Diffusion Posterior Sampling for Noisy Inverse Problems},
	author={Li, Ji and Wang, Chao},
	journal={SIAM Journal on Imaging Sciences},
	volume={18},
	number={2},
	pages={1468--1492},
	year={2025},
	publisher={SIAM}
}

@article{rockafellar1976monotone,
	title={Monotone Operators and the Proximal Point Algorithm},
	author={Rockafellar, R Tyrrell},
	journal={SIAM Journal on Control and Optimization},
	volume={14},
	number={5},
	pages={877--898},
	year={1976},
	publisher={SIAM}
}

@article{condat2023proximal,
	title={Proximal Splitting Algorithms for Convex Optimization: A Tour of Recent Advances, with New Twists},
	author={Condat, Laurent and Kitahara, Daichi and Contreras, Andr{\'e}s and Hirabayashi, Akira},
	journal={SIAM Review},
	volume={65},
	number={2},
	pages={375--435},
	year={2023},
	publisher={SIAM}
}

@article{eckstein1992douglas,
	title={On the Douglas--Rachford Splitting Method and the Proximal Point Algorithm for Maximal Monotone Operators},
	author={Eckstein, Jonathan and Bertsekas, Dimitri P},
	journal={Mathematical Programming},
	volume={55},
	number={1},
	pages={293--318},
	year={1992},
	publisher={Springer}
}

@article{xu2026polynomial,
	title={Polynomial Convergence of Riemannian Diffusion Models},
	author={Xu, Xingyu and Zhang, Ziyi and Nakahira, Yorie and Qu, Guannan and Chi, Yuejie},
	journal={arXiv preprint arXiv:2601.02499},
	year={2026}
}

@article{tropp2010computational,
	title={Computational Methods for Sparse Solution of Linear Inverse Problems},
	author={Tropp, Joel A and Wright, Stephen J},
	journal={Proceedings of the IEEE},
	volume={98},
	number={6},
	pages={948--958},
	year={2010},
	publisher={IEEE}
}

@book{bertero2021introduction,
	title={Introduction to Inverse Problems in Imaging},
	author={Bertero, Mario and Boccacci, Patrizia and De Mol, Christine},
	year={2021},
	publisher={CRC Press}
}

@article{soltani2019deep,
	title={Deep Learning-Based Channel Estimation},
	author={Soltani, Mehran and Pourahmadi, Vahid and Mirzaei, Ali and Sheikhzadeh, Hamid},
	journal={IEEE Communications Letters},
	volume={23},
	number={4},
	pages={652--655},
	year={2019},
	publisher={IEEE}
}

@article{ho2022cascaded,
	title={Cascaded Diffusion Models for High Fidelity Image Generation},
	author={Ho, Jonathan and Saharia, Chitwan and Chan, William and Fleet, David J and Norouzi, Mohammad and Salimans, Tim},
	journal={Journal of Machine Learning Research},
	volume={23},
	number={47},
	pages={1--33},
	year={2022}
}

@article{stuart2010inverse,
	title={Inverse Problems: A Bayesian Perspective},
	author={Stuart, Andrew M},
	journal={Acta Numerica},
	volume={19},
	pages={451--559},
	year={2010},
	publisher={Cambridge University Press}
}

@incollection{ye2019optimization,
	title={Optimization Methods for Inverse Problems},
	author={Ye, Nan and Roosta-Khorasani, Farbod and Cui, Tiangang},
	booktitle={2017 MATRIX Annals},
	pages={121--140},
	year={2019},
	publisher={Springer}
}

@article{afonso2010augmented,
	title={An Augmented Lagrangian Approach to the Constrained Optimization Formulation of Imaging Inverse Problems},
	author={Afonso, Manya V and Bioucas-Dias, Jos{\'e} M and Figueiredo, M{\'a}rio AT},
	journal={IEEE Transactions on Image Processing},
	volume={20},
	number={3},
	pages={681--695},
	year={2010},
	publisher={IEEE}
}

@article{knapik2011bayesian,
	title={Bayesian Inverse Problems with Gaussian Priors},
	author={Knapik, Bartek T and Van Der Vaart, Aad W and van Zanten, J Harry},
	journal={The Annals of Statistics},
	pages={2626--2657},
	year={2011},
	publisher={JSTOR}
}

@article{goodfellow2014generative,
	title={Generative Adversarial Nets},
	author={Goodfellow, Ian J and Pouget-Abadie, Jean and Mirza, Mehdi and Xu, Bing and Warde-Farley, David and Ozair, Sherjil and Courville, Aaron and Bengio, Yoshua},
	journal={Advances in Neural Information Processing Systems},
	volume={27},
	year={2014}
}

@article{janati2025bridging,
	title={Bridging Diffusion Posterior Sampling and Monte Carlo Methods: A Survey},
	author={Janati, Yazid and Moulines, Eric and Olsson, Jimmy and Oliviero-Durmus, Alain},
	journal={Philosophical Transactions of the Royal Society A},
	volume={383},
	number={2299},
	pages={20240331},
	year={2025},
	publisher={The Royal Society}
}

@article{haber2000optimization,
	title={On Optimization Techniques for Solving Nonlinear Inverse Problems},
	author={Haber, Eldad and Ascher, Uri M and Oldenburg, Doug},
	journal={Inverse Problems},
	volume={16},
	number={5},
	pages={1263},
	year={2000},
	publisher={IOP Publishing}
}

@inproceedings{yu2023freedom,
	title={Freedom: Training-Free Energy-Guided Conditional Diffusion Model},
	author={Yu, Jiwen and Wang, Yinhuai and Zhao, Chen and Ghanem, Bernard and Zhang, Jian},
	booktitle={Proceedings of the IEEE/CVF International Conference on Computer Vision},
	pages={23174--23184},
	year={2023}
}

@article{chang2025provable,
	title={Provable Diffusion Posterior Sampling for Bayesian Inversion},
	author={Chang, Jinyuan and Duan, Chenguang and Jiao, Yuling and Li, Ruoxuan and Yang, Jerry Zhijian and Yuan, Cheng},
	journal={arXiv preprint arXiv:2512.08022},
	year={2025}
}

@article{hu2020deep,
	title={Deep Learning for Channel Estimation: Interpretation, Performance, and Comparison},
	author={Hu, Qiang and Gao, Feifei and Zhang, Hao and Jin, Shi and Li, Geoffrey Ye},
	journal={IEEE Transactions on Wireless Communications},
	volume={20},
	number={4},
	pages={2398--2412},
	year={2020},
	publisher={IEEE}
}

@article{vincent2011connection,
	title={A Connection between Score Matching and Denoising Autoencoders},
	author={Vincent, Pascal},
	journal={Neural Computation},
	volume={23},
	number={7},
	pages={1661--1674},
	year={2011},
	publisher={MIT Press}
}

@article{liu2024survey,
	title={A Survey of Recent Advances in Optimization Methods for Wireless Communications},
	author={Liu, Ya-Feng and Chang, Tsung-Hui and Hong, Mingyi and Wu, Zheyu and So, Anthony Man-Cho and Jorswieck, Eduard A and Yu, Wei},
	journal={IEEE Journal on Selected Areas in Communications},
	year={2024},
	publisher={IEEE}
}

@article{meng2022diffusion,
	title={Diffusion Model Based Posterior Sampling for Noisy Linear Inverse Problems},
	author={Meng, Xiangming and Kabashima, Yoshiyuki},
	journal={arXiv preprint arXiv:2211.12343},
	year={2022}
}

@article{kawar2022denoising,
	title={Denoising Diffusion Restoration Models},
	author={Kawar, Bahjat and Elad, Michael and Ermon, Stefano and Song, Jiaming},
	journal={Advances in Neural Information Processing Systems},
	volume={35},
	pages={23593--23606},
	year={2022}
}

@article{ahmad2020plug,
	title={Plug-and-Play Methods for Magnetic Resonance Imaging: Using Denoisers for Image Recovery},
	author={Ahmad, Rizwan and Bouman, Charles A and Buzzard, Gregery T and Chan, Stanley and Liu, Sizhuo and Reehorst, Edward T and Schniter, Philip},
	journal={IEEE Signal Processing Magazine},
	volume={37},
	number={1},
	pages={105--116},
	year={2020},
	publisher={IEEE}
}

@article{rout2023solving,
	title={Solving Linear Inverse Problems Provably via Posterior Sampling with Latent Diffusion Models},
	author={Rout, Litu and Raoof, Negin and Daras, Giannis and Caramanis, Constantine and Dimakis, Alex and Shakkottai, Sanjay},
	journal={Advances in Neural Information Processing Systems},
	volume={36},
	pages={49960--49990},
	year={2023}
}

@article{zhang2022fast,
	title={Fast Sampling of Diffusion Models with Exponential Integrator},
	author={Zhang, Qinsheng and Chen, Yongxin},
	journal={arXiv preprint arXiv:2204.13902},
	year={2022}
}

@article{wu2024cddm,
	title={CDDM: Channel Denoising Diffusion Models for Wireless Semantic Communications},
	author={Wu, Tong and Chen, Zhiyong and He, Dazhi and Qian, Liang and Xu, Yin and Tao, Meixia and Zhang, Wenjun},
	journal={IEEE Transactions on Wireless Communications},
	volume={23},
	number={9},
	pages={11168--11183},
	year={2024},
	publisher={IEEE}
}

@article{zhou2025generative,
	title={Generative Diffusion Models for High Dimensional Channel Estimation},
	author={Zhou, Xingyu and Liang, Le and Zhang, Jing and Jiang, Peiwen and Li, Yong and Jin, Shi},
	journal={IEEE Transactions on Wireless Communications},
	year={2025},
	publisher={IEEE}
}

@article{tang2025accurate,
	title={Accurate and Fast Channel Estimation for Fluid Antenna Systems with Diffusion Models},
	author={Tang, Erqiang and Guo, Wei and He, Hengtao and Song, Shenghui and Zhang, Jun and Letaief, Khaled B},
	journal={arXiv preprint arXiv:2505.04930},
	year={2025}
}

@article{croitoru2023diffusion,
	title={Diffusion Models in Vision: A Survey},
	author={Croitoru, Florinel-Alin and Hondru, Vlad and Ionescu, Radu Tudor and Shah, Mubarak},
	journal={IEEE Transactions on Pattern Analysis and Machine Intelligence},
	volume={45},
	number={9},
	pages={10850--10869},
	year={2023},
	publisher={IEEE}
}

@article{fan2025generative,
	title={Generative Diffusion Models for Wireless Networks: Fundamental, Architecture, and State-of-the-Art},
	author={Fan, Dayu and Meng, Rui and Xu, Xiaodong and Liu, Yiming and Nan, Guoshun and Feng, Chenyuan and Han, Shujun and Gao, Song and Xu, Bingxuan and Niyato, Dusit and others},
	journal={arXiv preprint arXiv:2507.16733},
	year={2025}
}

@article{zhang2025generalization,
	title={Generalization of Diffusion Models Arises with a Balanced Representation Space},
	author={Zhang, Zekai and Li, Xiao and Li, Xiang and Shi, Lianghe and Wu, Meng and Tao, Molei and Qu, Qing},
	journal={arXiv preprint arXiv:2512.20963},
	year={2025}
}

@article{geyer1992practical,
	title={Practical Markov Chain Monte Carlo},
	author={Geyer, Charles J},
	journal={Statistical Science},
	pages={473--483},
	year={1992},
	publisher={JSTOR}
}

@article{benton2023nearly,
	title={Nearly $ d $-Linear Convergence Bounds for Diffusion Models via Stochastic Localization},
	author={Benton, Joe and De Bortoli, Valentin and Doucet, Arnaud and Deligiannidis, George},
	journal={arXiv preprint arXiv:2308.03686},
	year={2023}
}

@article{gu2023optimal,
	title={Optimal Transport-Guided Conditional Score-Based Diffusion Model},
	author={Gu, Xiang and Yang, Liwei and Sun, Jian and Xu, Zongben},
	journal={Advances in Neural Information Processing Systems},
	volume={36},
	pages={36540--36552},
	year={2023}
}

@article{zhang2023physics,
	title={A physics-based and data-driven approach for localized statistical channel modeling},
	author={Zhang, Shutao and Ning, Xinzhi and Zheng, Xi and Shi, Qingjiang and Chang, Tsung-Hui and Luo, Zhi-Quan},
	journal={IEEE Transactions on Wireless Communications},
	volume={23},
	number={6},
	pages={5409--5424},
	year={2023},
	publisher={IEEE}
}

@article{ma2018sparse,
	title={Sparse Bayesian learning for the time-varying massive MIMO channels: Acquisition and tracking},
	author={Ma, Jianpeng and Zhang, Shun and Li, Hongyan and Gao, Feifei and Jin, Shi},
	journal={IEEE Transactions on Communications},
	volume={67},
	number={3},
	pages={1925--1938},
	year={2018},
	publisher={IEEE}
}

@article{jin2017deep,
	title={Deep convolutional neural network for inverse problems in imaging},
	author={Jin, Kyong Hwan and McCann, Michael T and Froustey, Emmanuel and Unser, Michael},
	journal={IEEE Transactions on Image Processing},
	volume={26},
	number={9},
	pages={4509--4522},
	year={2017},
	publisher={IEEE}
}

@article{daras2024survey,
	title={A survey on diffusion models for inverse problems},
	author={Daras, Giannis and Chung, Hyungjin and Lai, Chieh-Hsin and Mitsufuji, Yuki and Ye, Jong Chul and Milanfar, Peyman and Dimakis, Alexandros G and Delbracio, Mauricio},
	journal={arXiv preprint arXiv:2410.00083},
	year={2024}
}

@article{yu2022role,
	title={Role of deep learning in wireless communications},
	author={Yu, Wei and Sohrabi, Foad and Jiang, Tao},
	journal={IEEE BITS the Information Theory Magazine},
	volume={2},
	number={2},
	pages={56--72},
	year={2022},
	publisher={IEEE}
}

@article{renaud2025stability,
	title={From stability of Langevin diffusion to convergence of proximal MCMC for non-log-concave sampling},
	author={Renaud, Marien and De Bortoli, Valentin and Leclaire, Arthur and Papadakis, Nicolas},
	journal={arXiv preprint arXiv:2505.14177},
	year={2025}
}

@article{ehrhardt2024proximal,
	title={Proximal Langevin sampling with inexact proximal mapping},
	author={Ehrhardt, Matthias J and Kuger, Lorenz and Sch{\"o}nlieb, Carola-Bibiane},
	journal={SIAM Journal on Imaging Sciences},
	volume={17},
	number={3},
	pages={1729--1760},
	year={2024},
	publisher={SIAM}
}

@article{habring2026diffusion,
	title={Diffusion at absolute zero: Langevin sampling using successive moreau envelopes},
	author={Habring, Andreas and Falk, Alexander and Zach, Martin and Pock, Thomas},
	journal={SIAM Journal on Imaging Sciences},
	volume={19},
	number={1},
	pages={35--77},
	year={2026},
	publisher={SIAM}
}

@article{alkhouri2024sitcom,
	title={Sitcom: Step-wise triple-consistent diffusion sampling for inverse problems},
	author={Alkhouri, Ismail and Liang, Shijun and Huang, Cheng-Han and Dai, Jimmy and Qu, Qing and Ravishankar, Saiprasad and Wang, Rongrong},
	journal={arXiv preprint arXiv:2410.04479},
	year={2024}
}

@article{boys2023tweedie,
	title={Tweedie moment projected diffusions for inverse problems},
	author={Boys, Benjamin and Girolami, Mark and Pidstrigach, Jakiw and Reich, Sebastian and Mosca, Alan and Akyildiz, O Deniz},
	journal={arXiv preprint arXiv:2310.06721},
	year={2023}
}

@article{mardani2023variational,
	title={A variational perspective on solving inverse problems with diffusion models},
	author={Mardani, Morteza and Song, Jiaming and Kautz, Jan and Vahdat, Arash},
	journal={arXiv preprint arXiv:2305.04391},
	year={2023}
}

@article{tolooshams2025equireg,
	title={EquiReg: Equivariance Regularized Diffusion for Inverse Problems},
	author={Tolooshams, Bahareh and Chandrashekar, Aditi and Zirvi, Rayhan and Mammadov, Abbas and Yao, Jiachen and Wang, Chuwei and Anandkumar, Anima},
	journal={arXiv preprint arXiv:2505.22973},
	year={2025}
}

@inproceedings{wu2024diffusion,
	title={Diffusion posterior proximal sampling for image restoration},
	author={Wu, Hongjie and He, Linchao and Zhang, Mingqin and Chen, Dongdong and Luo, Kunming and Luo, Mengting and Zhou, Ji-Zhe and Chen, Hu and Lv, Jiancheng},
	booktitle={Proceedings of the 32nd ACM International Conference on Multimedia},
	pages={214--223},
	year={2024}
}

@article{xu2025rethinking,
	title={Rethinking diffusion posterior sampling: From conditional score estimator to maximizing a posterior},
	author={Xu, Tongda and Cai, Xiyan and Zhang, Xinjie and Ge, Xingtong and He, Dailan and Sun, Ming and Liu, Jingjing and Zhang, Ya-Qin and Li, Jian and Wang, Yan},
	journal={arXiv preprint arXiv:2501.18913},
	year={2025}
}

@inproceedings{zhu2023denoising,
	title={Denoising diffusion models for plug-and-play image restoration},
	author={Zhu, Yuanzhi and Zhang, Kai and Liang, Jingyun and Cao, Jiezhang and Wen, Bihan and Timofte, Radu and Van Gool, Luc},
	booktitle={Proceedings of the IEEE/CVF Conference on Computer Vision and Pattern Recognition},
	pages={1219--1229},
	year={2023}
}

@article{wang2022zero,
	title={Zero-shot image restoration using denoising diffusion null-space model},
	author={Wang, Yinhuai and Yu, Jiwen and Zhang, Jian},
	journal={arXiv preprint arXiv:2212.00490},
	year={2022}
}

@inproceedings{dou2025hybrid,
	title={Hybrid regularization improves diffusion-based inverse problem solving},
	author={Dou, Hongkun and Li, Zeyu and Du, Jinyang and Yang, Lijun and Yao, Wen and Deng, Yue},
	booktitle={The Thirteenth International Conference on Learning Representations},
	year={2025}
}

@inproceedings{zhang2025improving,
	title={Improving diffusion inverse problem solving with decoupled noise annealing},
	author={Zhang, Bingliang and Chu, Wenda and Berner, Julius and Meng, Chenlin and Anandkumar, Anima and Song, Yang},
	booktitle={Proceedings of the Computer Vision and Pattern Recognition Conference},
	pages={20895--20905},
	year={2025}
}

@article{saharia2022image,
	title={Image super-resolution via iterative refinement},
	author={Saharia, Chitwan and Ho, Jonathan and Chan, William and Salimans, Tim and Fleet, David J and Norouzi, Mohammad},
	journal={IEEE Transactions on Pattern Analysis and Machine Intelligence},
	volume={45},
	number={4},
	pages={4713--4726},
	year={2022},
	publisher={IEEE}
}

@inproceedings{saharia2022palette,
	title={Palette: Image-to-image diffusion models},
	author={Saharia, Chitwan and Chan, William and Chang, Huiwen and Lee, Chris and Ho, Jonathan and Salimans, Tim and Fleet, David and Norouzi, Mohammad},
	booktitle={ACM SIGGRAPH 2022 Conference Proceedings},
	pages={1--10},
	year={2022}
}

@inproceedings{lugmayr2022repaint,
	title={Repaint: Inpainting using denoising diffusion probabilistic models},
	author={Lugmayr, Andreas and Danelljan, Martin and Romero, Andres and Yu, Fisher and Timofte, Radu and Van Gool, Luc},
	booktitle={Proceedings of the IEEE/CVF Conference on Computer Vision and Pattern Recognition},
	pages={11461--11471},
	year={2022}
}

@inproceedings{zirvi2025diffusion,
	title={Diffusion state-guided projected gradient for inverse problems},
	author={Zirvi, Rayhan and Tolooshams, Bahareh and others},
	booktitle={International Conference on Learning Representations},
	volume={2025},
	pages={31217--31242},
	year={2025}
}

@inproceedings{tran2021explore,
	title={Explore image deblurring via encoded blur kernel space},
	author={Tran, Phong and Tran, Anh Tuan and Phung, Quynh and Hoai, Minh},
	booktitle={Proceedings of the IEEE/CVF conference on computer vision and pattern recognition},
	pages={11956--11965},
	year={2021}
}

@article{lecun2002gradient,
	title={Gradient-based learning applied to document recognition},
	author={LeCun, Yann and Bottou, L{\'e}on and Bengio, Yoshua and Haffner, Patrick},
	journal={Proceedings of the IEEE},
	volume={86},
	number={11},
	pages={2278--2324},
	year={2002},
	publisher={Ieee}
}

@inproceedings{kazemi2014one,
	title={One millisecond face alignment with an ensemble of regression trees},
	author={Kazemi, Vahid and Sullivan, Josephine},
	booktitle={Proceedings of the IEEE Conference on Computer Vision and Pattern Recognition},
	pages={1867--1874},
	year={2014}
}

@inproceedings{liu2015deep,
	title={Deep learning face attributes in the wild},
	author={Liu, Ziwei and Luo, Ping and Wang, Xiaogang and Tang, Xiaoou},
	booktitle={Proceedings of the IEEE International Conference on Computer Vision},
	pages={3730--3738},
	year={2015}
}

@article{yu2015lsun,
	title={Lsun: Construction of a large-scale image dataset using deep learning with humans in the loop},
	author={Yu, Fisher and Seff, Ari and Zhang, Yinda and Song, Shuran and Funkhouser, Thomas and Xiao, Jianxiong},
	journal={arXiv preprint arXiv:1506.03365},
	year={2015}
}

@inproceedings{zhao2020maintaining,
	title={Maintaining discrimination and fairness in class incremental learning},
	author={Zhao, Bowen and Xiao, Xi and Gan, Guojun and Zhang, Bin and Xia, Shu-Tao},
	booktitle={Proceedings of the IEEE/CVF Conference on Computer Vision and Pattern Recognition},
	pages={13208--13217},
	year={2020}
}

@inproceedings{yuan2026improving,
	title={Improving Diffusion Inverse Problem Solving with Structure Consistency Regularization},
	author={Yuan, Zijie and Li, Ji and Liu, Zhaoqiang},
	booktitle={ICASSP 2026-2026 IEEE International Conference on Acoustics, Speech and Signal Processing (ICASSP)},
	pages={11497--11501},
	year={2026},
	organization={IEEE}
}

@article{fang2026beyond,
	title={Beyond scores: Proximal diffusion models},
	author={Fang, Zhenghan and D{\'\i}az, Mateo and Buchanan, Sam and Sulam, Jeremias},
	journal={Advances in Neural Information Processing Systems},
	volume={38},
	pages={94483--94532},
	year={2026}
}
	\bibliographystyle{icml2026}

	\newpage
	\appendix
	\onecolumn
	
	\section*{Outline of Appendices}
	
	The appendices are organized as follows:
	\begin{itemize}[leftmargin=30pt]
		\item [(i)] \textbf{Appendix} \ref{ap:limit} provides complement related works and detailed discussions on limitations and future research directions of this work.
		
		\item [(ii)] \textbf{Appendix} \ref{ap:experiment} presents more details of experiments, including detailed experimental setups and additional empirical results of inverse problems.
		
		\item [(iii)] \textbf{Appendix} \ref{ap:moreau} provides necessary preliminaries on Moreau-Yosida regularization, and the application in the Moreau approximation for nonsmooth log-concave densities.
		
		\item [(iv)] \textbf{Appendix} \ref{ap:equivalent} gives a characterization of Moreau-Yosida-Gaussian equivalence and a detailed proof of Lemma \ref{eq_lemma_c1}. Further, a generalized Moreau score and the proof of asymptotic equivalence in Proposition \ref{pro:score_error} are provided.
		
		\item [(v)] \textbf{Appendix} \ref{ap:prox_match} derives the formulation of Moreau score matching in Proposition \ref{pro:moreau_score}, and then provides a detailed characterization for the score estimation error in Proposition \ref{pro:score_estimate}.
		
		\item [(vi)] \textbf{Appendix} \ref{ap:discret} discusses the discretization schedule for the proximal diffusion sampling process, including Euler-Maruyama and exponential interpolation.
		
		\item [(vii)] \textbf{Appendix} \ref{ap:theorem} shows some dissipativity conditions on the forward process, reverse process, and tangent process, and then provides the proof of Theorem \ref{thm:dis_conv}.
	\end{itemize}
	
	\section{Limitations and Future Works}
	\label{ap:limit}
	
	\subsection{Complement Related Works}
	
	\textbf{Diffusion models for inverse problems.}
	Recent years have seen growing interest in using diffusion models as powerful structural priors for solving inverse problems, motivated by their success in high fidelity image generation \cite{dhariwal2021diffusion,ho2022cascaded,rombach2022high,li2025efficient,zhang2025generalization} and time-varying channel modeling \cite{croitoru2023diffusion,wu2024cddm,zhou2025generative,tang2025accurate,fan2025generative}. 
	
	To solve an inverse problem, the goal is to sample from the posterior distribution $p(\bm{x}_0 | \bm{y})$.  According to Bayes' theorem:
	\[
	p(\bm{x}_0 | \bm{y}) \propto p(\bm{y} |\bm{x}_0 ) p(\bm{x}_0).
	\]
	To implement this within the diffusion framework, we must modify the reverse process to sample from the conditional path $p(\bm{x}_{0:T} | \bm{y})$. For the reverse SDE, this means replacing the unconditional score $\nabla_{\bm{x}} \log p_t(\bm{x}_t)$ with the conditional score $\nabla_{\bm{x}} \log p_t(\bm{x}_t | \bm{y})$. 
	
	In some early works, a conditional neural network architecture is explicitly trained to map the specific measurement $\bm{y}$ directly to the clean image $\bm{x}_0$. They do this by learning the conditional score $\nabla_{\bm{x}} \log p_t(\bm{x}_t | \bm{y})$ directly from massive datasets of paired training data $(\bm{x}_0^{(i)},\bm{y}^{(i)})$. SR3 \cite{saharia2022image} and Palette \cite{saharia2022palette} define the foundational approach for supervised conditional diffusion. These models treat the inverse problem strictly as a supervised learning task. Although highly effective for fixed tasks, these approaches require large task-specific paired datasets and retraining for each new forward operator, limiting their flexibility and generalizability.
	
	In recent works, people are seeking methods that require absolutely no retraining or fine-tuning of the diffusion model parameters, called zero-shot posterior sampling. Instead, they mathematically modify the reverse sampling trajectory (either through hard geometric projections or gradient-based score approximations) to force the generated sample to agree with the measurement $\bm{y}$.
	
	For linear inverse problems, subspace and null-space projection is applied to modify the sampling trajectory by forcing the intermediate sample to directly satisfy the measurement consistency. DDRM \cite{kawar2022denoising} formulated one of the first mathematically exact, non-heuristic solvers specifically targeting arbitrary linear inverse problems. It relies heavily on the singular value decomposition (SVD) of the measurement matrix $A=U\Sigma V^{\top}$. DDNM \cite{wang2022zero} utilizes the Moore-Penrose pseudoinverse $A^{\dagger}$ to decompose the data strictly into a range-space and a null-space $\bm{x}=A^{\dagger}A\bm{x} + (I-A^{\dagger}A)\bm{x}$. It enforces absolute data consistency by overwriting the range-space with the exact measurements, while retaining the generative prior in the null space. DiffStateGrad \cite{zirvi2025diffusion} projected the measurement gradient onto a subspace that is a low-rank approximation of an intermediate state of the diffusion process. As a module, it can be added to a wide range of diffusion-based inverse solvers to improve the preservation of the diffusion process on the prior manifold and filter out artifact-inducing components.
	
	For nonlinear inverse problems, a series of approximations on the conditional score $\nabla_{\bm{x}} \log p_t(\bm{x}_t | \bm{y})$ is applied to obtain an explicit sampling process \cite{chung2022diffusion,song2023pseudoinverse,boys2023tweedie,wu2024diffusion,xu2025rethinking}. Applying Bayes' rule at an intermediate timestep $t$, we can decompose this conditional score as \eqref{eqn_pos_g}. The unconditional score is provided off-the-shelf by the pre-trained diffusion model $\bm{s}_{\bm{\theta}}(\bm{x}_t, t)$. Much of diffusion-based zero-shot nonlinear inverse problem solvers is fundamentally dedicated to computing or approximating the second term $\nabla_{\bm{x}} \log p(\bm{y}|\bm{x}_t)$. Note that the physical measurement $\bm{y}$ depends on the clean, uncorrupted signal $\bm{x}_0$ via $\bm{y} = \mathcal{A}(\bm{x}_0) + \bm{\xi}$, not on the intermediate noisy state $\bm{x}_t$. To compute the likelihood exactly, one must marginalize over all possible clean images
	\[
	\nabla_{\bm{x}} \log p(\bm{y}|\bm{x}_t)=\int_{\mathbb{R}^d} p(\bm{y}|\bm{x}_0)p(\bm{x}_0|\bm{x}_t) d\bm{x}_0.
	\]
	Here, $p(\bm{y}|\bm{x}_0)$ is known. However, $p(\bm{x}_0|\bm{x}_t)$ is the true reverse transition distribution mapping a noisy image back to the clean manifold. This distribution is highly complex, multimodal, and analytically intractable. Thus, the integral cannot be computed directly.
	
	A prominent class of approaches is posterior diffusion sampling, which focuses on estimating the time-dependent likelihood score $\nabla_{\bm{x}} \log p(\bm{y}|\bm{x}_t)$. Diffusion posterior sampling (DPS) \cite{chung2022diffusion} addresses this challenge by combining diffusion sampling with manifold constrained gradients \cite{chung2022improving}. This approach relies on a Laplace approximation of the conditional likelihood:
	\[
	p_t(\bm{y}|\bm{x}_t) \sim \mathcal{N}(A\hat{\bm{x}}_0(\bm{x}_t),\sigma^2_{\xi}I),
	\]
	where $\hat{\bm{x}}_0(\bm{x}_t)\approx \mathbb{E}[\bm{x}_0|\bm{x}_t]$ is a posterior mean estimator with a structured network. Instead of relying purely on the mean, TMPD \cite{boys2023tweedie} utilizes an extended Tweedie’s formula to estimate the full covariance matrix $\text{Cov}[\bm{x}_0|\bm{x}_t]$, which involves calculating the complex Jacobian of the score function. The diffusion process is then projected such that it matches the exact statistical moments of the true posterior at every step. Subsequent work of Resample \cite{song2023solving} extended this paradigm to latent diffusion models through a two-stage data consistency framework, incorporating data consistency by solving an optimization problem during the reverse sampling process. Despite their strong empirical performance, most existing methods require backpropagation through the posterior mean estimator at each sampling step, which leads to substantial computational overhead and memory consumption.
	
	Other categories to avoid calculation of the time-dependent conditional score $\nabla_{\bm{x}} \log p(\bm{y}|\bm{x}_t)$ are plug-and-play (PnP) and variational inference (VI). PnP algorithms adapt decades of classical convex optimization literature by substituting classical math filters with diffusion models acting as regularizing denoisers. VI formalizes inverse problem solving from a purely statistical perspective by minimizing the evidence lower bound (ELBO) of the true posterior. Both focus on solving specific regularized optimization problems. DiffPIR \cite{zhu2023denoising} bridges the gap between classical half-quadratic splitting (HQS) optimization and DDPM prior for nonlinear inverse problems. It splits the MAP objective into two alternating subproblems. The data-fidelity step solves the physics equation, while the prior step solves a regularized problem by replacing the classical prior entirely with the unconditional DDPM estimation. RED-Diff \cite{mardani2023variational} constructs an explicit regularized objective using a DDPM denoiser and performs joint gradient-based optimization. DAPS \cite{zhang2025improving} mathematically decouples the physical diffusion SDE timestep $t$ from an algorithmic annealing parameter. The formulation becomes a sequence of relaxed inverse problems where the assumed measurement noise variance is gradually reduced to zero completely independently of the SDE state, dynamically easing the strictness of data-consistency. EquiReg \cite{tolooshams2025equireg} further enforces physical symmetries, such as rotation or translation equivariance, during sampling.
	
	Despite their success, most of these methods are based on traditional score matching. As a result, for nonsmooth priors in the real world, these methods lack convergence guarantees and may degrade in performance for practical applications.

	\textbf{Proximal Langevin Algorithm.} Classical posterior sampling has long been dominated by Markov chain Monte Carlo (MCMC) methods. These approaches are most effective when the prior is conjugate to the likelihood or when the negative log-density is smooth with Lipschitz continuous gradients \cite{geyer1992practical}. Such assumptions are violated by many modern imaging priors involving nonsmooth regularizers, including $\ell_1$ penalties or total variation norms.
	
	To address this limitation, proximal Langevin algorithms have been developed \cite{pereyra2016proximal,salim2019stochastic,pillai2024optimal}. These methods constitute an important step toward bridging convex optimization and stochastic simulation by replacing gradients with proximal mappings and Moreau–Yosida envelopes. This convex analytic machinery enables scalable Bayesian inference for high-dimensional log-concave models that are not differentiable, thereby supporting uncertainty quantification and sparse regression in settings previously accessible mainly through deterministic optimization.
	
	However, existing proximal unadjusted Langevin algorithms still rely on explicit knowledge of the prior potential, which restricts their applicability in imaging sciences. In parallel, the field has witnessed a shift away from hand-crafted regularizers toward expressive learned priors, including variational autoencoders, generative adversarial networks, and diffusion models \cite{kingma2013auto,goodfellow2014generative,ho2020denoising,chung2022diffusion}. Such implicit priors are accessible only through samples or score networks, preventing the direct application of traditional Langevin or proximal MCMC schemes that require analytic access to the underlying potential function.
	
	\textbf{Convergence of the diffusion model.} The empirical success of diffusion models has motivated rigorous efforts to understand their theoretical convergence properties \cite{de2022convergence,benton2023nearly,gu2023optimal,li2024accelerating,zhang2026gradient}. For example, \cite{lee2022convergence} provided the first polynomial convergence guarantee for smooth distributions satisfying the log-Sobolev inequality. \cite{chen2022sampling} accommodated a broad family of data distributions under the assumption that the score functions along the entire forward process trajectory are Lipschitz. A non-asymptotic analysis for the discretized probability flow ODE was derived by \cite{chen2023restoration}, establishing the first non-asymptotic convergence guarantees. Further, \cite{li2023towards} developed a suite of non-asymptotic theory to understand the discrete-time data generation process, demonstrating a faster convergence rate. 
	
	Despite their success, a fundamental limitation persists: an early-stopping mechanism is typically required to ensure smoothness, which introduces a bias dependent on the stopping time. This early-stopping error corresponds to the deviation between the reverse sampling time region $t\in[\epsilon,\infty)$ and the forward diffusion time region $t\in[0,\infty)$. It typically arises from partitioning the sampling interval into $[0,\epsilon]$ and $[\epsilon,T]$ to address nonsmoothness near $t=0$. As $t$ approaches zero, the true score $\nabla \log \pi_t(\bm{x}_t)$ diverges, creating a singularity that must be avoided. 
	
	Currently, there are still limited convergence results for diffusion models in the context of solving inverse problems. In \cite{song2023solving}, a variance induction result is provided for the stochastic resampling process. Further in \cite{li2025efficient}, a total variation distance error is characterized for the middle approximation processes. However, these results rely heavily on strong assumptions, such as the Gaussian distribution or dirac posterior assumption.

	\subsection{Connections with Existing Algorithms}
	
	\textbf{Connection with P-ULA \cite{pereyra2016proximal}.} Consider a VE-SDE forward process with $\mu(t)\equiv1$ and $ \sigma(\tau_{k+1})/\sigma(\tau_{k}) \equiv 1-\rho_k$, the iteration \eqref{eq:sample} reduces to
	\[
	\bar{\bm{x}}_{k+1} = \left(1-\rho_k\right)\bar{\bm{x}}_k+ \rho_k \bm{P}_k+ \sqrt{\lambda(\tau_k)/\lambda} \bm{\xi}_k,
	\]
	which recovers the P-ULA iteration in \eqref{eq:pula} up to a minor difference in the noise scaling. If we instead set $\sigma(\tau_{k+1})/\sigma(\tau_{k}) = \rho_k^2$ and consider the reverse-time ordinary differential equation associated with \eqref{eqn_reverse1}, then
	\[
	\bm{x}_{k+1}= (1-\rho_k)\bm{x}_k + \rho_k \bm{P}_k,
	\]
	which can be interpreted as a generalized proximal point iteration \cite{rockafellar1976monotone,eckstein1992douglas}.
	
	\textbf{Connection with DPS \cite{chung2022diffusion}.} In the context of linear inverse problem \eqref{eqn_inverse} with $\mathcal{A}(\bm{x}) = A\bm{x}$, suppose that $p(\bm{x}_0|\bm{x}_t)\sim \mathcal{N}(\hat{\bm{x}}_0(\bm{x}_t),\lambda(t)I)$ for some posterior mean estimator $\hat{\bm{x}}_0(\bm{x}_t)$ and $b(t)\equiv1$, then DPS \cite{chung2022diffusion} takes the reverse SDE:
	\[
	d\bm{x}_t = \frac{\bm{x}_t-\hat{\bm{x}}_0(\bm{x}_t) + \eta_t \lambda(t) \nabla_{\bm{x}_t} f_{\bm y}(\hat{\bm{x}}_0(\bm{x}_t)) }{\lambda(t)}dt +d\bar{\bm{B}}_t ,
	\]
	with a predefined weighting $\eta_t$. In contrast, PGM employs
	\[
	d\bm{x}_t = \frac{\bm{x}_t-\text{Prox}_g^{\lambda(t)}\left(\bm{x}_t-\beta \lambda(t)\nabla f_{\bm{y}}\left(\bm{x}_t\right)\right)}{\lambda(t)}dt+ d\bar{\bm{B}}_t.
	\]
	Compared with DPS, PGM avoids backpropagation through the posterior mean estimator $\hat{\bm{x}}_0(\bm{x}_t)$, resulting in substantial savings in memory usage and sampling time, which is confirmed by the experiments in Section \ref{sec:experiments}.
	
	\textbf{Connection with DiffPIR \cite{zhu2023denoising}.} Both methods address Bayesian inverse problems with prior \(p\propto \exp\{-g\}\) and likelihood \(\exp\{-f\}\), and both are related to proximal-type optimization ideas. However, the key difference lies in how the optimization problem is approximated and solved. From an optimization viewpoint, the target problem is
	\[
	\min_{\bm{x}}\  f(\bm{x})+g(\bm{x}).
	\]
	DiffPIR adopts a half-quadratic splitting strategy and considers the augmented problem
	\[
	\min_{\bm{x},\bm{z}}\ f(\bm{x})+g(\bm{z})+\frac{1}{2\lambda/\mu}\|\bm{x}-\bm{z}\|^2,
	\]
	which leads to the two subproblems
	\[
	\bm{z}_k=\arg\min_{\bm{z}} \frac{1}{2\lambda/\mu}\|\bm{z}-\bm{x}_k\|^2+g(\bm{z}),
	\]
	\[
	\bm{x}_{k+1}=\arg\min_{\bm{x}} f(\bm{x})+\mu\sigma_n^2\|\bm{x}-\bm{z}_k\|^2.
	\]
	To solve the first subproblem, DiffPIR further uses a first-order approximation
	\[
	g(\bm{z})\approx g(\bm{x}_k)+\nabla g(\bm{x}_k)^\top (\bm{z}-\bm{x}_k),
	\]
	where the Stein score \(\bm{s}_{\bm{\theta}}(\bm{x})\) is used to approximate \(\nabla \log p(\bm{x})\), equivalently \(-\nabla g(\bm{x})\). This approximation relies on the smoothness of \(g\), which does not hold in our nonsmooth setting. In contrast, PGM does not approximate the prior term \(g\). Instead, we linearize the likelihood term
	\[
	f(\bm{x})\approx f(\bm{x}_k)+\nabla f(\bm{x}_k)^\top (\bm{x}-\bm{x}_k),
	\]
	and apply a proximal-gradient step directly to the original problem:
	\[
	\bm{x}_{k+1}=\arg\min_{\bm{x}} \nabla f(\bm{x}_k)^\top (\bm{x}-\bm{x}_k)+g(\bm{x})+\frac{1}{2\mu\sigma_n^2}\|\bm{x}-\bm{x}_k\|^2.
	\]
	This leads to the proximal splitting
	\[
	\bm{x}_{k+1}
	=
	\operatorname{Prox}_{g}^{\lambda/\mu}\!\left(
	\bm{x}_k-\frac{1}{2\mu\sigma_n^2}\nabla f(\bm{x}_k)
	\right).
	\]
	We then learn a proximal network \(\bm{\phi}_{\bm{\theta}}(\bm{x},\lambda)\) to approximate \(\operatorname{Prox}_g^\lambda(\bm{x})\), which naturally accommodates nonsmooth priors.
	
	Overall, while both methods are motivated by proximal optimization, DiffPIR approximates and solves a split formulation by smoothing the prior subproblem, whereas PGM applies proximal gradient directly to the original objective and approximates the proximal operator itself. This difference in the approximation and optimization strategy leads to two distinct frameworks, suitable for different problem settings.

	\subsection{Limitations}
	
	Despite strong performance, PGM has some limitations:
	\begin{itemize}
		\item \textbf{Training complexity}: Learning proximal operators requires careful hyperparameter tuning.
		\item \textbf{Network design}: Network architectures could be more consistent with the proximal properties.
		\item \textbf{Nonconvex priors}: The theoretical guarantee is insufficient under nonconvex settings.
	\end{itemize}
	
	These limitations suggest directions for future work, including more efficient architectures and extensions to non-convex settings.
	
	\subsection{Future Works}
	
	While our framework shows promising results, several interesting directions remain for future research:
	
	\begin{itemize}
		\item \textbf{Nonconvex priors}: Our current theory primarily addresses convex priors. Extending the framework to handle nonconvex priors through nonconvex proximal operators or other generalizations would significantly broaden its applicability.
		
		\item \textbf{Convergence rates}: While we established convergence guarantees, deriving explicit convergence rates under more general conditions and characterizing the dependence on the problem dimension would provide deeper theoretical understanding.
	\end{itemize}
	
	\section{More Details of Experiments}
	\label{ap:experiment}
	\subsection{Experimental Setup}
	
	\textbf{Parameters.} Our PGM implementation uses the following default configurations:
	\begin{itemize}
		\item U-Net architecture: 128 hidden channels, 128 time-embedding dimension, 4 encoder blocks, and 4 decoder blocks.
		\item Training: AdamW optimizer, learning rate $10^{-4}$, weight decay $10^{-5}$, batch size 8, and warm-up cosine scheduler.
		\item Regularization schedule: $\lambda(t) = \exp(10t-8)$. 
		\item Proximal matching schedule: $\zeta(t) = 10/t$.
		\item Inverse temperature: $\beta = 10.0$. 
		\item Number of steps: $K=100$.
	\end{itemize}
	Our code is available at \href{https://github.com/boyangzhang2000/PGM}{https://github.com/boyangzhang2000/PGM}.
	
	\textbf{Network backbone.} All baseline methods use the same pre-trained backbone ADM \cite{dhariwal2021diffusion} to ensure a fair comparison, which improves the U-Net architecture with time embedding, multi-head attention, and BigGAN residual blocks for upsampling/downsampling. PGM uses a similar U-Net backbone for the proximal network, which adopts $\lambda$ embedding and simpler CNN residual blocks.
	
	\textbf{Inverse problem settings.} We apply PGM to inverse problems across four image restoration tasks: super-resolution, random mask inpainting, Gaussian deblurring, and nonlinear deblurring \cite{tran2021explore}. Models are trained on the MNIST \cite{lecun2002gradient}, FFHQ \cite{kazemi2014one}, CelebA-HQ \cite{liu2015deep}, LSUN-Bedroom \cite{yu2015lsun}, and ImageNet-100 \cite{zhao2020maintaining} datasets. Various quantitative metrics are used for evaluation including learned perceptual image patch similarity (LPIPS) distance, peak signal-to-noise-ratio (PSNR), and structural similarity index (SSIM).
	
	The experimental setup for different inverse problems is summarized in Table \ref{tab:exp_settings}:
	\begin{table}[ht]
		\centering
		\caption{Inverse problem configurations}
		\label{tab:exp_settings}
		\begin{tabular}{ccc}
			\hline
			\textbf{Problem type} & \textbf{Key parameters} & \textbf{Values} \\ \hline
			\multirow{3}{*}{\textbf{Inpainting}} & Type & random \\ 
			& Mask ratio & 0.7 \\ 
			& Noise level & 0.01 / 0.05  \\ \hline
			\multirow{3}{*}{\textbf{Super-resolution}} & Scale factor & $4\times$ \\ 
			& Upsampling & Bilinear \\ 
			& Noise level & 0.01 / 0.05 \\ \hline
			\multirow{3}{*}{\textbf{Gaussian deblurring}} & Kernel size & 61 \\ 
			& Sigma ($\sigma$) & 3.0 \\ 
			& Noise level & 0.01 / 0.05 \\ \hline
		\end{tabular}
	\end{table}
	
	\textbf{Baseline methods.} Our proposed PGM framework is compared with state-of-the-art generative models: DPS \cite{chung2022diffusion}, manifold constrained gradients (MCG) \cite{chung2022improving}, denoising diffusion destoration models (DDRM) \cite{kawar2022denoising}, DDNM \cite{wang2022zero}, diffusion model posterior sampling (DMPS) \cite{meng2022diffusion}, plug-and-play using ADMM (ADMM-PnP) \cite{ahmad2020plug}, $\Pi$GDM \cite{song2023pseudoinverse}, posterior sampling with latent diffusion (PSLD) \cite{rout2023solving}, Resample \cite{song2023solving}, Tweedie moment projected diffusion (TMPD) \cite{boys2023tweedie}, RED-diff \cite{mardani2023variational}, DiffPIR \cite{zhu2023denoising}, SITCOM \cite{alkhouri2024sitcom}, decoupled annealing posterior sampling (DAPS) \cite{zhang2025improving}, diffusion state-guided projected gradient (DiffStateGrad) \cite{zirvi2025diffusion}, and equivariance regularized (EquiReg) diffusion \cite{tolooshams2025equireg}.

	\subsection{Training Costs}
	
	As a kind of zero-shot posterior sampling method, PGM still follows the pre-trained prior paradigm. Once the proximal operator is learned, it can be used as a pre-trained operator for any sampling processes or further fine-tuning. Though this training is performed only once on prior samples and can be reused across different inverse problems, there is still necessity to compare the training costs with pre-trained unconditional diffusion models. Some results on the runtime and NFEs of ImageNet-100 dataset are shown in Table \ref{tab:traincost}, where all baseline methods and PGM are tested on 4 NVIDIA A800 GPUs.
	\begin{table}[h]
		\centering
		\caption{Training and sampling costs comparison.}
		\label{tab:traincost}
		\begin{tabular}{llllll}
			\hline
			Method             & Model type                          & Model size              & Training time        & Sampling time & NFEs \\ \hline
			PGM(Ours)          & ProxNet (Our)                       & 2055MB                  & 51h32min             & \textbf{4s}            & \textbf{100}  \\
			DAPS               & \multirow{3}{*}{Pre-trained {[}4{]}} & \multirow{3}{*}{2108MB} & \multirow{3}{*}{N/A} & 33s           & 1000 \\
			DiffStateGrad-DAPS &                                     &                         &                      & 37s           & 1000 \\
			SITCOM             &                                     &                         &                      & 21s           & 600  \\ \hline
		\end{tabular}
	\end{table}
	
	Compared with the baseline methods, the main advantage of PGM is that it achieves substantially lower sampling time and fewer NFEs while remaining competitive in the reconstruction quality. Although training the proximal operator requires additional computations, this is a one-time cost. Once the proximal operator is learned, it can be used as a pre-trained network for any optimization/sampling algorithms or further fine-tuning.
	
	\subsection{Additional Results on MNIST Inpainting}

	We evaluate PGM on the MNIST inpainting task. We compare with VAE-based methods, e.g., VQ-VAE-2 \cite{razavi2019generating}, S-IntroVAE \cite{daniel2021soft}, and $\beta$-VAET \cite{solera2024beta}, as well as diffusion models DDPM \cite{ho2020denoising} and DeepONet \cite{fang2025deeponet}. Various quantitative metrics are used for evaluation, including peak signal-to-noise-ratio (PSNR), structural similarity index (SSIM), and mean square error (MSE). 
	
	\begin{table}[ht]
		\captionof{table}{Quantitative evaluation of samples for MNIST.}
		\label{tab:mnist_inpa1}
		\centering
		\renewcommand{\arraystretch}{1.2}
		\begin{tabular}{llll}
			\hline
			\multicolumn{1}{c}{\multirow{2}{*}{Method}} & \multicolumn{3}{c}{MNIST inpainting} \\ \cline{2-4} 
			\multicolumn{1}{c}{}                         & PSNR $\uparrow$       & SSIM $\uparrow$      & MSE $\downarrow$      \\ \hline
			VAE \cite{kingma2013auto}                                          & 18.0573     & 0.6963     & 0.0477    \\
			VQ-VAE-2 \cite{razavi2019generating}                                        & 20.1591     & 0.7107     & 0.0211    \\
			S-IntroVAE \cite{daniel2021soft}                                  & 21.6251     & 0.8703     & 0.0114    \\
			DDPM \cite{ho2020denoising}                                        & 26.0319     & 0.9448     & 0.0090    \\
			$\beta$-VAET \cite{solera2024beta}                                      & 14.0290     & 0.4101     & 0.5788    \\
			DeepONet \cite{fang2025deeponet}                                    & 28.1062     & 0.9673     & 0.0085    \\ \hline
			PGM (Ours)                                         & \textbf{28.6867}     & \textbf{0.9908}     & \textbf{0.0048}    \\ \hline
		\end{tabular}
	\end{table}
	
	Quantitative results in Table~\ref{tab:mnist_inpa1} show that PGM achieves the highest reconstruction fidelity across all evaluation metrics (including PSNR, SSIM, and MSE), outperforming all the compared baselines. Specifically, PGM reaches a remarkably high SSIM, implying that PGM has excellent structural and visual fidelity, and can perfectly reconstruct important visual structure and texture information in the original images.
	
	\begin{figure}[t]
		\centering
		\includegraphics[width=0.45\linewidth]{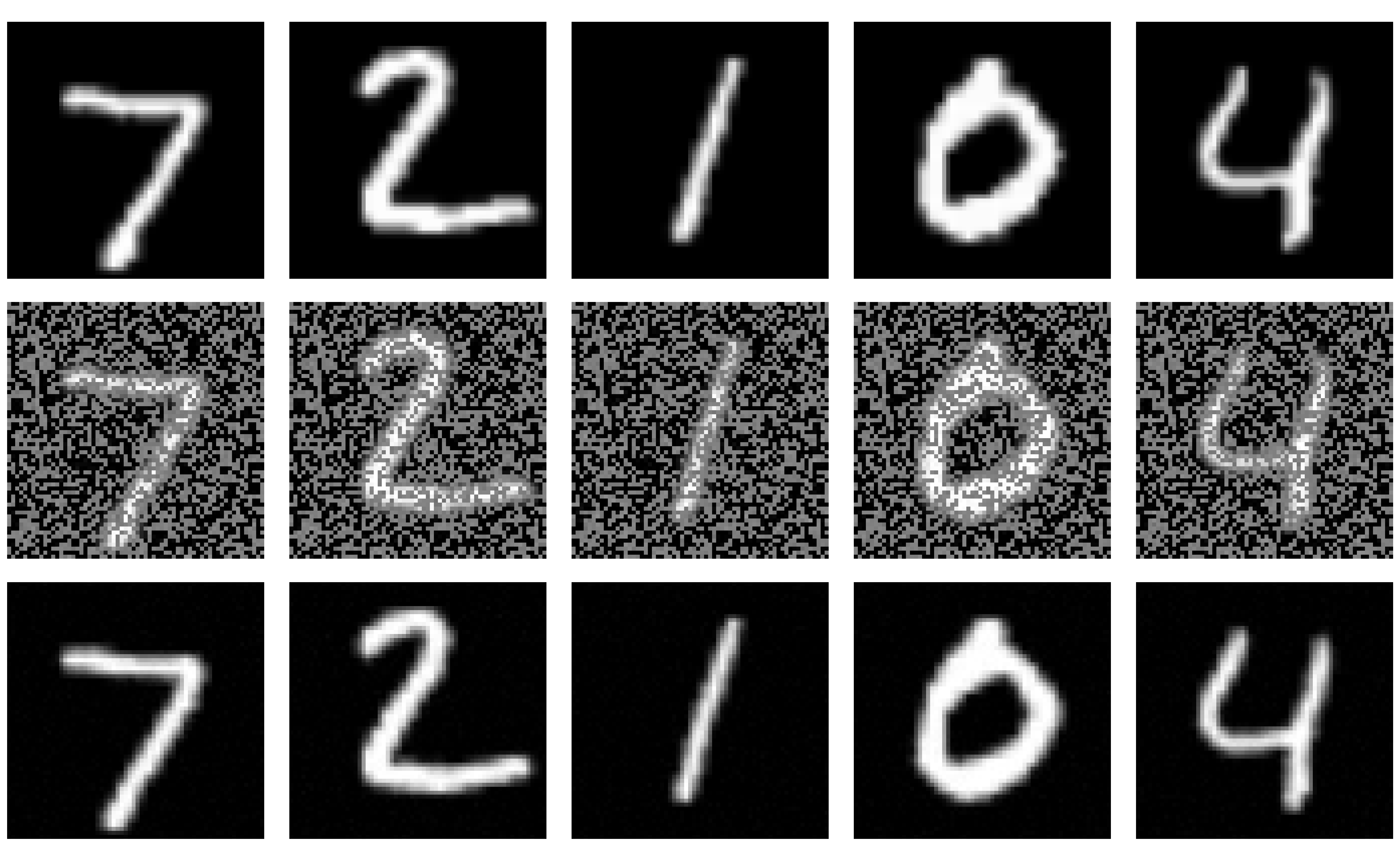}
		\caption{Samples for MNIST (first line: original images, second line: measurements, third line: reconstructed images). }
		\label{fig:mnist1}
	\end{figure}

	\subsection{Additional Results on Human Face Reconstruction}
	
	Qualitative results on FFHQ and CelebA-HQ are displayed in Figures~\ref{fig:ffhq2} and~\ref{fig:celeb2}, respectively, visually confirming the model's capability to generate high-fidelity and natural-looking images.
	
	\begin{figure}[htbp]
		\centering
		\includegraphics[width=0.85\linewidth]{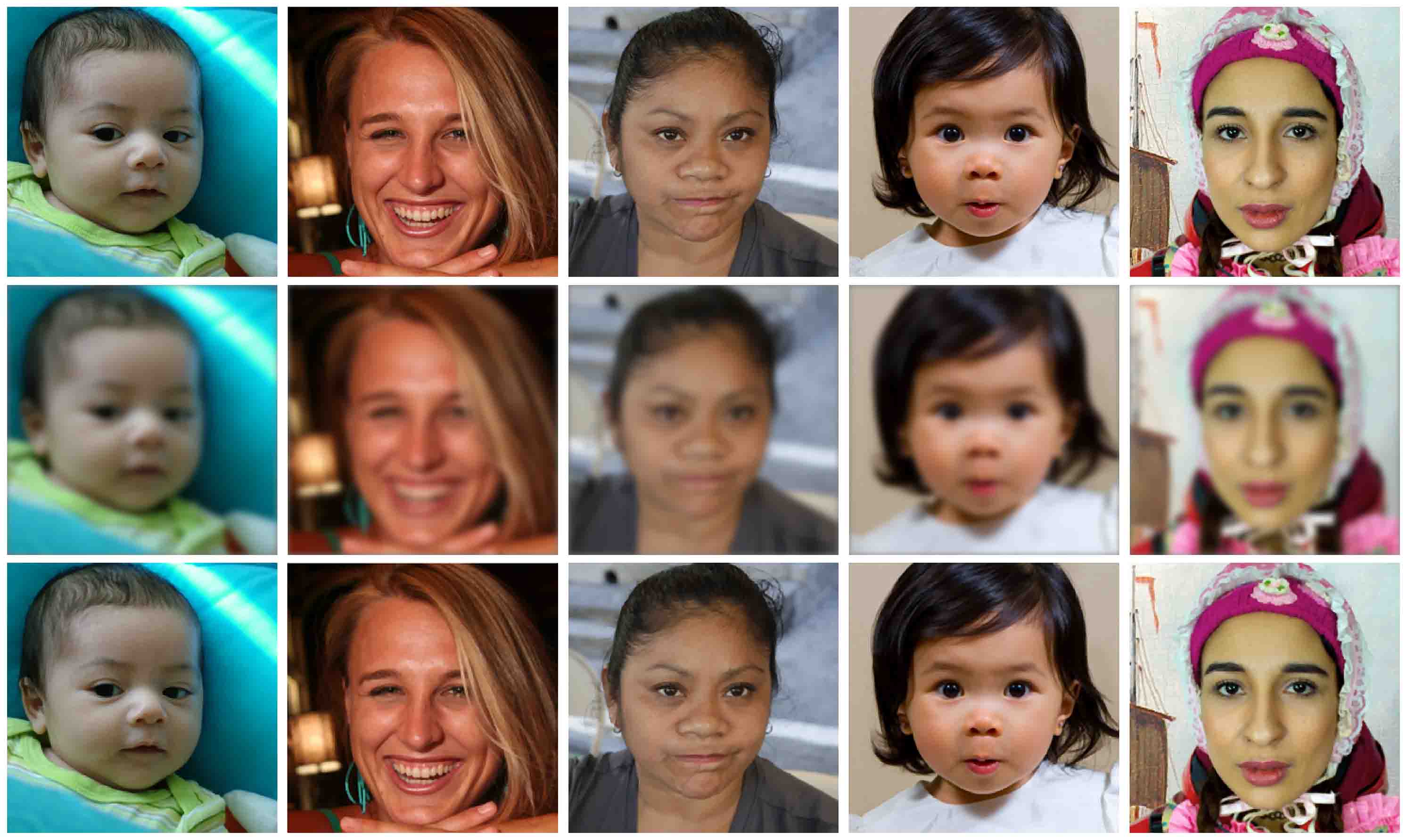}
		\caption{Samples for FFHQ (first line: original images, second line: measurements, third line: reconstructed images).}
		\label{fig:ffhq2}
	\end{figure}
	
	\begin{figure}[htbp]
		\centering
		\includegraphics[width=0.85\linewidth]{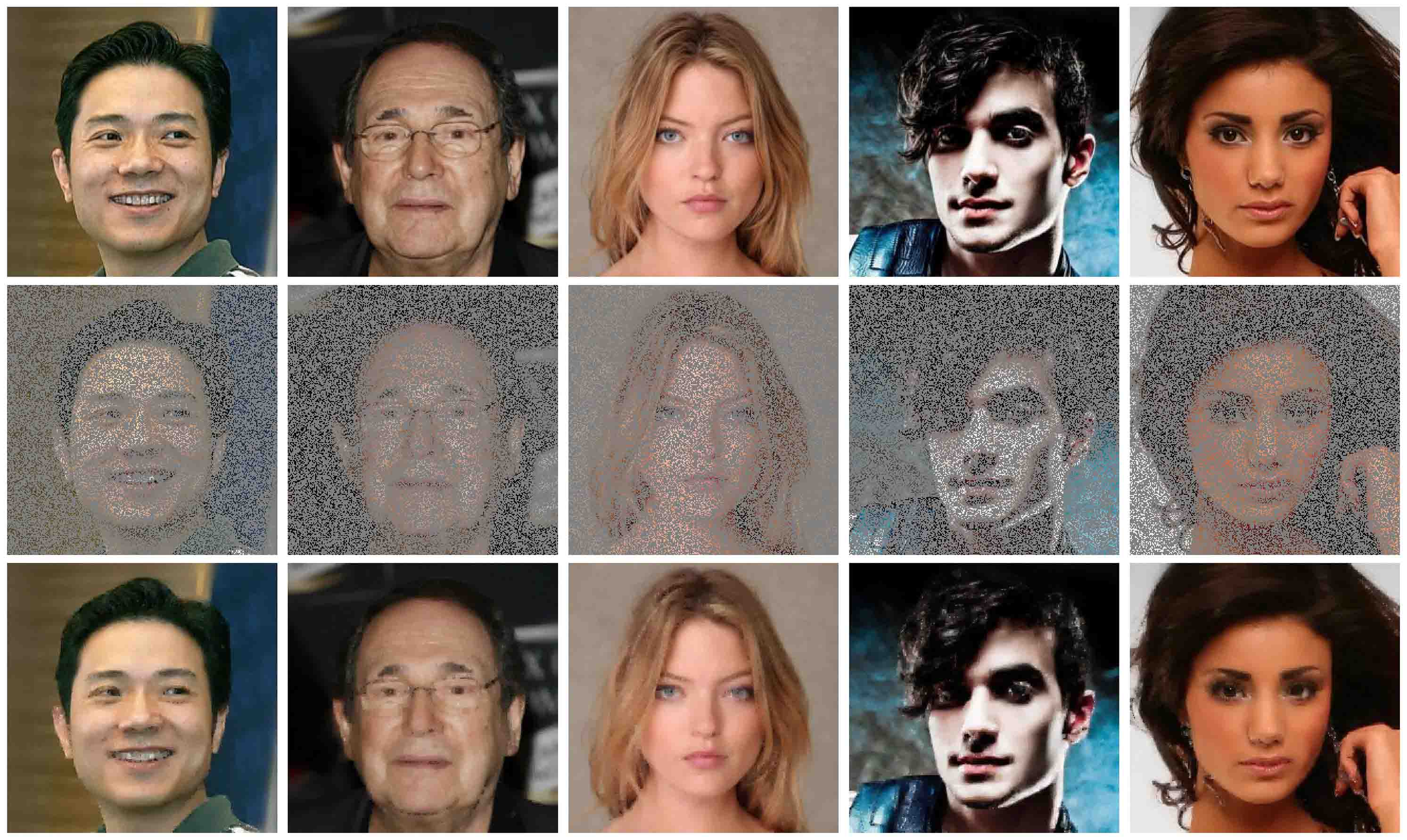}
		\caption{Samples for CelebA-HQ (first line: original images, second line: measurements, third line: reconstructed images). }
		\label{fig:celeb2}
	\end{figure}
	
	Further, we provide the Pareto front to compare the trade-off between reconstruction quality and inference time. In Figure \ref{fig:pareto}, we show the least inference cost, measured by the number of neural function evaluations (NFEs), of various methods with a fixed target reconstruction quality (PSNR) for the FFHQ Gaussian deblurring problem.
	
	\begin{figure}[t]
		\centering
		\includegraphics[width=0.45\linewidth]{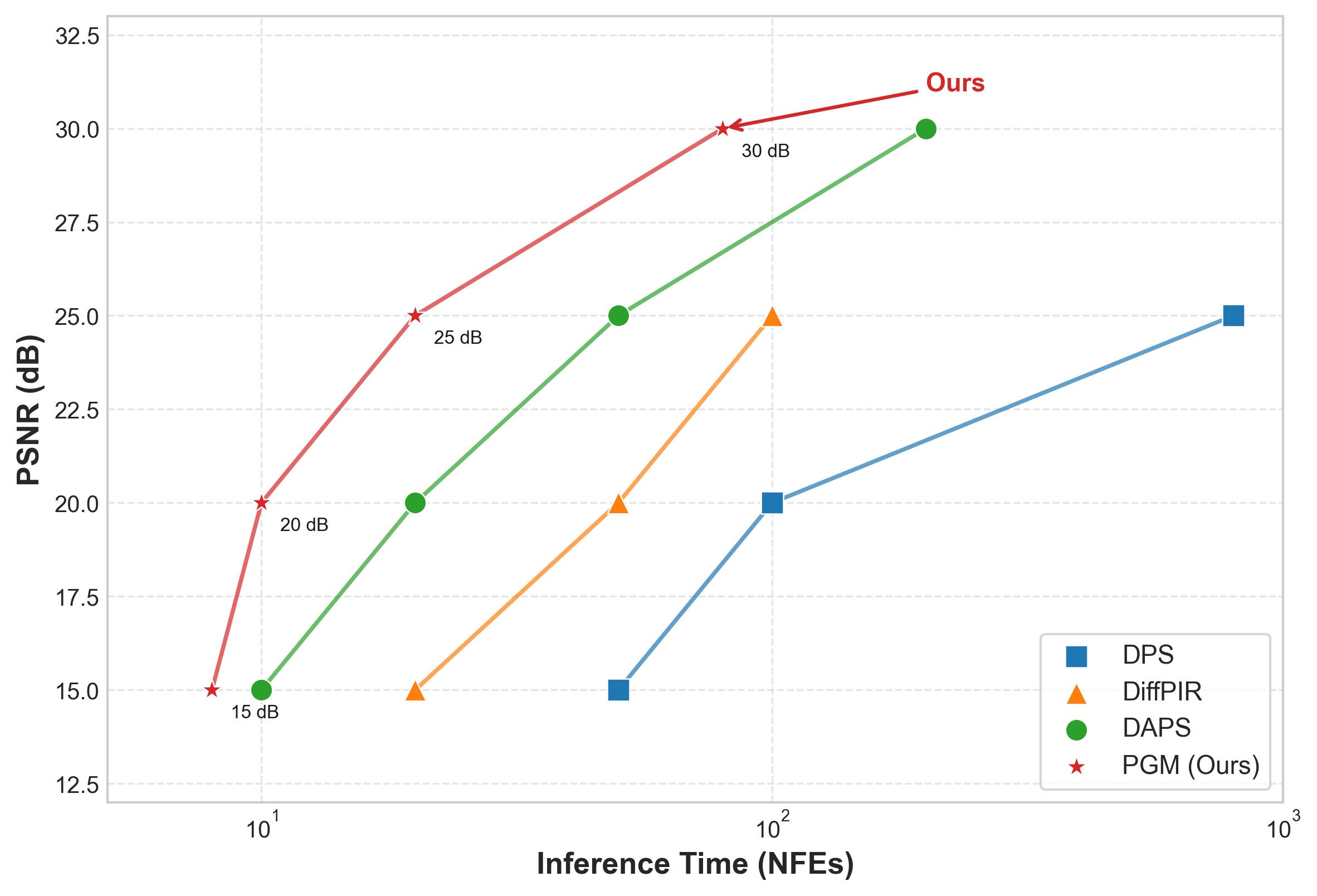}
		\caption{Trade-off between reconstruction quality and inference time. }
		\label{fig:pareto}
	\end{figure}
	
	The experimental results indicate that, for a fixed PSNR target, the number of NFEs required by PGM is less than that of other methods, and this difference becomes more pronounced as the PSNR increases. This suggests that, compared to existing methods, PGM exhibits a better performance trade-off.
	
	\subsection{Additional Results on ImageNet-100}
	
	We further test performances on ImageNet-100 dataset and compare with recent works on diffusion-based inverse problems. Some results are shown in Table \ref{tab:imagenet}. PSNR and LPIPS results are presented as mean ± standard deviation over 100 fixed validation images. Some results are shown below, where performances of baseline methods are taken from \cite{alkhouri2024sitcom,zhang2025improving,zirvi2025diffusion}.
	\begin{table}[ht]
		\captionof{table}{Quantitative evaluation of samples for ImageNet-100.}
		\label{tab:imagenet}
		\centering
		\renewcommand{\arraystretch}{1.2}
		\begin{tabular}{ccccccc}
			\hline
			ImageNet-100           & \multicolumn{2}{c}{Super Resolution 4×}    & \multicolumn{2}{c}{Inpainting Random 70\%} & \multicolumn{2}{c}{Gaussian Deblurring}    \\
			Method             & PSNR$\uparrow$                & LPIPS$\downarrow$                & PSNR$\uparrow$                & LPIPS$\downarrow$                & PSNR$\uparrow$                & LPIPS$\downarrow$                \\ \hline
			PGM(Ours)          & \textbf{26.60±2.90} & 0.264±0.060          & \textbf{31.14±3.28} & 0.116±0.045          & 26.65±2.93          & \textbf{0.206±0.050} \\
			DPS                & 23.86±0.34          & 0.357±0.069          & 24.26±0.42          & 0.326±0.034          & 21.86±0.45          & 0.362±0.034          \\
			DDNM               & 23.96±0.89          & 0.475±0.044          & 29.22±0.55          & 0.191±0.048          & \textbf{28.06±0.52} & 0.278±0.089          \\
			DAPS               & 25.67±0.73          & 0.256±0.067          & 28.44±0.45          & 0.135±0.052          & 26.12±0.78          & 0.245±0.022          \\
			DiffStateGrad-DAPS & 26.40±3.44          & \textbf{0.229±0.057} & 29.78±4.17          & \textbf{0.107±0.037} & 25.87±3.56          & 0.243±0.075          \\
			SITCOM             & 26.35±1.21          & 0.232±0.038          & 29.60±0.78          & 0.127±0.039          & 27.40±0.45          & 0.236±0.039          \\ \hline
		\end{tabular}
	\end{table}
	
	Based on the quantitative results on ImageNet-100, our proposed PGM achieves competitive performance across three inverse problems. For super resolution and inpainting, PGM attains the highest PSNR while DiffStateGrad-DAPS obtains the best LPIPS. For Gaussian deblurring, PGM obtains the best LPIPS and competitive PSNR. Overall, PGM demonstrates strong reconstruction quality in terms of both distortion and perceptual similarity.
	
	\subsection{Additional Results on Nonlinear Deblurring}
	
	In our PGM framework, the only information about \(f\) required is its gradient \(\nabla f\); therefore, the method naturally extends to nonlinear inverse problems. Since nonlinearity often leads to nonconvexity, the assumptions underlying Theorem \ref{thm:dis_conv} may no longer hold in such settings. As a result, the current convergence theory is not directly applicable. Nevertheless, the algorithm itself remains well-defined and can still be used in practice. To evaluate its behavior beyond the convex regime, we add additional experiments on nonlinear inverse problems. We use the default setting and leverage the neural network approximated forward model as described in \cite{tran2021explore} for nonlinear deblurring, i.e., $f_{\bm{y}}(\bm{x})=\frac{1}{2}\|\bm{y}-\mathcal{G}_{\bm{\theta}}(\bm{x})\|^2$. All Gaussian noise is added to the measurement domain with $\sigma = 0.05$. Selected results are listed in Table \ref{tab:nonlinear}.
	\begin{table}[ht]
		\captionof{table}{Quantitative evaluation of nonlinear deblurring.}
		\label{tab:nonlinear}
		\centering
		\renewcommand{\arraystretch}{1.2}
		\begin{tabular}{cccccccc}
			\hline
			\multicolumn{2}{c}{Nonlinear  Deblurring} & \multicolumn{3}{c}{FFHQ}                         & \multicolumn{3}{c}{ImageNet}                     \\
			Methods       & Type                       & PSNR           & SSIM           & LPIPS          & PSNR           & SSIM           & LPIPS          \\ \hline
			PGM(Ours)     & \multirow{4}{*}{Pixel}     & 28.47          & 0.756          & 0.231          & 26.39          & 0.727          & 0.280          \\
			DPS           &                            & 23.39          & 0.623          & 0.278          & 22.49          & 0.591          & 0.306          \\
			RED-diff      &                            & \textbf{30.86} & \textbf{0.795} & 0.160          & \textbf{30.07} & \textbf{0.754} & 0.211          \\
			DAPS          &                            & 28.29          & 0.783          & \textbf{0.155} & 27.73          & 0.724          & \textbf{0.169} \\ \hline
			ReSample      & \multirow{2}{*}{Latent}    & 28.24          & 0.742          & 0.185          & 26.20          & 0.653          & 0.206          \\
			LatentDAPS    &                            & 28.11          & 0.713          & 0.235          & 25.34          & 0.615          & 0.314          \\ \hline
		\end{tabular}
	\end{table}
	
	The experimental results indicate that, although the performance of PGM is superior to that of the classical DPS, it is inferior compared to the latest methods suitable for handling nonlinear inverse problems. This is caused by the first-order approximation \(f(\bm{x})\approx f(\bm{x}_k) + \nabla f(\bm{x}_k)^\top (\bm{x}-\bm{x}_k)\) in operator splitting. As nonlinear inverse problems often suffer from severe nonconvexity, the performance of PGM may degrade. In future work, we will investigate how to extend PGM to general nonconvex problems and explore its convergence results in such settings.

	\section{Properties of Moreau Approximation}
	\label{ap:moreau}
	
	\subsection{Moreau-Yosida Regularization}
	
	Given a proper closed convex function $f: \mathbb{R}^d \rightarrow \mathbb{R} \cup \{+\infty\}$, its Moreau-Yosida regularization (or Moreau envelope) with parameter $\lambda > 0$ is defined as:
	\begin{equation}
		f^\lambda(\bm{x}) = \min_{\bm{u}\in\mathbb{R}^d} \left\{ f(\bm{u}) + \frac{1}{2\lambda} \|\bm{u} - \bm{x}\|^2 \right\}.
	\end{equation}
	And the proximal operator associated with $f$ is:
	\begin{equation}
		\text{Prox}^{\lambda}_{f}(\bm{x}) = \mathop{\arg\min}_{\bm{u}\in\mathbb{R}^d} \left\{ f(\bm{u}) + \frac{1}{2\lambda} \|\bm{u} - \bm{x}\|^2 \right\}.
	\end{equation}
	The Moreau envelope possesses several fundamental properties (for the proof, please refer to \cite{rockafellar1998variational}).
	
	\begin{lemma}[Properties of Moreau envelope]
		\label{lemma_ME}
		\normalfont
		Let $f$ be proper, closed, and convex. For any $\lambda>0$, we have:
		\begin{itemize}
			\item (Convergence) $f^\lambda(\bm{x}) \rightarrow f(\bm{x})$ as $\lambda \rightarrow 0^+$;
			\item (Differentiability) $f^\lambda(\bm{x})$ is convex and continuously differentiable with gradient 
			\begin{equation}
				\nabla f^\lambda(\bm{x}) = \frac{1}{\lambda}(\bm{x} - \text{Prox}^{\lambda}_{f}(\bm{x}));
			\end{equation}
			\item (Subdifferential) The point $\frac{1}{\lambda}(\bm{x}-\text{Prox}^{\lambda}_{f}(\bm{x}))$ belongs to the subdifferential of $f(\bm{u})$, i.e., 
			\begin{equation}\label{eqn_prox_uf}
				\bm{u}=\text{Prox}^{\lambda}_{f}(\bm{x}) \iff \frac{1}{\lambda}(\bm{x}-\bm{u})\in \partial f(\bm{u});
			\end{equation}
			\item (L-smoothness) $f^\lambda(\bm{x})$ has $\frac{1}{\lambda}$-Lipschitz continuous gradient.
		\end{itemize}
	\end{lemma}
	
	\subsection{Moreau Approximation}
	
	Given a probability density $\pi(\bm{x}) \propto \exp\{-f(\bm{x})\}$, its $\lambda$-Moreau approximation is defined as:
	\begin{equation}
		\label{eqn_pi_lambda}
		\pi^{\lambda}(\bm{x})=\frac{1}{Z} \sup_{\bm{u}\in \mathbb{R}^d} \left\{\pi(\bm{u}) \exp\left\{-\|\bm{u}-\bm{x}\|^2/2\lambda\right\}\right\},
	\end{equation}
	where $Z$ is the normalizing constant ensuring that $\pi^{\lambda}$ integrates to one.
	
	The connection between this density approximation and the Moreau envelope follows directly from the definition in \eqref{eqn_pi_lambda}:
	\begin{align}
		\label{eqn_sup_exp}
		\pi^{\lambda}(\bm{x})\propto\sup_{\bm u}\exp\!\left\{- f(\bm{u})- \frac{1}{2\lambda}\|\bm{u}-\bm{x}\|^2\right\} 
		\propto\exp\!\left\{- \inf_{\bm u}\left\{f(\bm{u})+ \frac{1}{2\lambda}\|\bm{u}-\bm{x}\|^2\right\}\right\}  
		\propto\exp\!\left\{- f^{\lambda}(\bm{x})\right\},
	\end{align}
	where the second proportionality follows from the monotonicity of the exponential function, and the last one uses the definition of the Moreau--Yosida regularization. This shows that applying the Moreau approximation to a density $\pi$ is equivalent to applying the Moreau--Yosida regularization to its negative log-density.

	\begin{proposition}[Properties of Moreau approximation]
		\normalfont
		Let $\pi(\bm{x}) \propto \exp\{-f(\bm{x})\}$, where $f: \mathbb{R}^d \rightarrow \mathbb{R} \cup \{+\infty\}$ is proper, closed, and convex. For any $\lambda>0$, we have:
		\begin{itemize}
			\item (Convergence) $\pi^{\lambda}(\bm{x}) \rightarrow \pi(\bm{x})$ point-wise as $\lambda \rightarrow 0^+$;
			\item (Differentiability) $\pi^{\lambda}(\bm{x})$ is convex and continuously differentiable with log-gradient 
			\begin{equation}\label{eqn_score_mor}
				\nabla \log \pi^{\lambda}(\bm{x}) = \frac{1}{\lambda}(\text{Prox}^{\lambda}_{f}(\bm{x})-\bm{x});
			\end{equation}
			\item (Subdifferential) The point $\frac{1}{\lambda}(\text{Prox}^{\lambda}_{f}(\bm{x})-\bm{x})$ belongs to the subdifferential of $\log \pi(\bm{u})$, i.e., 
			\begin{equation}
				\bm{u}=\text{Prox}^{\lambda}_{f}(\bm{x}) \iff \frac{1}{\lambda}(\bm{u}-\bm{x})\in \partial \log \pi(\bm{u});
			\end{equation}
			\item (L-smoothness) $\pi^{\lambda}(\bm{x})$ has $\frac{1}{\lambda}$-Lipschitz continuous log-gradient;
		\end{itemize}
	\end{proposition}
	
	\begin{proof} 
		The convergence follows from the fact that the Gaussian kernel $\exp\left\{-\|\bm{u}-\bm{x}\|^2/(2\lambda)\right\}$ converges to a Dirac delta at $\bm{x}$ as $\lambda \to 0^+$. Consequently,
		\begin{equation*}
			\lim_{\lambda \to 0^+} \pi^\lambda(\bm{x}) = \pi(\bm{x}),\quad \forall~ \bm{x}.
		\end{equation*}
		
		From \eqref{eqn_sup_exp}, we have $\pi^{\lambda}(\bm{x}) \propto \exp\{-f^{\lambda}(\bm{x})\}$. By Lemma~\ref{lemma_ME}, $f^{\lambda}$ is continuously differentiable with gradient 
		$$\nabla f^{\lambda}(\bm{x}) = \frac{1}{\lambda}\left(\bm{x} - \text{Prox}^{\lambda}_{f}(\bm{x})\right).$$
		Therefore, \(\pi^\lambda(\bm{x}) = \exp\{-f^\lambda(\bm{x})\}/Z_\lambda\) is continuously differentiable and its log-gradient 
		\begin{equation*}
			\nabla \log \pi^\lambda(\bm{x}) = -\nabla f^\lambda(\bm{x}) = \frac{1}{\lambda}\bigl(\text{Prox}_f^\lambda(\bm{x}) - \bm{x}\bigr).
		\end{equation*}
		
		By the optimality condition of the proximal operator, we get
		\begin{equation*}
			\bm{u} = \text{Prox}_f^\lambda(\bm{x}) \iff \mathbf{0} \in \partial f(\bm{u}) + \frac{1}{\lambda}(\bm{u} - \bm{x}).
		\end{equation*}
		Since $\log \pi(\bm{u}) = -f(\bm{u}) + \text{const}$, we have $\partial \log \pi(\bm{u}) = -\partial f(\bm{u})$. Thus,
		\begin{equation*}
			\bm{u} = \text{Prox}_f^\lambda(\bm{x}) \iff \frac{1}{\lambda}(\bm{u} - \bm{x}) \in -\partial f(\bm{u}) = \partial \log \pi(\bm{u}).
		\end{equation*}
		
		The gradient of $f^\lambda$ is Lipschitz continuous with constant $1/\lambda$, i.e., 
		\begin{equation*}
			\|\nabla f^\lambda(\bm{x}) - \nabla f^\lambda(\bm{y})\| \leq \frac{1}{\lambda} \|\bm{x} - \bm{y}\|.
		\end{equation*}
		Since $\nabla \log \pi^\lambda = -\nabla f^\lambda$, the log-gradient of $\pi^\lambda$ is also $1/\lambda$-Lipschitz.
		
	\end{proof}

	\section{Moreau-Yosida-Gaussian Equivalence}
	\label{ap:equivalent}
	
	\subsection{Stein Score and Moreau Score}
	
	Firstly, we provide a detailed discussion comparing the Stein score $\nabla_{\bm{x}} \log \pi_t(\bm{x}_t)$ used in diffusion models and Moreau score $\nabla_{\bm{x}} \log \pi_t^{\lambda}$.
	
	In diffusion models, the classic \emph{Stein score} is the gradient of the log-marginal density \(\pi_t\) of the forward process. By Tweedie's formula~\cite{efron2011tweedie}, it can be expressed in terms of the posterior mean:
	\[
	\nabla_{\bm{x}} \log \pi_t(\bm{x}_t)=\frac{\mu(t)\,\mathbb{E}[\bm{x}_0|\bm{x}_t]-\bm{x}_t}{\sigma^2(t)},
	\]
	where \(\mu(t)\) and \(\sigma^2(t)\) are the scale and noise variance of the forward transition. In contrast, the proposed \emph{Moreau score} replaces the conditionally expected clean image \(\mathbb{E}[\bm{x}_0|\bm{x}_t]\) with the proximal operator of the negative log-prior \(U\), yielding
	\[
	\nabla_{\bm{x}} \log \pi_t^{\lambda}(\bm{x}_t)=\frac{\mu(t)\,\mathrm{Prox}_{U}^{\lambda(t)}\!\bigl(\frac{\bm{x}_t}{\mu(t)}\bigr)-\bm{x}_t}{\sigma^2(t)}.
	\]
	However, the classic Stein score faces two fundamental challenges. First, due to the integral form of \(\pi_t\), it does not admit a closed-form expression and therefore can only be learned approximately via denoising score matching in \eqref{eq:denoise}. Second, for nonsmooth priors, the Stein score may not be well-defined. In contrast, the Moreau score is given explicitly by a deterministic proximal operator, and it remains well-defined for any proper, closed, convex potential \(U\).
	
	From a Bayesian perspective, \(\mathbb{E}[\bm{x}_0|\bm{x}_t]\) is the posterior mean, whereas the proximal output \(\mathrm{Prox}_{U}^{\lambda(t)}\!\bigl(\frac{\bm{x}_t}{\mu(t)}\bigr)\) is the posterior mode under the forward Gaussian likelihood. Hence the Stein score and the Moreau score differ precisely by the deviation between the posterior mean and the posterior mode. When the potential \(U\) is quadratic, the mean and mode of the posterior coincide, and the two scores become exactly equal (Lemma~\ref{pro:equivalence}). For general $m-$strongly convex \(U\), Proposition~\ref{pro:asy_equ} bounds the score discrepancy uniformly. Consequently, the Moreau score provides a principled, analytically tractable approximation to the Stein score, with an error controlled by the curvature of the potential, while naturally handling nonsmooth priors that are inaccessible to the classical formulation.

	\subsection{Proof of Lemma \ref{eq_lemma_c1}}
	In this subsection, we prove Lemma~\ref{eq_lemma_c1}, which shows that convolving a
	quadratic-density distribution with a Gaussian kernel is equivalent---up to a
	normalization constant---to applying the Moreau--Yosida regularization to its
	potential function.
	
	The main idea comes from the similarity between Gaussian convolution and Moreau-Yosida regularization, both of which involving quadratic Gaussian kernels. One key observation is that for quadratic potentials, the integration in convolution can be replaced equivalently by the supremum via Laplace approximation, which typically yields the Moreau-Yosida regularization. 
	
	\begin{lemma}[Moreau--Yosida--Gaussian equivalence]
		\normalfont
		\label{pro:equivalence}
		Let $\pi(\bm{x}) \propto \exp\{-U(\bm{x})\}$ be a probability density on 
		$\mathbb{R}^d$, where $U(\bm{x})=\bm{x}^\top A\bm{x}$ with
		$A\succ 0$. For any $\lambda>0$, the following two densities that are proportional to each other:
		
		\begin{itemize}
			\item[(i)] \textbf{Moreau approximation:}
			\begin{equation}
				\label{eqn_pi_lamda_def}
				\pi_t^{\lambda}(\bm{x})
				\propto
				\sup_{\bm{u}\in\mathbb{R}^d}
				\pi(\bm{u})
				\exp\!\left\{-\|\bm{u}-\bm{x}\|^2/(2\lambda)\right\}
				\propto
				\exp\{-U^\lambda(\bm{x})\}.
			\end{equation}
			
			\item[(ii)] \textbf{Gaussian convolution:}
			\begin{equation}
				\label{eqn_conv}
				\pi_t(\bm x)\propto(\pi * \mathcal{N}(\bm{0},\lambda I))(\bm{x})
				\propto
				\int_{\mathbb{R}^d}
				\pi(\bm{u})
				\exp\!\left\{-\|\bm{u}-\bm{x}\|^2/(2\lambda)\right\}
				\, d\bm{u}.
			\end{equation}
		\end{itemize}
	\end{lemma}
	
	\begin{proof}
		Since $U$ is quadratic, the function
		\[
		\bm{u}\mapsto U(\bm{u}) + \frac{1}{2\lambda}\|\bm{u}-\bm{x}\|^2
		\]
		is strictly convex and admits a unique minimizer
		\begin{equation*}
			\bm{u}^*
			=
			\arg\min_{\bm{u}\in\mathbb{R}^d}
			\left\{
			U(\bm{u})
			+
			\frac{1}{2\lambda}\|\bm{u}-\bm{x}\|^2
			\right\}.
		\end{equation*}
		The first-order optimality condition yields
		\begin{equation*}
			\nabla U(\bm{u}^*)
			+
			\frac{1}{\lambda}(\bm{u}^*-\bm{x})
			=
			\bm{0}.
		\end{equation*}
		
		Consider the convolution integral in \eqref{eqn_conv}. Using
		$\pi(\bm{u})\propto \exp\{-U(\bm{u})\}$, we have
		\begin{align}
			\pi_t(\bm{x})
			&\propto
			\int_{\mathbb{R}^d}
			\exp\!\left\{
			-
			\left[
			U(\bm{u})
			+
			\frac{1}{2\lambda}\|\bm{u}-\bm{x}\|^2
			\right]
			\right\}
			\, d\bm{u}.
			\label{eqn_conv_expand}
		\end{align}
		
		Because $U$ is quadratic, the exponent is itself a quadratic function of
		$\bm{u}$ and can be written exactly as
		\[
		U(\bm{u})
		+
		\frac{1}{2\lambda}\|\bm{u}-\bm{x}\|^2
		=
		U(\bm{u}^*)
		+
		\frac{1}{2\lambda}\|\bm{u}^*-\bm{x}\|^2
		+
		(\bm{u}-\bm{u}^*)^\top
		(A+ \lambda^{-1} I)
		(\bm{u}-\bm{u}^*).
		\]
		
		Substituting it into \eqref{eqn_conv_expand} gives
		\begin{align*}
			\pi_t(\bm{x})
			&\propto
			\exp\!\left\{
			-
			\left[
			U(\bm{u}^*)
			+
			\frac{1}{2\lambda}\|\bm{u}^*-\bm{x}\|^2
			\right]
			\right\}
			\int_{\mathbb{R}^d}
			\exp\!\left\{
			-(\bm{u}-\bm{u}^*)^\top
			(A+ \lambda^{-1} I)
			(\bm{u}-\bm{u}^*)
			\right\}
			d\bm{u}.
		\end{align*}
		
		The remaining integral is Gaussian and depends only on
		$A$ and $\lambda$, hence it contributes a multiplicative constant that does
		not depend on $\bm{x}$. Therefore,
		\begin{equation*}
			\pi_t(\bm{x})
			\propto
			\exp\!\left\{
			-
			\inf_{\bm{u}\in\mathbb{R}^d}
			\left[
			U(\bm{u})
			+
			\frac{1}{2\lambda}\|\bm{u}-\bm{x}\|^2
			\right]
			\right\}.
		\end{equation*}
		
		By the definition of the Moreau--Yosida envelope,
		\[
		U^\lambda(\bm{x})
		=
		\inf_{\bm{u}\in\mathbb{R}^d}
		\left\{
		U(\bm{u})
		+
		\frac{1}{2\lambda}\|\bm{u}-\bm{x}\|^2
		\right\},
		\]
		which, together with \eqref{eqn_pi_lamda_def}, implies
		\[
		\pi_t(\bm{x})
		\propto
		\exp\{-U^\lambda(\bm{x})\}
		\propto
		\pi_t^\lambda(\bm{x}).
		\]
		This proves the claim.
	\end{proof}

	This lemma reveals a fundamental duality between variational and probabilistic smoothing techniques for Gaussian measures. The Moreau--Yosida regularization, originating from nonsmooth optimization, performs a pointwise supremum that yields a regularized potential $U^\lambda$. Gaussian convolution, a standard probabilistic operation, smooths the density by integrating with a Gaussian kernel. This result provides a bridge between optimization and diffusion processes, enabling a new perspective for nonsmooth sampling.
	
	\subsection{Generalized Moreau Approximation}
	
	In this subsection, we also derive the Moreau-Yosida-Gaussian equivalence in the general context of diffusion models. Consider the forward diffusion process:
	\begin{equation}
		\label{eq:ap_forward}
		d\bm{x}_t^{\rightarrow} = \left\{a(t)\bm{x}_t^{\rightarrow}\right\}dt + b(t)d\bm{B}_t ,\quad t\in \left[0,T\right],
	\end{equation}
	where the marginal 
	\begin{equation}\label{eqn_gen_pt}
		\pi_t(\bm{x}_t) \propto \int_{\mathbb{R}^d} \pi(\bm{x}_0) \exp\left\{-\frac{\|\mu(t)\bm{x}_0-\bm{x}_t\|^2}{2\sigma^2(t)}\right\}\, d\bm{x}_0,
	\end{equation}
	with
	\begin{equation}
		\nonumber
		\mu(t)=\exp\left\{\int_{0}^{t}a(s)ds\right\},\quad \sigma^2(t)=\int_{0}^{t} b^2(s)\exp\left\{2\int_{s}^{t}a(r)dr\right\} ds.
	\end{equation}
	The reverse process is also a diffusion process, running backwards in time and given by the reverse-time SDE:
	\begin{equation}
		\label{eq:ap_reverse}
		d\bm{x}_t^{\leftarrow} = \left\{a(t)\bm{x}_t^{\leftarrow}-b^2(t)\nabla \log \pi_t(\bm{x}_t^{\leftarrow})\right\}dt + b(t)d\bar{\bm{B}}_t .
	\end{equation}
	Define the time-varying regularization parameter
	\begin{equation}
		\lambda(t)=\sigma^2(t)/\mu^2(t) = \int_{0}^{t} b^2(s)/\mu^2(s) ds,
	\end{equation}
	the Moreau approximation is given by
	\begin{equation}
		\pi_t^{\lambda}(\bm{x}_t)\propto \exp\left\{ -U^{\lambda(t)}\left( \frac{\bm{x}_t}{\mu(t)} \right) \right\}.
	\end{equation}
	Using \eqref{eqn_score_mor}, the Moreau score is obtained by
	\begin{equation}
		\label{eq:ap_moreau_score}
		\begin{aligned}
			\nabla_{\bm{x}} \log \pi_t^\lambda(\bm{x}_t)=\frac{\mu(t)\text{Prox}_{U}^{\lambda(t)}(\frac{\bm{x}_t}{\mu(t)})-\bm{x}_t}{\sigma^2(t)}.
		\end{aligned}
	\end{equation}
	Compared with the original Moreau score in \eqref{eq:score_prox1}, the generalized expression in \eqref{eq:ap_moreau_score} involves a rescaling of the proximal input by $1/\mu(t)$  and substitutes the denominator with $\sigma^2(t)$.
	
	\begin{proposition}
		\normalfont
		\label{pro:prox_diff}
		Let $\pi\propto\exp\{-U(\bm x)\}$, where $U$ is a positive semidefinite quadratic function. Then the posterior score
		$\nabla_{\bm{x}} \log \pi_t(\bm{x}_t)$ based on \eqref{eqn_gen_pt}
		coincides with the Moreau score in \eqref{eq:ap_moreau_score}, namely,
		\[
		\nabla_{\bm{x}} \log \pi_t(\bm{x}_t)
		=
		\nabla_{\bm{x}} \log \pi_t^{\lambda}(\bm{x}_t).
		\]
	\end{proposition}
	
	\begin{proof}
		Define the scaled energy
		\[
		\tilde U_t(\bm y)
		=
		U\!\left(\frac{\bm y}{\mu(t)}\right).
		\]
		Let $\bm y = \mu(t)\bm x_0$. Then $d\bm y = \mu(t) d\bm x_0$, and the target density becomes
		\[
		\pi(\bm x_0)
		\propto
		\exp\{-U(\bm x_0)\}
		=
		\exp\!\left\{-U\!\left(\frac{\bm y}{\mu(t)}\right)\right\}
		=
		\exp\{-\tilde U_t(\bm y)\}.
		\]
		
		From the forward noising model \eqref{eqn_gen_pt}, the marginal is given by
		\begin{equation}
			\label{eq:ap_pixt}
			\begin{aligned}
				\pi_t(\bm x_t)\propto
				\int
				\exp\{-\tilde U_t(\bm y)\}
				\,
				\frac{1}{(2\pi\sigma^2(t))^{d/2}}
				\exp\!\left\{
				-\frac{\|\bm x_t-\bm y\|^2}{2\sigma^2(t)}
				\right\}
				\, d\bm y .
			\end{aligned}
		\end{equation}
		
		Hence $\pi_t$ is obtained by convolving $\exp\{-\tilde U_t\}$ with the Gaussian kernel
		$\mathcal N(0,\sigma^2(t)I)$.
		By Proposition~\ref{pro:equivalence}, such a Gaussian convolution is equivalent to the
		Moreau--Yosida regularization of $\tilde U_t$ with parameter $\sigma^2(t)$, namely,
		\[
		\pi_t(\bm x_t)
		\propto
		\exp\!\left\{-\tilde U_t^{\sigma^2(t)}(\bm x_t)\right\}.
		\]
		
		We now relate this envelope to that of the original function $U$.
		For any $\mu>0$ and $\lambda>0$, if $h(\bm y)=U(\bm y/\mu)$, then the scaling property of
		the Moreau envelope yields
		\[
		h^{\lambda}(\bm x)
		=
		U^{\lambda/\mu^2}\!\left(\frac{\bm x}{\mu}\right).
		\]
		Applying this identity with $h=\tilde U_t$, $\mu=\mu(t)$, and $\lambda=\sigma^2(t)$ gives
		\begin{equation}
			\label{eq:ap_ulat}
			\tilde U_t^{\sigma^2(t)}(\bm x_t)
			=
			U^{\sigma^2(t)/\mu(t)^2}\!\left(\frac{\bm x_t}{\mu(t)}\right)
			=
			U^{\lambda(t)}\!\left(\frac{\bm x_t}{\mu(t)}\right).
		\end{equation}
		
		Substituting \eqref{eq:ap_ulat} into the previous expression \eqref{eq:ap_pixt} for $\pi_t$ yields
		\[
		\pi_t(\bm x_t)
		\propto
		\exp\!\left\{
		-
		U^{\lambda(t)}\!\left(\frac{\bm x_t}{\mu(t)}\right)
		\right\}
		\propto
		\pi_t^{\lambda}(\bm x_t),
		\]
		which implies equality of their score functions.
	\end{proof}



\subsection{Asymptotic Moreau-Yosida-Gaussian Equivalence}

In this subsection, we provide some asymptotic error bounds between the true score and the Moreau score under non-quadratic settings.

With Tweedie's formula \cite{efron2011tweedie}, we have
\begin{equation}
	\label{eq:ap_tweedie}
	\nabla_{\bm{x}} \log \pi_t(\bm{x}_t) = \frac{\mu(t)\mathbb{E}\left[\bm{x}_0|\bm{x}_t,\bm{y}\right]-\bm{x}_t}{\sigma^2(t)},
\end{equation}
where $\mathbb{E}\left[\bm{x}_0|\bm{x}_t,\bm{y}\right]$ is the posterior mean. Recall that the Moreau score is given by
\begin{equation}
	\label{eq:ap_moreau}
	\nabla_{\bm{x}} \log \pi_t^{\lambda}(\bm{x}_t)=\frac{\mu(t)\text{Prox}_{U}^{\sigma^2(t)/\mu^2(t)}\left(\frac{\bm{x}_t}{\mu(t)}\right)-\bm{x}_t}{\sigma^2(t)}.
\end{equation}
Similar to \eqref{eqn_prox_U1}, we can see that the proximal operator 
\begin{equation}
	\nonumber		\text{Prox}_{U}^{\sigma^2(t)/\mu^2(t)}\left(\frac{\bm{x}_t}{\mu(t)}\right) =\mathop{\mathop{\arg\max}}_{\bm{x}_0\in\mathbb{R}^d}\exp \left\{-U(\bm{x}_0)-\frac{\mu^2(t)\|\frac{\bm{x}_t}{\mu(t)}-\bm{x}_0\|^2}{2\sigma^2(t)}\right\}  = \mathop{\mathop{\arg\max}}_{\bm{x}_0\in\mathbb{R}^d}\ p(\bm{x}_0|\bm{x}_t,\bm{y})
\end{equation}
is the posterior mode. This reveals the connection between the true score and the Moreau score, and the difference can be characterized by the deviation between the posterior mean $\mathbb{E}\left[\bm{x}_0|\bm{x}_t,\bm{y}\right]$ and mode $\text{Prox}_{U}^{\sigma^2(t)/\mu^2(t)}(\frac{\bm{x}_t}{\mu(t)})$.

\begin{lemma}[Deviation between mean and mode]
	\normalfont
	\label{lem:mean_mode}
	Let $\pi(\bm{x})\propto \exp\left\{-U(\bm{x})\right\}$ be a probability distribution on $\mathbb{R}^d$, with a proper, closed, convex function $U$. Denote the mean of $\pi(\bm{x})$ as 
	\begin{equation}
		\bm{\mu} = \mathbb{E}_{\bm{x}\sim \pi}\left[\bm{x}\right] = \int_{\mathbb{R}^d} \bm{x}\pi(\bm{x})d\bm{x},
	\end{equation}
	and the mode
	\begin{equation}
		\bm{u}^*=\mathop{\mathop{\arg\max}}_{\bm{u}\in\mathbb{R}^d} \left\{ \pi(\bm{u}) \right\}.
	\end{equation}
	Then the following hold.
	\begin{itemize}
		\item[1.] If $\pi(\bm{x})$ is symmetric unimodal, e.g., $U$ is symmetric convex, then we have
		\begin{equation}
			\nonumber
			\bm{\mu}=\bm{u}^*.
		\end{equation}
		\item [2.] \cite{stein1981estimation} If $U$ is $m$-strongly convex, then we have
		\begin{equation}
			\nonumber
			\|\bm{\mu}-\bm{u}^*\|\leq \sqrt{\frac{2d}{m}}.
		\end{equation}
	\end{itemize}
	
\end{lemma}

The following proposition generalizes Lemma~\ref{eq_lemma_c1} to non-quadratic potentials and bounds the error between the true score and the Moreau score, offering a general asymptotic analysis for the Moreau-Yosida-Gaussian equivalence and the proof of Proposition \ref{pro:score_error}.

\begin{proposition}[Asymptotic Moreau-Yosida-Gaussian equivalence]
	\normalfont
	\label{pro:asy_equ}
	
	Consider the diffusion process in \eqref{eq:ap_forward} with $\bm{x}_0 \sim \pi(\bm{x}) \propto \exp\{-U(\bm{x})\}$ and 
	\begin{equation}
		\nonumber
		\pi_t(\bm{x}_t) = \int_{\mathbb{R}^d} \pi(\bm{x}_0) p_{0t}(\bm{x}_t|\bm{x}_0)d\bm{x}_0,
	\end{equation}
	where $U = \beta f_{\bm{y}} + g$ with $f_{\bm{y}}$ being $m$-strongly convex and $g$ being convex. Let $\mu(t)$, $\sigma^2(t)$, and $\lambda(t)$ be as defined in Section~\ref{sec:method}. Then, for any $t>0$,
	\begin{equation}
		\|\nabla\log\pi_t(\bm{x}_t) - \nabla\log\pi_t^\lambda(\bm{x}_t)\| \leq \frac{\mu(t)}{\sigma(t)}\sqrt{\frac{2d}{\beta m\sigma^2(t) + \mu^2(t)}}.
	\end{equation}
	In particular, as $\beta \to \infty$, we have
	\begin{equation}
		\nonumber
		\|\nabla_{\bm{x}_t} \log \pi_t(\bm{x}_t)-\nabla_{\bm{x}_t} \log \pi_t^{\lambda}(\bm{x}_t)\|\to 0 .
	\end{equation}
\end{proposition}

\begin{proof}
	By Tweedie's formula, we have
	$$\nabla\log\pi_t(\bm{x}_t) = (\mu(t)\mathbb{E}[\bm{x}_0|\bm{x}_t,\bm{y}] - \bm{x}_t)/\sigma^2(t).$$
	Moreover, recall the Moreau score
	$$\nabla\log\pi_t^\lambda(\bm{x}_t) = (\mu(t)\mathrm{Prox}_{U}^{\lambda(t)}(\bm{x}_t/\mu(t)) - \bm{x}_t)/\sigma^2(t).$$ 
	Thus,
	\begin{equation}
		\label{eq:ap_score_error}
		\|\nabla\log\pi_t(\bm{x}_t) - \nabla\log\pi_t^\lambda(\bm{x}_t)\| = \frac{\mu(t)}{\sigma^2(t)} \|\mathbb{E}[\bm{x}_0|\bm{x}_t,\bm{y}] - \mathrm{Prox}_{U}^{\lambda(t)}(\bm{x}_t/\mu(t))\|.
	\end{equation}
	Note that $\text{Prox}_{U}^{\lambda(t)}(\bm{x}_t/\mu(t))$ is the mode of the posterior $p(\bm{x}_0|\bm{x}_t,\bm{y})$, from Bayes' rule, we get
	\begin{equation}
		\nonumber
		p(\bm{x}_0|\bm{x}_t,\bm{y}) \propto p(\bm{x}_t|\bm{x}_0,\bm{y}) p(\bm{x}_0|\bm{y}) \propto p(\bm{x}_t|\bm{x}_0) \pi(\bm{x}_0) \propto \exp\left\{-U(\bm{x}_0)-\frac{\|\bm{x}_t-\mu(t)\bm{x}_0\|^2}{2\sigma^2(t)}\right\}.
	\end{equation}
	Denote $c(\bm{x})=U(\bm{x})+\frac{\|\bm{x}_t-\mu(t)\bm{x}\|^2}{2\sigma^2(t)}$. Then $c(\bm{x})$ is $(\beta m+\frac{\mu^2(t)}{\sigma^2(t)})$-strongly convex. Denote 
	\begin{equation}
		\nonumber
		\bm{\mu}=\mathbb{E}\left[\bm{x}_0|\bm{x}_t,\bm{y}\right], \quad \bm{u}^* = \mathop{\mathop{\arg\max}}_{\bm{x}_0\in\mathbb{R}^d} p(\bm{x}_0|\bm{x}_t,\bm{y}).
	\end{equation}
	Then from Lemma \ref{lem:mean_mode}, we have
	\begin{equation}
		\nonumber
		\|\bm{\mu}-\bm{u}^*\|\leq \sqrt{\frac{2d}{\beta m+\frac{\mu^2(t)}{\sigma^2(t)}}}=\sigma(t)\sqrt{\frac{2d}{\beta m\sigma^2(t)+\mu^2(t)}}.
	\end{equation}
	Substituting the above into the score error \eqref{eq:ap_score_error}, we have
	\begin{equation}
		\nonumber
		\|\nabla_{\bm{x}_t} \log \pi_t(\bm{x}_t)-\nabla_{\bm{x}_t} \log \pi_t^{\lambda}(\bm{x}_t)\| = \frac{\mu(t)}{\sigma^2(t)}\|\bm{\mu}-\bm{u}^*\|\leq \frac{\mu(t)}{\sigma(t)}\sqrt{\frac{2d}{\beta m\sigma^2(t)+\mu^2(t)}},
	\end{equation}
	or equivalently
	\begin{equation}
		\nonumber
		\|\nabla_{\bm{x}_t} \log \pi_t(\bm{x}_t)-\nabla_{\bm{x}_t} \log \pi_t^{\lambda}(\bm{x}_t)\|\leq \frac{1}{\mu(t)}\sqrt{\frac{2d}{\beta m\lambda^2(t)+\lambda(t)}}.
	\end{equation}
	Further, as $\beta\to \infty$, we have the error bound tends to $0$ with rate $\mathcal{O}(\beta^{-1/2})$.
	
\end{proof}

This proposition demonstrates that the Moreau--Yosida approximation serves as a provably accurate surrogate for the true score in diffusion processes when the potential \(U\) contains a strongly convex likelihood term. The error bound quantifies how the discrepancy depends on the diffusion parameters \(\mu(t)\), \(\sigma(t)\), the strong convexity constant \(m\), and the likelihood weight \(\beta\). Crucially, as \(\beta \to \infty\), the bound tends to zero, revealing an asymptotic equivalence between the traditional score and the Moreau score.

\section{Moreau Score Matching}
\label{ap:prox_match}

\subsection{Learning of the Proximal Operator}

If we have realizations $\left\{\left(\bm{x}^{(i)},\text{Prox}_{g}^{\lambda(t)}\left(\bm{x}^{(i)}\right)\right)\right\}$, then we can use a network to learn the proximal mapping through data. Unfortunately, paired ground-truth do not exist in common settings, making supervised training infeasible. Instead, we seek to train a proximal network using only i.i.d. samples from the unknown data distribution $\pi(\bm{x})\propto \exp\left\{-g(\bm{x})\right\}$ in an unsupervised way \cite{fang2023s,fang2026beyond}.

The key insight comes from the observation that the proximal operator is the maximum a posteriori (MAP) denoiser for additive Gaussian noise. For $\bm{x}_0 \sim p_0(\bm{x})\propto \exp\left\{-g(\bm{x}_0)\right\}$ and $\bm{x}_t = \bm{x}_0 + \sqrt{\lambda} \bm{\xi}_t$ with $\bm{\xi}_t \sim \mathcal{N}(\bm{0}, I)$, we have $\text{Prox}_g^{\lambda}(\bm{x}_t) = \mathop{\arg\max}_{\bm{x}_0} p(\bm{x}_0|\bm{x}_t)$. Consider a loss function of the form $\mathbb{E}_{\bm{x}_0,\bm{x}_t}\left[\text{dist}(\bm{x}_0,\bm{\phi}(\bm{x}_t))\right]$. If we use squared $\ell_2$ distance,
minimizing this loss will lead to the minimum mean square error (MMSE) estimator given by the mean of the posterior $p(\bm{x}_0|\bm{x}_t)$. However, our goal is to learn the mode of the posterior. 

The core idea is to design a distance metric \(\text{dist}_{\zeta}\) that, in the limit \(\zeta \to 0\), it recovers the posterior mode (which equals the proximal operator) as the minimizer of the expected loss. A key observation is that as \(\zeta \to 0\), the Gaussian kernel \(\mathcal{N}(\cdot; \bm{0}, \zeta^2 I)\) behaves like a Dirac delta function, which can be designed to forces \(\bm{\phi}\) to output the posterior mode for each \(\bm{x}_t\). In the following proposition, we will extend the proximal matching technique in \cite{fang2023s} to general diffusion processes, thereby obtaining a training loss to learn the proximal operator.

\begin{proposition}
	\normalfont
	Suppose that $U(\bm{x})=\beta f_{\bm{y}}(\bm{x})+g(\bm{x})$, where $f: \mathbb{R}^d \to \mathbb{R}$ is $L$-smooth and $m$-strongly convex, $g: \mathbb{R}^d \to \mathbb{R}$ is convex and $\text{dom}(g) = \mathcal{X}$ is compact, and $\beta > 0$. Denote $\bm{x}_0 \sim p_0(\bm{x})\propto \exp\left\{-g(\bm{x}_0)\right\}$ and $\bm{x}_t = \bm{x}_0 + \sqrt{\lambda(t)} \bm{\xi}_t$ where $\bm{\xi}_t \sim \mathcal{N}(\bm{0}, I)$. Let $\text{dist}_{\zeta}(\bm{x},\bm{x}^\prime) : \mathbb{R}^d\times\mathbb{R}^d \to \mathbb{R}$ be defined as
	\begin{equation}
		\text{dist}_{\zeta}(\bm{x},\bm{x}^\prime) = 1-\mathcal{N}(\bm{x}-\bm{x}^\prime;\bm{0},\zeta^2I), \quad \zeta > 0,
	\end{equation}
	where $\mathcal{N}(\bm{x}-\bm{x}^\prime;\bm{0},\zeta^2I) = \frac{1}{(2\pi\zeta^2)^{d/2}}\exp\left\{-\frac{\|\bm{x}-\bm{x}^\prime\|_2^2}{2\zeta^2}\right\}$. Consider the optimization problem
	\begin{equation}
		\bm{\theta}^* = \mathop{\mathop{\arg\min}}_{\bm{\theta}} \lim_{\zeta \to 0} \mathbb{E}_{t\sim \mathcal{U}[0,T]}\mathbb{E}_{\bm{x}_0\sim p_0}\mathbb{E}_{\bm{x}_t|\bm{x}_0\sim \mathcal{N}(\bm{0}, \lambda(t)I)} \left[\text{dist}_{\zeta}\left(\bm{\phi}_{\bm{\theta}}(\bm{x}_t,\lambda(t)),\bm{x}_0\right)\right].
	\end{equation}
	Then, almost surely (i.e., for almost every $\bm{x}_t,~t$),
	\begin{equation}
		\bm{\phi}_{\bm{\theta}^*}(\bm{x}_t,\lambda(t)) =\mathop{\mathop{\arg\max}}_{\bm x_0\in\mathbb{R}^d}\ p(\bm x_0|\bm x_t)= \text{Prox}_{g}^{\lambda(t)}(\bm{x}_t).
	\end{equation}
\end{proposition}

\begin{proof}
	From the definition of $\text{dist}_{\zeta}(\bm{x},\bm{x}^\prime)$, we have
	\begin{equation}
		\nonumber
		\lim_{\zeta \to 0} \mathbb{E}_{t,\bm{x}_0, \bm{x}_t} \left[ \text{dist}_{\zeta}(\bm{\phi}_{\bm{\theta}}(\bm{x}_t,\lambda(t)),\bm{x}_0) \right] = 1 - \lim_{\zeta \to 0} \mathbb{E}_{t,\bm{x}_0, \bm{x}_t} \left[ \mathcal{N}(\bm{\phi}_{\bm{\theta}}(\bm{x}_t,\lambda(t))-\bm{x}_0;\bm{0},\zeta^2I) \right].
	\end{equation}
	Using the law of the total expectation, we have
	\begin{equation}
		\nonumber
		\mathbb{E}_{t,\bm{x}_0, \bm{x}_t} \left[ \mathcal{N}(\bm{\phi}_{\bm{\theta}}(\bm{x}_t,\lambda(t))-\bm{x}_0;\bm{0},\zeta^2I) \right] = \mathbb{E}_{t}\mathbb{E}_{\bm{x}_t|t}\mathbb{E}_{\bm{x}_0|\bm{x}_t,t} \left[ \mathcal{N}(\bm{\phi}_{\bm{\theta}}(\bm{x}_t,\lambda(t))-\bm{x}_0;\bm{0},\zeta^2I) \right].
	\end{equation}
	We now analyze the inner conditional expectation. For fixed $\bm{x}_t,t$, define the function
	\begin{equation}
		\nonumber
		F_\zeta(\bm{c}) = \mathbb{E}_{\bm{x}_0|\bm{x}_t,t} \left[ \mathcal{N}(\bm{c}-\bm{x}_0;\bm{0},\zeta^2I) \right] = \int_{\mathbb{R}^d} \mathcal{N}(\bm{c}-\bm{x}_0;\bm{0},\zeta^2I) p(\bm{x}_0|\bm{x}_t,t) \, d\bm{x}_0.
	\end{equation}
	Since $p(\bm{x}_0|\bm{x}_t,t) \propto p_{0}(\bm{x}_0) \exp\left\{-\|\bm{x}_t - \bm{x}_0\|^2/ (2\lambda(t)) \right\}$ and $p_{0}$ is bounded, the product is integrable. Moreover, as $\zeta\to 0$, the Gaussian kernel $\mathcal{N}(\bm{c}-\bm{x}_0;\bm{0},\zeta^2I)$ tends to Dirac $\delta_{\bm{c}}(\bm{x}_0)$. Then for almost every $\bm{x}_t$ and for each $\bm{c}$, there holds
	\begin{equation}
		\nonumber
		\lim_{\zeta \to 0} F_\zeta(\bm{c}) = p(\bm{c}|\bm{x}_t,t).
	\end{equation}
	In particular, for $\bm{c} = \bm{\phi}_{\bm{\theta}}(\bm{x}_t,\lambda(t))$, we have
	\begin{equation}
		\nonumber
		\lim_{\zeta \to 0} \mathbb{E}_{\bm{x}_0|\bm{x}_t,t} \left[ \mathcal{N}(\bm{\phi}_{\bm{\theta}}(\bm{x}_t,\lambda(t))-\bm{x}_0;\bm{0},\zeta^2I) \right] = p(\bm{\phi}_{\bm{\theta}}(\bm{x}_t,\lambda(t))|\bm{x}_t,t).
	\end{equation}
	Since $p_0$ is supported on a compact set $\mathcal{X}$, it follows that $\mathcal{N}(\bm{x}-\bm{x}^\prime;\bm{0},\zeta^2I)$ is bounded by its maximum at $\bm{0}$, and $p(\bm{x}_0|\bm{x}_t,t)$ is bounded. Hence, the integrand is uniformly bounded by an integrable function. By the dominated convergence theorem, we have
	\begin{equation}
		\nonumber
		\lim_{\zeta \to 0} \mathbb{E}_{\bm{x}_t,t}\mathbb{E}_{\bm{x}_0|\bm{x}_t,t} \left[ \mathcal{N}(\bm{\phi}_{\bm{\theta}}(\bm{x}_t,\lambda(t))-\bm{x}_0;\bm{0},\zeta^2I) \right]=\mathbb{E}_{\bm{x}_t,t} \left[\lim_{\zeta \to 0} \mathbb{E}_{\bm{x}_0|\bm{x}_t,t} \left[ \mathcal{N}(\bm{\phi}_{\bm{\theta}}(\bm{x}_t,\lambda(t))-\bm{x}_0;\bm{0},\zeta^2I) \right]\right].
	\end{equation}
	Combining the above results, we obtain
	\begin{equation}
		\nonumber
		\lim_{\zeta \to 0} \mathbb{E}_{t,\bm{x}_0, \bm{x}_t} \left[ \text{dist}_{\zeta}(\bm{\phi}_{\bm{\theta}}(\bm{x}_t,\lambda(t)),\bm{x}_0) \right] = 1 -  \mathbb{E}_{t,\bm{x}_t} \left[p(\bm{\phi}_{\bm{\theta}}(\bm{x}_t,\lambda(t))|\bm{x}_t,t) \right].
	\end{equation}
	Minimizing the left-hand side over measurable functions $\bm{\phi}_{\bm{\theta}}$ is equivalent to maximizing $\mathbb{E}_{t,\bm{x}_t} \left[p(\bm{\phi}_{\bm{\theta}}(\bm{x}_t,\lambda(t))|\bm{x}_t,t) \right]$. Since the expectation is over $\bm{x}_t,t$ and $\bm{\phi}_{\bm{\theta}}(\bm{x}_t,\lambda(t))$ can be chosen pointwise, the optimum is achieved by choosing, for each $\bm{x}_t,t$, $\bm{\phi}_{\bm{\theta}}(\bm{x}_t,\lambda(t))$ to maximize $p(\bm{\phi}_{\bm{\theta}}(\bm{x}_t,\lambda(t))|\bm{x}_t,t)$. Therefore,
	\begin{equation}
		\nonumber
		\bm{\phi}_{\bm{\theta}^*}(\bm{x}_t,\lambda(t)) = \mathop{\mathop{\arg\max}}_{\bm{x}_0\in\mathbb{R}^d} p(\bm{x}_0|\bm{x}_t,t).
	\end{equation}
	Recall that in Section \ref{sec:method}, we give the equivalence between the posterior mode and proximal operator, i.e., $\mathop{\mathop{\arg\max}}_{\bm{x}_0\in\mathbb{R}^d} p(\bm{x}_0|\bm{x}_t,t) =\text{Prox}_{g}^{\lambda(t)}(\bm{x}_t)$. Then we obtain that
	\begin{equation}
		\nonumber
		\bm{\phi}_{\bm{\theta}^*}(\bm{x}_t,\lambda(t)) =\text{Prox}_{g}^{\lambda(t)}(\bm{x}_t).
	\end{equation}
	This completes the proof.
\end{proof}

This proposition provides a principled training objective for learning proximal operators in an unsupervised manner, offering a theoretical foundation for using neural networks to approximate proximal operators in iterative denoising algorithms.

\subsection{Score Estimation Error}
\label{ap:score}

In this subsection, we provide a detailed characterization for the score estimation error, which will be used in the following convergence analysis. The score estimation error is mainly composed of two components: proximal approximation error and deviation between posterior mean and mode.

Note that in practice, obtaining realizations from the composite distribution is intractable. Here we assume that only realizations from $\exp\left\{-g(\bm{x})\right\}$ can be obtained. As a consequence, we use the proximal splitting:
\begin{equation}
	\text{Prox}_{U}^{\lambda(t)}(\bm{x})\approx \text{Prox}_{g}^{\lambda(t)}(\bm{x}-\lambda(t)\beta \nabla_{\bm{x}} f(\bm{x})).
\end{equation}
Next, we give the detailed error bound for this first-order approximation.

\begin{proposition}[Proximal splitting]
	\normalfont
	\label{pro:prox_appr} Suppose Assumption~\ref{assumption1} holds. For any $\bm{x}\in\mathcal{X}$, define 
	\begin{equation}
		\begin{aligned}
			&\bm{u}^*(\bm{x}) = \operatorname{Prox}_{\beta f + g}^{\lambda}(\bm{x}) = \mathop{\mathop{\arg\min}}_{\bm{u} \in \mathbb{R}^d} \left\{\beta f(\bm{u}) + g(\bm{u}) + \frac{1}{2\lambda}\|\bm{u} - \bm{x}\|^2\right\},\\
			&\hat{\bm{u}}^*(\bm{x}) = \operatorname{Prox}_{g}^{\lambda}(\bm{x} - \lambda\beta\nabla f(\bm{x})) = \mathop{\mathop{\arg\min}}_{\bm{u} \in \mathbb{R}^d} \left\{g(\bm{u}) + \frac{1}{2\lambda}\|\bm{u} - (\bm{x} - \lambda\beta\nabla f(\bm{x}))\|^2 \right\}.
		\end{aligned}
	\end{equation}
	Then, for $\beta\geq2$ and sufficiently small $\lambda$, we have
	\begin{equation}
		\|\bm{u}^*(\bm{x}) - \hat{\bm{u}}^*(\bm{x})\| \leq \frac{4}{\beta L}\left(\|\nabla f(\bm{x})\| + 1\right).
	\end{equation}
	In particular, as $\beta \to \infty$, the error bound tends to 0 at a rate of $\mathcal{O}(\beta^{-1})$.
\end{proposition}

\begin{proof}
	By the first-order optimality condition for $\bm{u}^*$, there exists
	$\bm{\eta}\in\partial g(\bm{u}^*)$ such that
	\begin{equation*}
		\beta\nabla f(\bm{u}^*) + \bm{\eta}
		+ \frac{1}{\lambda}(\bm{u}^*-\bm{x})
		= \bm{0},
	\end{equation*}
	which implies
	\begin{equation}
		\label{eq:u_star_relation}
		\bm{u}^*-\bm{x}
		=
		-\lambda\beta\nabla f(\bm{u}^*)
		-\lambda\bm{\eta}.
	\end{equation}
	Equivalently,
	$\bm{x}-\lambda\beta\nabla f(\bm{u}^*)-\bm{u}^*
	\in\lambda\partial g(\bm{u}^*)$, and by the characterization of proximal maps in \eqref{eqn_prox_uf}, we get
	\begin{equation*}
		\bm{u}^*
		=
		\operatorname{Prox}_{g}^{\lambda}
		(\bm{x}-\lambda\beta\nabla f(\bm{u}^*)).
	\end{equation*}
	On the other hand, by definition,
	\begin{equation*}
		\hat{\bm{u}}^* = \operatorname{Prox}_{g}^{\lambda}(\bm{x} - \lambda\beta \nabla f(\bm{x})).
	\end{equation*}
	Since the proximal operator $\operatorname{Prox}_{g}^{\lambda}$ is nonexpansive, we have
	\begin{equation}
		\label{eq:diff_bound}
		\|\bm{u}^* - \hat{\bm{u}}^*\| \leq \| (\bm{x} - \lambda\beta \nabla f(\bm{u}^*)) - (\bm{x} - \lambda\beta \nabla f(\bm{x})) \| = \lambda\beta \|\nabla f(\bm{u}^*) - \nabla f(\bm{x})\|. 
	\end{equation}
	Using the $L$-smoothness of $f$, we obtain
	\begin{equation}
		\label{eq:bound_by_distance}
		\|\bm{u}^* - \hat{\bm{u}}^*\| \leq \lambda\beta L \|\bm{u}^* - \bm{x}\|.
	\end{equation}
	We now bound $\|\bm{u}^*-\bm{x}\|$. Taking norms in
	\eqref{eq:u_star_relation} yields
	\begin{equation}
		\label{eq:u_x_norm}
		\|\bm{u}^* - \bm{x}\| \leq \lambda\beta \|\nabla f(\bm{u}^*)\| + \lambda \|\bm{\eta}\|.
	\end{equation}
	Since $g: \mathbb{R}^d \to \mathbb{R}$ is convex with a compact domain $\operatorname{dom} g = \mathcal{X}$, there exists a bounded constant $L_g$ such that $g$ is $L_g$-Lipschitz on $\mathcal{X}$. For sufficiently large $\beta\geq\max\left\{2,L_g\right\}$, we have $\|\bm{\eta}\| \leq \beta$. Moreover, by the $L$-smoothness of $f$, we obtain
	\begin{equation}\label{eq:grad_u_star}
		\|\nabla f(\bm{u}^*)\| \leq \|\nabla f(\bm{x})\| + L \|\bm{u}^* - \bm{x}\|.
	\end{equation}
	Combining \eqref{eq:u_x_norm} and \eqref{eq:grad_u_star}, we get
	\begin{equation*}
		\|\bm{u}^* - \bm{x}\| \leq \lambda\beta \left( \|\nabla f(\bm{x})\| + L \|\bm{u}^* - \bm{x}\| \right) + \lambda \beta.
	\end{equation*}
	Rearranging gives
	\begin{equation*}
		(1-\lambda\beta L)\|\bm{u}^*-\bm{x}\|
		\leq
		\lambda\beta
		\bigl(\|\nabla f(\bm{x})\|+1\bigr).
	\end{equation*}
	For $\lambda\beta L<1$, we get
	\begin{equation}
		\label{eq:bound_u_x}
		\|\bm{u}^*-\bm{x}\|
		\leq
		\frac{\lambda\beta}{1-\lambda\beta L}
		\bigl(\|\nabla f(\bm{x})\|+1\bigr).
	\end{equation}
	
	Substituting \eqref{eq:bound_u_x} into \eqref{eq:bound_by_distance} gives
	\begin{equation*}
		\|\bm{u}^*-\hat{\bm{u}}^*\|
		\leq
		\frac{\lambda^2\beta^2L}{1-\lambda\beta L}
		\bigl(\|\nabla f(\bm{x})\|+1\bigr).
	\end{equation*}
	Choosing $\lambda\leq (2\beta^{3/2}L)^{-1}$ ensures
	$1-\lambda\beta L\geq 1/2$, from which it follows that
	\[
	\|\bm{u}^*-\hat{\bm{u}}^*\|
	\leq
	\frac{4}{\beta L}
	\bigl(\|\nabla f(\bm{x})\|+1\bigr).
	\]
	The stated $\mathcal{O}(\beta^{-1})$ rate is immediate.
\end{proof}

Then, we can combine the above results and obtain a total score estimation error bound, which corresponds to the results in Proposition \ref{pro:score_estimate}. The proof employs a two-step error decomposition: first, between the true score and the Moreau score \(\nabla \log \pi_t^\lambda(\bm{x}_t)\); second, between the Moreau score and the network approximation \(\bm{v}_t\). The first term is bounded using Proposition~\ref{pro:asy_equ}, which exploits the strong convexity of the posterior to control the distance between the posterior mean and mode. The second term is bounded by the proximal splitting approximation in Proposition~\ref{pro:prox_appr}, which leverages the composite structure \(U = \beta f + g\) to approximate the full proximal operator. 

\begin{theorem}[Score estimation error]
	\normalfont
	\label{thm:score_est}
	Under Assumption \ref{assumption1}, consider a given distribution $\pi(\bm{x})\propto \exp\left\{-U(\bm{x})\right\}\propto \exp\left\{-(\beta f_{\bm{y}}+g)(\bm{x})\right\}$ with $\bm x$ and $\bm y$ satisfying the inverse model \eqref{eqn_inverse}. Let $\bm{v}_t$ be the approximated score using a proximal network $\bm{\phi}_{\bm{\theta}^*}(\bm{x}_t,\lambda(t)) = \text{Prox}_{g}^{\lambda(t)}(\bm{x}_t)$, i.e.,
	\begin{equation}
		\bm{v}_t = \frac{\mu(t)\bm{\phi}_{\bm{\theta}^*}(\frac{\bm{x}_t}{\mu(t)}-\lambda(t)\beta \nabla_{\bm{x}} f(\frac{\bm{x}_t}{\mu(t)}),\lambda(t))-\bm{x}_t}{\mu^2(t)\lambda(t)}.
	\end{equation}
	Then, for any $t\in(0,T]$, we have
	\begin{equation}
		\|\nabla \log \pi_t(\bm{x}_t)-\bm{v}_t\|\leq  \frac{1}{\mu(t)}\left(\sqrt{\frac{2d}{\beta m\lambda^2(t)+\lambda(t)}}+\frac{4}{\beta L}\left(\|\nabla f(\bm{x})\| + 1\right)\right).
	\end{equation}
	Further, there exists $M =\mathcal{O}(\beta^{-1/2})$ such that
	\begin{equation}
		\|\nabla\log\pi_t(\bm{x}_t) - \bm{v}_t\| \leq \frac{M(1 + \|\bm{x}_t\|)}{\sigma^2(t)}.
	\end{equation}
	
\end{theorem}

\begin{proof}
	Decompose the total score estimation error as
	\begin{align}
		\label{eq:ap_total}
		\|\nabla\log\pi_t(\bm{x}_t) - \bm{v}_t\| &\leq \underbrace{\|\nabla\log\pi_t(\bm{x}_t) - \nabla\log\pi_t^\lambda(\bm{x}_t)\|}_{I_1} + \underbrace{\|\nabla\log\pi_t^\lambda(\bm{x}_t) - \bm{v}_t\|}_{I_2}.
	\end{align}
	From Proposition~\ref{pro:asy_equ}, we have
	\begin{equation}
		\label{eq:ap_I1}
		I_1 \leq \frac{\mu(t)}{\sigma(t)}\sqrt{\frac{2d}{\beta m\sigma^2(t) + \mu^2(t)}}.
	\end{equation}
	For the error $I_2$, note that $\nabla\log\pi_t^\lambda(\bm{x}_t) = (\mu(t)\mathrm{Prox}_{U}^{\lambda(t)}(\bm{x}_t/\mu(t)) - \bm{x}_t)/\sigma^2(t)$.
	Using Proposition~\ref{pro:prox_appr} with the proximal splitting approximation gives
	\begin{equation}
		\label{eq:ap_I2}
		I_2 \leq \frac{\mu(t)}{\sigma^2(t)} \cdot \frac{4}{\beta L}\left(\|\nabla f_{\bm{y}}(\bm{x}_t/\mu(t))\| + 1\right).
	\end{equation}
	Substituting \eqref{eq:ap_I1} and \eqref{eq:ap_I2} into \eqref{eq:ap_total}, we obtain the score estimation error
	\begin{equation}
		\nonumber
		\|\nabla_{\bm{x}_t} \log \pi_t(\bm{x}_t)-\bm{v}_t\|\leq  \frac{1}{\mu(t)}\left(\sqrt{\frac{2d}{\beta m\lambda^2(t)+\lambda(t)}}+\frac{4}{\beta L \lambda(t)}\left(\|\nabla f_{\bm{y}}(\bm{x}_t)\| + 1\right) \right).
	\end{equation}
	Further, we reformulate the above score estimation error as
	\begin{equation}
		\nonumber
		\|\nabla_{\bm{x}_t} \log \pi_t(\bm{x}_t)-\bm{v}_t\|\leq  \frac{1}{\sqrt{\beta}}\frac{1}{\mu(t)\lambda(t)}\left(\sqrt{\frac{2d}{m+\frac{1}{\beta\lambda(t)}}}+\frac{4}{\sqrt{\beta} L}\left(\|\nabla f_{\bm{y}}(\bm{x}_t)\| + 1\right) \right).
	\end{equation}
	From the $L$-smoothness of $f_{\bm{y}}$, we have
	\begin{equation}
		\nonumber
		\|\nabla f_{\bm{y}}(\bm{x}_t)\|\leq L\|\bm{x}_t\| + \|\nabla f_{\bm{y}}(\bm{0})\|.
	\end{equation}
	Suppose that $\beta\geq2$, then we have
	\begin{equation}
		\nonumber
		\sqrt{\frac{2d}{m+\frac{1}{\beta\lambda(t)}}}+\frac{4}{\sqrt{\beta} L}\left(\|\nabla f_{\bm{y}}(\bm{x}_t)\| + 1\right)\leq \sqrt{\frac{2d}{m}} + 2\sqrt{2}\left(\|\bm{x}_t\| + \frac{\|\nabla f_{\bm{y}}(\bm{0})\|}{L}\right).
	\end{equation}
	Denote $C_{f} = \max\left\{2\sqrt{2}, \sqrt{\frac{2d}{m}}+\frac{2\sqrt{2}\|\nabla f_{\bm{y}}(\bm{0})\|}{L}\right\}$ and $M = C_{f}/\sqrt{\beta}$. Then we have
	\begin{equation}
		\nonumber
		\|\nabla_{\bm{x}_t} \log \pi_t(\bm{x}_t)-\bm{v}_t\|\leq  \frac{C_{f}}{\sqrt{\beta}}\frac{\mu(t)}{\mu^2(t)\lambda(t)}\left(1+\|\bm{x}_t\|\right)\leq  \frac{M\left(1+\|\bm{x}_t\|\right)}{\sigma^2(t)},
	\end{equation}
	where $M = \mathcal{O}(\beta^{-1/2})$.
	
\end{proof}

Theorem \ref{thm:score_est} provides a rigorous non-asymptotic error bound for the Moreau score approximation. It demonstrates that the approximation \(\bm{v}_t\)—which combines a gradient step on \(f\) with a learned proximal operator for \(g\)—can approximate the true score with an error that decays as \(\mathcal{O}(\beta^{-1/2})\). This justifies the use of Moreau score for proximal diffusion sampling in Bayesian inverse problems.

\section{Discretization Schedule}
\label{ap:discret}

To implement the reverse-time SDE \eqref{eq:reverse} in practice, we require a numerical discretization scheme. In this section, we provide details of the exponential interpolation method used for proximal diffusion sampling.

\subsection{Euler-Maruyama Discretization}

Consider the following reverse process
\begin{equation}
	\label{eq:ap_rev}
	d\bm{x}_t^{\leftarrow} = \left\{a(t)\bm{x}_t^{\leftarrow}-b^2(t)\nabla \log \pi_t(\bm{x}_t^{\leftarrow})\right\}dt + b(t)d\bar{\bm{B}}_t ,
\end{equation}
where $d\bar{\bm{B}}_t$ is a reverse-time Wiener process and $\nabla \log \pi_t(\bm{x}_t^{\leftarrow})$ is the score function. To generate samples from the reverse dynamics, we use numerical integration for discretization. 

Consider a sequence of stepsizes $\left\{\gamma_k\right\}_{k=0,1,\ldots,K}$, and denote $t_0=0, t_{k+1}=\sum_{j=0}^{k} \gamma_j$. Then the interpolation process of Euler-Maruyama discretization on the interval $t\in\left[t_k,t_{k+1}\right]$ is defined by
\begin{equation}
	\nonumber
	d\bar{\bm{x}}_{t}^{EM} = \left\{a(t_k)\bar{\bm{x}}_{t_k}^{EM}-b^2(t_k)\nabla \log \pi_{t_k}(\bar{\bm{x}}_{t_k}^{EM})\right\}dt + b(t_k)d\bar{\bm{B}}_t .
\end{equation}
The corresponding iteration form is given by
\begin{equation}
	\nonumber
	\bar{\bm{x}}_{t_{k+1}}^{EM} = \bar{\bm{x}}_{t_{k}}^{EM}+\gamma_k \left\{a(t_k)\bar{\bm{x}}_{t_k}^{EM}-b^2(t_k)\nabla\log \pi_{t_k}(\bar{\bm{x}}_{t_k}^{EM})\right\} + b(t_k) \sqrt{\gamma_k} \bm{z}, \quad \bm{z}\sim\mathcal{N}(\bm{0,I}).
\end{equation}
After using the Moreau score in \eqref{eq:ap_moreau_score} to replace the traditional score in \eqref{eq:ap_rev}, the reverse process becomes
\begin{equation}
	\label{eq:ap_rp}
	\begin{aligned}
		d\bm{x}_t^{\leftarrow}=& \left\{a(t)\bm{x}_t^{\leftarrow}-b^2(t)\nabla \log \pi_t(\bm{x}_t^{\leftarrow})\right\}dt + b(t)d\bar{\bm{B}}_t \\
		=& \left\{\left[a(t)+\frac{b^2(t)}{\sigma^2(t)}\right]\bm{x}_t^{\leftarrow}-\frac{b^2(t)\mu(t)}{\sigma^2(t)}\text{Prox}_{U}^{\sigma^2(t)/\mu^2(t)}\left(\frac{\bm{x}_t^{\leftarrow}}{\mu(t)}\right)\right\}dt + b(t)d\bar{\bm{B}}_t.
	\end{aligned}
\end{equation}
This leads to the following interpolation process
\begin{equation}
	\nonumber
	d\bar{\bm{y}}_t^{EM} = \left\{-\left[a(T-t_k)+\frac{b^2(T-t_k)}{\sigma^2(T-t_k)}\right]\bar{\bm{y}}_{t_k}^{EM}+\frac{b^2(T-t_k)\mu(T-t_k)}{\sigma^2(T-t_k)}\text{Prox}_{U}^{\lambda(T-t_k)}\left(\frac{\bar{\bm{y}}_{t_k}^{EM}}{\mu(T-t_k)}\right)\right\}dt + b(T-t_k)d\bm{B}_t.
\end{equation}

And the corresponding iteration form is given by:
\begin{equation}
	\nonumber
	\begin{aligned}
		\bar{\bm{y}}_{t_{k+1}}^{EM} &= \bar{\bm{y}}_{t_{k}}^{EM} - \gamma_k  \left\{a(t_{k})\bar{\bm{y}}_{t_{k}}^{EM}-\frac{b^2(t_{k})\mu(t_{k})}{\sigma^2(t_{k})}\left[\text{Prox}_{U}^{\sigma^2(t_{k})/\mu^2(t_k)}\left(\frac{\bar{\bm{y}}_{t_{k}}^{EM}}{\mu(t_{k})}\right)-\frac{\bar{\bm{y}}_{t_{k}}^{EM}}{\mu(t_{k})}\right]\right\} + \sqrt{\gamma_k }b(t_{k})\bm{z}, \\
		&= \left(1-\delta_k a(t_{k})-\gamma_k\frac{b^2(t_{k})}{\sigma^2(t_{k})}\right)\bar{\bm{y}}_{t_{k}}^{EM}+\gamma_k\frac{b^2(t_{k})\mu(t_{k})}{\sigma^2(t_{k})}\text{Prox}_{U}^{\sigma^2(t_{k})/\mu^2(t_k)}\left(\frac{\bar{\bm{y}}_{t_{k}}^{EM}}{\mu(t_{k})}\right)+ \sqrt{\gamma_k }b(t_{k})\bm{z}.
	\end{aligned}
\end{equation}

Similarly, for the probability flow ODE, the iteration can be formulated as
\begin{equation}
	\nonumber
	\bar{\bm{y}}_{t_{k+1}}^{EM} = \left(1-\delta_k a(t_{k})-\gamma_k\frac{b^2(t_{k})}{2\sigma^2(t_{k})}\right)\bar{\bm{y}}_{t_{k}}^{EM}+\gamma_k\frac{b^2(t_{k})\mu(t_{k})}{2\sigma^2(t_{k})}\text{Prox}_{U}^{\sigma^2(t_{k})/\mu^2(t_k)}\left(\frac{\bar{\bm{y}}_{t_{k}}^{EM}}{\mu(t_{k})}\right).
\end{equation}

\subsection{Exponential Interpolation}

As Euler-Maruyama discretization approximates the drift and diffusion coefficients as constant in the SDEs, it suffers from a large discretization error and has poor numerical stability. To obtain a better convergence guarantee, here we consider the exponential interpolation \cite{zhang2022fast}.

A key observation is that the only nonlinear term with respect to $\bm{x}_t^{\leftarrow}$ in the reverse process \eqref{eq:ap_rp} is the proximal operator. Fixing it at time $t=t_k$ will transform the reverse process into a linear SDE with time-dependent coefficients, which admits explicit solutions. 

\begin{proposition}[Exponential interpolation]
	\normalfont
	\label{pro:expoint}
	Consider the reverse SDE in \eqref{eq:ap_rp}, and define a sequence of stepsizes $\left\{\gamma_k\right\}_{k=0,1,\ldots,K}$ and $t_0=0, t_{k+1}=\sum_{j=0}^{k} \gamma_j$. For $\mu(t)$, $\sigma^2(t)$, and $\lambda(t)=\sigma^2(t)/\mu^2(t)$ defined as in Section~\ref{sec:method}, the interpolation process on the interval $t\in\left[t_k,t_{k+1}\right]$ is given by
	\begin{equation}
		\label{eq:ap_ei}
		d\bar{\bm{x}}_t^{EI} = \left\{-\left[a(T-t)+\frac{b^2(T-t)}{\sigma^2(T-t)}\right]\bar{\bm{x}}_t^{EI}+\frac{b^2(T-t)\mu(T-t)}{\sigma^2(T-t)}\text{Prox}_{U}^{\lambda(T-t_k)}\left(\frac{\bar{\bm{x}}_{t_k}^{EI}}{\mu(T-t_k)}\right)\right\}dt + b(T-t)d\bm{B}_t.
	\end{equation}
	Moreover, the corresponding iteration form is given by
	\begin{equation}
		\label{eq:ap_pds}
		\bar{\bm{x}}_{t_{k+1}}^{EI} = \alpha_{1,k}\bar{\bm{x}}_{t_{k}}^{EI}+ \alpha_{2,k} \text{Prox}_{U}^{\lambda(T-t_k)}\left(\frac{\bar{\bm{x}}_{t_{k}}^{EI}}{\mu(T-t_k)}\right) + \alpha_{3,k} \bm{\xi}_k,
	\end{equation}
	where $\bm{\xi}_k\sim\mathcal{N}(\bm{0},I)$ and
	\begin{equation}
		\begin{aligned}
			\alpha_{1,k} &= \frac{\lambda(T-t_{k+1})\mu(T-t_{k+1})}{\lambda(T-t_{k})\mu(T-t_{k})},\quad \alpha_{2,k} = \mu(T-t_{k+1})\left(1-\frac{\lambda(T-t_{k+1})}{\lambda(T-t_{k})}\right), \\
			\alpha_{3,k} &= \mu(T-t_{k+1})\sqrt{\lambda(T-t_{k+1})}\sqrt{1-\frac{\lambda(T-t_{k+1})}{\lambda(T-t_{k})}}.
		\end{aligned}
	\end{equation}
\end{proposition}

\begin{proof}
	For the interpolation process \eqref{eq:ap_ei} on the interval $t\in\left[t_k,t_{k+1}\right]$, denote $\bm{P}_k=\text{Prox}_{U}^{\lambda(T-t_k)}\left(\frac{\bar{\bm{x}}_{t_k}^{EI}}{\mu(T-t_k)}\right)$ and 
	\begin{equation}
		\nonumber
		\alpha(t)=-\left[a(T-t)+\frac{b^2(T-t)}{\sigma^2(T-t)}\right],\quad \beta(t)= \frac{b^2(T-t)\mu(T-t)}{\sigma^2(T-t)}\bm{P}_k,\quad\gamma(t)= b(T-t).
	\end{equation}
	Then the interpolation process can be reformulated as
	\begin{equation}
		\label{eq:ap_interp}
		d\bar{\bm{x}}_t^{EI} = \left\{\alpha(t)\bar{\bm{x}}_t^{EI}+\beta(t)\right\}dt+\gamma(t)d\bm{B}_t, t\in\left[t_k,t_{k+1}\right].
	\end{equation}
	This is a linear SDE and the exact solution at $\tau_{k+1}$ can be expressed explicitly. Denote $\Phi(t,s)=\exp\left\{\int_{s}^{t}\alpha(u)du\right\}$. Then the solution to \eqref{eq:ap_interp} can be written as
	\begin{equation}
		\label{eq:ap_eiinterp}
		\bar{\bm{x}}_{t_{k+1}}^{EI}=\underbrace{\Phi(t_{k+1},t_{k})\bar{\bm{x}}_{t_{k}}^{EI}}_{H}+\underbrace{\int_{t_{k}}^{t_{k+1}}\Phi(t_{k+1},s)\beta(s)ds}_{I} + \underbrace{\int_{t_{k}}^{t_{k+1}}\Phi(t_{k+1},s)\gamma(s)d\bm{B}_s}_{J}.
	\end{equation}
	To show the proposition, it suffices to calculate $H$, $I$, and $J$ in \eqref{eq:ap_eiinterp}. In the following, we calculate them one by one. Firstly, we calculate $H$. For simplicity, we denote the reverse time $\tau_k = T-t_k$, and then 
	\begin{equation}
		\nonumber
		\alpha(t)=-\left[a(\tau)+\frac{b^2(\tau)}{\sigma^2(\tau)}\right], \quad\beta(t)= \frac{b^2(\tau)\mu(\tau)}{\sigma^2(\tau)}P_k,\quad \gamma(t)= b(\tau).
	\end{equation}
	Note that $dt=-d\tau$, and as $u$ increases from $t_k$ to $t_{k+1}$, $\tau$ decreases from $\tau_k$ to $\tau_{k+1}$. Then we have
	\begin{equation}
		\label{eq:ap_phi}
		\Phi(t_{k+1},t_{k})=\exp\left\{\int_{t_{k}}^{t_{k+1}}\alpha(u)du\right\} = \exp\left\{\int_{\tau_{k+1}}^{\tau_k}\alpha(T-\tau)d\tau\right\}.
	\end{equation}
	As $\alpha(T-\tau)=-\left[a(\tau)+\frac{b^2(\tau)}{\sigma^2(\tau)}\right]$, by integration, the exponential term in \eqref{eq:ap_phi} can be written as
	\begin{equation}
		\label{eq:ap_alphau}
		\int_{t_{k}}^{t_{k+1}}\alpha(u)du = \int_{\tau_{k+1}}^{\tau_k}-\left[a(\tau)+\frac{b^2(\tau)}{\sigma^2(\tau)}\right]d\tau = -\int_{\tau_{k+1}}^{\tau_k} a(\tau) d\tau-\int_{\tau_{k+1}}^{\tau_k} \frac{b^2(\tau)}{\sigma^2(\tau)} d\tau.
	\end{equation}
	Recall that $\mu(\tau)=\exp\left\{\int_{0}^{\tau}a(s)ds\right\}$. Then we have $\frac{d}{d\tau} \log \mu(\tau) = a(\tau)$, and then
	\begin{equation}
		\label{eq:ap_alphatau}
		\int_{\tau_{k+1}}^{\tau_k} a(\tau) d\tau = \log \mu(\tau_{k})-\log \mu(\tau_{k+1})=\log \frac{\mu(\tau_{k})}{\mu(\tau_{k+1})}.
	\end{equation}
	Further, note that $\sigma^2(\tau)=\mu^2(\tau)\int_{0}^{\tau}(b(s)/\mu(s))^2ds$. Denote $\lambda(\tau)=\int_{0}^{\tau}(b(s)/\mu(s))^2ds = \sigma^2(\tau)/\mu^2(\tau)$. Then we have
	\begin{equation}
		\nonumber
		\frac{b^2(\tau)}{\sigma^2(\tau)}=\frac{b^2(\tau)}{\mu^2(\tau)\lambda(\tau)}=\frac{1}{\lambda(\tau)}\frac{b^2(\tau)}{\mu^2(\tau)}.
	\end{equation}
	Also, we have $\lambda^{\prime}(\tau)=(b(\tau)/\mu(\tau))^2$, and then
	\begin{equation}
		\nonumber
		\frac{b^2(\tau)}{\sigma^2(\tau)}=\frac{1}{\lambda(\tau)}\frac{b^2(\tau)}{\mu^2(\tau)}=\frac{\lambda^{\prime}(\tau)}{\lambda(\tau)}=\frac{d}{d\tau} \log \lambda(\tau).
	\end{equation}
	So we have
	\begin{equation}
		\label{eq:ap_bsigma}
		\int_{\tau_{k+1}}^{\tau_k} \frac{b^2(\tau)}{\sigma^2(\tau)} d\tau = \log \lambda(\tau_{k})-\log \lambda(\tau_{k+1}) = \log \frac{\lambda(\tau_{k})}{\lambda(\tau_{k+1})}.
	\end{equation}
	Substituting \eqref{eq:ap_alphau}, \eqref{eq:ap_alphatau}, and \eqref{eq:ap_bsigma} into \eqref{eq:ap_phi} yields
	\begin{equation}
		\label{eq:ap_EIH}
		\Phi(t_{k+1},t_{k})=\exp\left\{\int_{t_{k}}^{t_{k+1}}\alpha(u)du\right\} = \frac{\lambda(\tau_{k+1})}{\lambda(\tau_{k})} \frac{\mu(\tau_{k+1})}{\mu(\tau_{k})}= \frac{\sigma^2(\tau_{k+1})}{\sigma^2(\tau_{k})} \frac{\mu(\tau_{k})}{\mu(\tau_{k+1})}.
	\end{equation}
	
	Next, we calculate the second term $I=\int_{t_{k}}^{t_{k+1}}\Phi(t_{k+1},s)\beta(s)ds$ in \eqref{eq:ap_eiinterp}. Similarly, for $s\in\left[t_{k}, t_{k+1}\right]$, we denote $\eta = T-s = \rho$. Then we have
	\begin{equation}
		\nonumber
		\Phi(t_{k+1},s)=\exp\left\{\int_{s}^{t_{k+1}}\alpha(u)du\right\} = \exp\left\{\int_{\tau_{k+1}}^{\rho}\alpha(T-\eta)d\eta\right\}.
	\end{equation}
	Note that 
	\begin{equation}
		\nonumber
		\int_{\tau_{k+1}}^{\rho}\alpha(T-\eta)d\eta = \int_{\tau_{k+1}}^{\rho}-\left[a(\eta)+\frac{b^2(\eta)}{\sigma^2(\eta)}\right]d\eta = -\int_{\tau_{k+1}}^{\rho} a(\eta) d\eta-\int_{\tau_{k+1}}^{\rho} \frac{b^2(\eta)}{\sigma^2(\eta)} d\eta.
	\end{equation}
	Similarly, we have
	\begin{equation}
		\nonumber
		\Phi(t_{k+1},s)= \frac{\lambda(\tau_{k+1})}{\lambda(\rho)} \frac{\mu(\tau_{k+1})}{\mu(\rho)}= \frac{\sigma^2(\tau_{k+1})}{\sigma^2(\rho)} \frac{\mu(\rho)}{\mu(\tau_{k+1})}.
	\end{equation}
	Substituting the above equation into $I$, we have
	\begin{equation}
		\label{eq:ap_EII}
		\begin{aligned}
			I=\int_{t_{k}}^{t_{k+1}}\Phi(t_{k+1},s)\beta(s)ds =\int_{\tau_{k+1}}^{\tau_{k}} \frac{\sigma^2(\tau_{k+1})}{\sigma^2(\rho)} \frac{\mu(\rho)}{\mu(\tau_{k+1})} \frac{b^2(\rho)\mu(\rho)}{\sigma^2(\rho)}P_k d\rho= \frac{\sigma^2(\tau_{k+1})}{\mu(\tau_{k+1})}P_k \int_{\tau_{k+1}}^{\tau_{k}} \frac{b^2(\rho)\mu^2(\rho)}{\sigma^4(\rho)} d\rho.
		\end{aligned}
	\end{equation}
	Note that 
	\begin{equation}
		\nonumber
		\frac{\mu^2(\rho)b^2(\rho)}{\sigma^4(\rho)}=\frac{b^2(\rho)}{\mu^2(\rho)\lambda^2(\rho)}=\frac{\lambda^{\prime}(\rho)}{\lambda^2(\rho)}=-\frac{d}{d\rho}\left(\frac{1}{\lambda(\rho)}\right),
	\end{equation}
	then by integration, we get
	\begin{equation}
		\int_{\tau_{k+1}}^{\tau_k}\frac{\mu^2(\rho)b^2(\rho)}{\sigma^4(\rho)}d\rho=\int_{\tau_{k+1}}^{\tau_k}\frac{\lambda^{\prime}(\rho)}{\lambda^2(\rho)}d\rho=\left[-\frac{1}{\lambda(\rho)}\right]_{\tau_{k+1}}^{\tau_k}=\frac{1}{\lambda(\tau_{k+1})}-\frac{1}{\lambda(\tau_k)}.
	\end{equation}
	Substituting the above into \eqref{eq:ap_EII}, we get
	\begin{equation}
		\label{eq:ap_EIII}
		I= \frac{\sigma^2(\tau_{k+1})}{\mu(\tau_{k+1})}P_k \int_{\tau_{k+1}}^{\tau_{k}} \frac{b^2(\rho)\mu^2(\rho)}{\sigma^4(\rho)} d\rho = \frac{\sigma^2(\tau_{k+1})}{\mu(\tau_{k+1})}P_k\left(\frac{1}{\lambda(\tau_{k+1})}-\frac{1}{\lambda(\tau_k)}\right).
	\end{equation}
	
	Finally, we calculate the stochastic It\^{o} integral of the third term $J$ in \eqref{eq:ap_eiinterp}. The variance of $J$ can be expressed as
	\begin{equation}
		\nonumber
		\text{Var}(J)=\int_{t_k}^{t_{k+1}}\left[\Phi(t_{k+1},s)\gamma(s)\right]^2ds.
	\end{equation}
	Similarly, we have
	\begin{equation}
		\nonumber
		[\Phi(t_{k+1},T-\rho)b(\rho)]^2=\frac{\mu^2(\rho)\sigma^4(\tau_{k+1})}{\mu^2(\tau_{k+1})\sigma^4(\rho)}b^2(\rho)=\frac{\sigma^4(\tau_{k+1})}{\mu^2(\tau_{k+1})}\cdot\frac{\mu^2(\rho)b^2(\rho)}{\sigma^4(\rho)},
	\end{equation}
	and further
	\begin{equation}
		\nonumber
		\begin{aligned}
			\int_{t_k}^{t_{k+1}}\left[\Phi(t_{k+1},s)\gamma(s)\right]^2ds&=\int_{\tau_{k+1}}^{\tau_k}\left[\Phi(t_{k+1},T-\rho)b(\rho)\right]^2d\rho=\frac{\sigma^4(\tau_{k+1})}{\mu^2(\tau_{k+1})}\int_{\tau_{k+1}}^{\tau_k}\frac{\mu^2(\rho)b^2(\rho)}{\sigma^4(\rho)}d\rho\\
			&=\frac{\sigma^4(\tau_{k+1})}{\mu^2(\tau_{k+1})}\left(\frac{1}{\lambda(\tau_{k+1})}-\frac{1}{\lambda(\tau_k)}\right).
		\end{aligned}
	\end{equation}
	For $\bm{\xi}_k\sim\mathcal{N}(\bm{0},I)$, we have
	\begin{equation}
		\label{eq:ap_EIJ}
		J=\frac{\sigma^2(\tau_{k+1})}{\mu(\tau_{k+1})}\sqrt{\frac{1}{\lambda(\tau_{k+1})}-\frac{1}{\lambda(\tau_k)}}\bm{\xi}_k.
	\end{equation}
	
	Combining \eqref{eq:ap_EIH}, \eqref{eq:ap_EIII}, and \eqref{eq:ap_EIJ}, we have
	\begin{equation}
		\begin{aligned}
			\bar{\bm{x}}_{t_{k+1}}^{EI}=&\Phi(t_{k+1},t_{k})\bar{\bm{x}}_{t_{k}}^{EI}+\int_{t_{k}}^{t_{k+1}}\Phi(t_{k+1},s)\beta(s)ds + \int_{t_{k}}^{t_{k+1}}\Phi(t_{k+1},s)\gamma(s)d\bm{B}_s\\
			=& \frac{\sigma^2(\tau_{k+1})}{\sigma^2(\tau_{k})} \frac{\mu(\tau_{k})}{\mu(\tau_{k+1})}\bar{\bm{x}}_{t_{k}}^{EI} + \frac{\sigma^2(\tau_{k+1})}{\mu(\tau_{k+1})}P_k\left(\frac{1}{\lambda(\tau_{k+1})}-\frac{1}{\lambda(\tau_k)}\right)+ \frac{\sigma^2(\tau_{k+1})}{\mu(\tau_{k+1})}\sqrt{\frac{1}{\lambda(\tau_{k+1})}-\frac{1}{\lambda(\tau_k)}}\bm{\xi}_k.
		\end{aligned}
	\end{equation}
	Denote the coefficients
	\begin{equation}
		\nonumber
		\begin{aligned}
			\alpha_{1,k} &= \frac{\lambda(T-t_{k+1})\mu(T-t_{k+1})}{\lambda(T-t_{k})\mu(T-t_{k})}, \quad \alpha_{2,k} = \mu(T-t_{k+1})\left(1-\frac{\lambda(T-t_{k+1})}{\lambda(T-t_{k})}\right), \\
			\alpha_{3,k} &= \mu(T-t_{k+1})\sqrt{\lambda(T-t_{k+1})}\sqrt{1-\frac{\lambda(T-t_{k+1})}{\lambda(T-t_{k})}}.
		\end{aligned}
	\end{equation}
	Then the corresponding iteration form is given by
	\begin{equation}
		\nonumber
		\bar{\bm{x}}_{t_{k+1}}^{EI} = \alpha_{1,k}\bar{\bm{x}}_{t_{k}}^{EI}+ \alpha_{2,k} \text{Prox}_{U}^{\lambda(T-t_k)}\left(\frac{\bar{\bm{x}}_{t_k}^{EI}}{\mu(T-t_k)}\right) + \alpha_{3,k} \bm{\xi}_k, \quad \bm{\xi}_k\sim\mathcal{N}(\bm{0,I}).
	\end{equation}
	
\end{proof}

The exponential interpolation scheme provides an exact discretization of the reverse SDE with proximal operator. The resulting iteration in \eqref{eq:ap_pds} has a clear interpretation: it is a convex combination of the previous state \(\bar{\bm{x}}_{t_k}^{EI}\) and the proximal update plus appropriate noise. The coefficients \(\alpha_{1,k}, \alpha_{2,k}, \alpha_{3,k}\) depend only on the scheduling \(\mu(t)\) and \(\lambda(t)\), making the scheme easy to implement.

\section{Convergence Results}
\label{ap:theorem}

In this section, we establish the main convergence result of the work, providing a comprehensive error analysis for the proposed proximal diffusion sampling algorithm. Firstly, in subsection \ref{ap:sub1}, we show several quantitative controls on the forward density \(\pi_t\). Next, we provide uniform second-moment bounds for the reverse diffusion processes with approximated Moreau score in subsection \ref{ap:sub2}. Then, we analyze the tangent process associated with the reverse SDE in subsection \ref{ap:sub3}, which quantifies how small perturbations propagate along the reverse-time dynamics. Finally, in subsection \ref{ap:sub4}, we present the proof of Theorem~\ref{thm:dis_conv}, which establishes the non‑asymptotic convergence of the proximal diffusion sampler. 

\subsection{Properties of the Forward Process}
\label{ap:sub1}

In this subsection, we establish several technical bounds for the forward diffusion process that will be used to control the growth of moments and to analyze the errors in the sampling phase. 

Recall that the forward SDE is given by
\[
d\bm{x}_t^{\rightarrow}=a(t)\bm{x}_t^{\rightarrow}dt+b(t)d\bm{B}_t,\qquad t\in[0,T],
\]
with deterministic coefficients
\[
\mu(t)=\exp\Bigl\{\int_0^t a(s)ds\Bigr\},\qquad 
\sigma^2(t)=\int_0^t b^2(s)\exp\Bigl\{2\int_s^t a(r)dr\Bigr\}ds .
\]

We work under Assumptions~\ref{assumption1} and~\ref{assumption2}. Denote $\text{diam}(\mathcal{X})$ the diameter of the manifold defined by $\text{diam}(\mathcal{X})=\sup\{\|x-y\|:x,y\in\mathcal{X}\}$. We emphasize that the compact assumption on $\mathcal{X}$ encompasses not only all distributions which admit a continuous density on a lower dimensional manifold but also all empirical densities. 

Under Assumption \ref{assumption2}, by directly restricting the log-gradient of $\mu(t)$ and $\lambda(t)$, we can obtain simple expressions for the ratio in the discrete iteration form, i.e., 
\begin{equation}
	\exp\left\{\gamma_k m_\lambda\right\}\leq \frac{\lambda(T-t_{k})}{\lambda(T-t_{k+1})}\leq \exp\left\{\gamma_k M_\lambda\right\},\quad \exp\left\{\gamma_k m_\mu\right\}\leq\frac{\mu(T-t_{k+1})}{\mu(T-t_{k})}\leq \exp\left\{\gamma_k M_\mu\right\}.
\end{equation}
The gradient can also be expressed as
\begin{equation}
	|\partial_{t} \mu(t)|\leq |\partial_{t} \log \mu(t)| \leq M_{\mu},\quad \partial_{t} \lambda(t)\leq \partial_{t} \log  \lambda(t)\leq \bar{\lambda} M_{\lambda}.
\end{equation}
Further, we can conclude that both $\mu(t)$ and $\lambda(t)$ can be bounded, i.e., 
\begin{equation}
	\mu(t)\in\left[\exp\left\{-M_{\mu}T\right\},1\right], \quad \lambda(t)\in\left[\bar{\lambda}\exp\left\{-M_{\lambda}T\right\},\bar{\lambda}\right], \quad \forall~t\in[0,T].
\end{equation}

In order to show the stability and growth of the processes at hand, we need to control quantities related to the gradient and Hessian of the log-density. For the target posterior $\pi(\bm{x}_0)\propto\exp\{-U(\bm{x}_0)\}$, denote $(\pi_t)_{t\in[0,T]}$ as the density w.r.t. the Lebesgue measure of the distribution of $\bm{x}_t^{\rightarrow}$.  Similarly, we suppose $\pi^N$ to be an empirical version of $\pi$, i.e., 
\begin{equation}
	\nonumber
	\pi^N=\frac{1}{N}\sum_{k=1}^{N} \bm{x}^{(k)}, \quad \bm{x}^{(k)}\sim \pi.
\end{equation}
And we denote $(\pi_t^N)_{t\in[0,T]}$ as the density w.r.t. the Lebesgue measure of the empirical version of distribution. Next we will give the regularity of the forward density. We adopt the backbone of the proof in \cite{de2022convergence} and extend it to more general settings.

\begin{lemma}[Regularity of the forward density]
	\normalfont
	\label{lem:forward}
	Under Assumption \ref{assumption1}, for every $t\in[0,T]$ and $\bm{x}\in\mathbb{R}^d$ the following inequalities hold.
	\begin{itemize}
		\item[1.] (Control on the gradient) 
		\begin{equation}
			\label{eq:lemma_c1}
			\begin{aligned}
				&\langle\nabla\log \pi_{t}(\bm{x}_{t}),\bm{x}_{t}\rangle\leq-\frac{\left\|\bm{x}_{t}\right\|^{2}}{\sigma^{2}(t)}+\frac{\mu(t)\text{diam}(\mathcal{X})\left\|\bm{x}_{t}\right\|}{\sigma^{2}(t)},\\
				&\|\nabla\log \pi_{t}(\bm{x}_{t})\|^{2}\le\frac{2\|\bm{x}_{t}\|^{2}}{\sigma^{4}(t)}+\frac{2\mu^{2}(t)\text{diam}(\mathcal{X})^{2}}{\sigma^{4}(t)}.
			\end{aligned}
		\end{equation}
		\item[2.] (Control on the Hessian) For any $\bm{y}\in\mathbb{R}^d$, 
		\begin{equation}
			\label{eq:lemma_c2}
			\begin{aligned}
				&\langle\bm{y},\nabla^2\log \pi_t(\bm{x}_t)\bm{y}\rangle\leq-\frac{2\sigma^2(t)-\mu^2(t)\text{diam}(\mathcal{X})^2}{2\sigma^4(t)}\|\bm{y}\|^2,\\
				&\|\nabla^2 \log \pi_t(\bm{x}_t)\| \leq \frac{1 + \text{diam}(\mathcal{X})^2}{\sigma^4(t)} .
			\end{aligned}
		\end{equation}
		\item[3.] (Control on the time derivative)
		\begin{equation}
			\label{eq:lemma_c3}
			\|\partial_{t}\nabla\log \pi_{t}(\bm{x})\|\leq M_{f,t}\frac{(\text{diam}(\mathcal{X})^{2}+\bar{\lambda})(\text{diam}(\mathcal{X})+\|x\|)}{\sigma^{6}(t)},
		\end{equation}
		where $M_{f,t} = 2(M_{\lambda}+4M_{\mu})\bar{\lambda}$.
	\end{itemize}
\end{lemma}

\begin{proof}
	We prove each part in sequence. 
	
	\noindent\textbf{Part (i): Control on the gradient.}
	
	We first show the dissipativity condition on the gradient, i.e., the inequalities in \eqref{eq:lemma_c1}. For any $t\in[0,T]$ and $\bm{x}_{t}\in\mathbb{R}^{d}$ we have
	\begin{equation}
		\nonumber
		\pi_{t}^{N}(\bm{x}_{t})=\frac{1}{N}\sum_{k=1}^{N}\frac{\exp\bigl[-\|\bm{x}_{t}-\mu(t)\bm{x}^{(k)}\|^{2}/(2\sigma^2(t))\bigr]}{(2\pi\sigma^2(t))^{d/2}}.
	\end{equation}
	And the gradient of $\log \pi_{t}^{N}$ is given by
	\begin{equation}
		\nabla\log \pi_{t}^{N}(\bm{x}_{t})=\frac{-\frac{1}{N}\sum_{k=1}^{N}(\bm{x}_{t}-\mu(t)\bm{x}^{(k)})\exp\bigl[-\|\bm{x}_{t}-\mu(t)\bm{x}^{(k)}\|^{2}/(2\sigma^2(t))\bigr]/(2\pi\sigma^2(t))^{d/2}}{\sigma^2(t)\,\pi_{t}^{N}(\bm{x}_{t})}.
	\end{equation}
	Define the weights
	\begin{equation}
		\nonumber
		w_{k}:=\frac{\exp\bigl[-\|\bm{x}_{t}-\mu(t)X^{k}\|^{2}/(2\sigma^{2}(t))\bigr]}{\sum_{j=1}^{N}\exp\bigl[-\|\bm{x}_{t}-\mu(t)\bm{x}^{(j)}\|^{2}/(2\sigma^{2}(t))\bigr]},\qquad k=1,2,\dots,N.
	\end{equation}
	Then we have $w_{k}\ge0$ and $\sum_{k=1}^{N}w_{k}=1$, and
	\begin{equation}
		\label{eq:ap_ptn}
		\nabla\log \pi_{t}^{N}(\bm{x}_{t})=-\frac{1}{\sigma^{2}(t)}\Bigl(\bm{x}_{t}-\mu(t)\sum_{k=1}^{N}w_{k}\bm{x}^{(k)}\Bigr).
	\end{equation}
	
	Taking the inner product of \eqref{eq:ap_ptn} with $\bm{x}_{t}$ gives
	\begin{equation}
		\label{eq:ap_ptninner}
		\langle\nabla\log \pi_{t}^{N}(\bm{x}_{t}),\bm{x}_{t}\rangle
		=-\frac{1}{\sigma^{2}(t)}\Bigl(\|\bm{x}_{t}\|^{2}-\mu(t)\Bigl\langle \bm{x}_{t},\sum_{k=1}^{N}w_{k}\bm{x}^{(k)}\Bigr\rangle\Bigr)\leq -\frac{\|\bm{x}_{t}\|^{2}}{\sigma^{2}(t)}+\frac{\mu(t)}{\sigma^{2}(t)}\Bigl|\Bigl\langle \bm{x}_{t},\sum_{k=1}^{N}w_{k}\bm{x}^{(k)}\Bigr\rangle\Bigr|.
	\end{equation}
	By the Cauchy-Schwarz inequality, we have
	\begin{equation}
		\label{eq:ap_ptninner2}
		\Bigl|\Bigl\langle \bm{x}_{t},\sum_{k=1}^{N}w_{k}\bm{x}^{(k)}\Bigr\rangle\Bigr|
		\le \|\bm{x}_{t}\|\,\Bigl\|\sum_{k=1}^{N}w_{k}\bm{x}^{(k)}\Bigr\|
		\le \|\bm{x}_{t}\|\sum_{k=1}^{N}w_{k}\|\bm{x}^{(k)}\|.
	\end{equation}
	Since $\mathcal{X}$ is bounded with diameter $\text{diam}(\mathcal{X})$, we have $\|\bm{x}^{(k)}\|\leq\text{diam}(\mathcal{X})$ for every $k$ (after possibly translating $\mathcal{X}$ so that it lies inside a ball of radius $\text{diam}(\mathcal{X})$ centered at the origin). Consequently,
	\begin{equation}
		\label{eq:ap_ptninner3}
		\sum_{k=1}^{N}w_{k}\|\bm{x}^{(k)}\|\leq\text{diam}(\mathcal{X})\sum_{k=1}^{N}w_{k}=\text{diam}(\mathcal{X}).
	\end{equation}
	Substituting \eqref{eq:ap_ptninner2} and \eqref{eq:ap_ptninner3} into \eqref{eq:ap_ptninner}, we have
	\begin{equation}
		\langle\nabla\log \pi_{t}^{N}(\bm{x}_{t}),\bm{x}_{t}\rangle\leq-\frac{\|\bm{x}_{t}\|^{2}}{\sigma^{2}(t)}+\frac{\mu(t)\text{diam}(\mathcal{X})\|\bm{x}_{t}\|}{\sigma^{2}(t)}.
	\end{equation}
	Because the right‑hand side does not depend on \(N\), the same estimate holds for the limiting density $\pi_{t}$ by letting $N\to+\infty$.
	
	Next, we show the second inequality in \eqref{eq:lemma_c1}. From the above equation \eqref{eq:ap_ptn}, we obtain
	\begin{equation}
		\|\nabla\log \pi_{t}^{N}(\bm{x}_{t})\|=\frac{1}{\sigma^{2}(t)}\Bigl\|\bm{x}_{t}-\mu(t)\sum_{k=1}^{N}w_{k}\bm{x}^{(k)}\Bigr\|.
	\end{equation}
	Using the inequality $\|a-b\|^{2}\le2\|a\|^{2}+2\|b\|^{2}$, we have
	\begin{equation}
		\Bigl\|\bm{x}_{t}-\mu(t)\sum_{k=1}^{N}w_{k}\bm{x}^{(k)}\Bigr\|^{2}
		\le 2\|\bm{x}_{t}\|^{2}+2\mu^{2}(t)\Bigl\|\sum_{k=1}^{N}w_{k}\bm{x}^{(k)}\Bigr\|^{2}.
	\end{equation}
	Again by the convexity of the norm and the bound $\|\bm{x}^{(k)}\|\leq\text{diam}(\mathcal{X})$, we get
	\begin{equation}
		\Bigl\|\sum_{k=1}^{N}w_{k}\bm{x}^{(k)}\Bigr\|\le\sum_{k=1}^{N}w_{k}\|\bm{x}^{(k)}\|\le\text{diam}(\mathcal{X}).
	\end{equation}
	Hence,
	\begin{equation}
		\Bigl\|\bm{x}_{t}-\mu(t)\sum_{k=1}^{N}w_{k}\bm{x}^{(k)}\Bigr\|^{2}
		\le 2\|\bm{x}_{t}\|^{2}+2\mu^{2}(t)\text{diam}(\mathcal{X})^{2},
	\end{equation}
	and consequently
	\begin{equation}
		\|\nabla\log \pi_{t}^{N}(\bm{x}_{t})\|^{2}\le\frac{2\|\bm{x}_{t}\|^{2}}{\sigma^{4}(t)}+\frac{2\mu^{2}(t)\text{diam}(\mathcal{X})^{2}}{\sigma^{4}(t)}.
	\end{equation}
	Again, letting $N\to+\infty$ yields the same bound for $\pi_{t}$.
	
	\noindent\textbf{Part (ii): Control on the Hessian.}
	
	We now prove the control on the Hessian $\nabla^2\log \pi_t$. Denote $\bar{\pi}_t^N = \pi_t^N (2 \pi \sigma^2(t))^{d/2}$. Then we have
	\begin{equation}
		\label{eq:ap_pitnbar}
		\bar{\pi}_t^N(\bm{x}) = (1/N) \sum_{k=1}^N \exp[-\|\bm{x} - \mu(t) \bm{x}^{(k)}\|^2 / 2 \sigma^2(t)].
	\end{equation}
	Hence, taking the gradient on \eqref{eq:ap_pitnbar}, we have
	\begin{equation}
		\nonumber
		\nabla \log \bar{\pi}_t^N(\bm{x}) = (-1/N) \sum_{k=1}^N (\bm{x} - \mu(t) \bm{x}^{(k)}) \exp[-\|\bm{x} - \mu(t) \bm{x}^{(k)}\|^2 / 2 \sigma^2(t)] / (\sigma^2(t) \bar{\pi}_t^N(\bm{x})).
	\end{equation}
	The Hessian can also be calculated as
	\begin{equation}
		\label{eq:ap_pitnbar2}
		\begin{aligned}
			\nabla^2 \log \bar{\pi}_t^N(\bm{x}) =& -\frac{1}{\sigma^2(t)} I  + (1/N) \sum_{k=1}^N (\bm{x} - \mu(t) \bm{x}^{(k)}) \otimes (\bm{x} - \mu(t) \bm{x}^{(k)}) \exp[-\|\bm{x} - \mu(t) \bm{x}^{(k)}\|^2 / 2 \sigma^2(t)] / (\sigma^4(t) \bar{\pi}_t^N(\bm{x})) \\
			&- (1/N^2) (\sum_{k=1}^N (\bm{x} - \mu(t) \bm{x}^{(k)}) \exp[-\|\bm{x} - \mu(t) \bm{x}^{(k)}\|^2 / 2 \sigma^2(t)]) \\
			&\quad \otimes (\sum_{k=1}^N (\bm{x} - \mu(t) \bm{x}^{(k)}) \exp[-\|\bm{x} - \mu(t) \bm{x}^{(k)}\|^2 / 2 \sigma^2(t)]) / (\sigma^2(t) \bar{\pi}_t^N(\bm{x}))^2.
		\end{aligned}
	\end{equation}
	For any $k \in \{0,1, \ldots, N-1\}$, denote $\bm{f}_t^k = -(\bm{x} - \mu(t) \bm{x}^{(k)}) / \sigma^2(t)$ and $e_t^k = \exp[-\|\bm{f}_t^k\|^2]$. Substituting this into \eqref{eq:ap_pitnbar2}, we have
	\begin{equation}
		\nonumber
		\begin{aligned}
			\nabla^{2}\log\bar{\pi}_{t}^{N}(\bm{x})&=-\frac{1}{\sigma^2(t)} I+\sum_{k=1}^{N}\bm{f}_t^k\otimes \bm{f}_t^ke_t^k/\sum_{k=1}^Ne_t^k-(\sum_{k=1}^N\bm{f}_t^ke_t^k/\sum_{k=1}^Ne_t^k)\otimes(\sum_{k=1}^N\bm{f}_t^ke_t^k/\sum_{k=1}^Ne_t^k)\\
			&=-\frac{1}{\sigma^2(t)}I +(1/2)\sum_{j,k=1}^N(\bm{f}_t^k-\bm{f}_t^j)\otimes(\bm{f}_t^k-\bm{f}_t^j)e_t^ke_t^j/\sum_{k,j=1}^Ne_t^ke_t^j.
		\end{aligned}
	\end{equation}
	
	In addition, using that for any $\ell\in\{1,2,\ldots,N\},\bm{x}^{(\ell)}\in\mathcal{X}$ we have
	\begin{equation}
		\nonumber
		\|\bm{f}_t^k-\bm{f}_t^j\|=\mu(t)\|\bm{x}^{(k)}-\bm{x}^{(j)}\|/\sigma^2(t)\leq \mu(t)\text{diam}(\mathcal{X})/\sigma^2(t).
	\end{equation}
	
	Therefore, we get that, for any $\bm{x}$,
	\begin{equation}
		\nonumber
		\langle\bm{x},\nabla^2\log\bar{\pi}_t^N(\bm{x})\bm{x}\rangle\leq-(1-\mu(t)^2\text{diam}(\mathcal{X})^2/(2\sigma^2(t)))/\sigma^2(t)\|\bm{x}\|^2.
	\end{equation}
	
	Using the fact that $\mathcal{X}$ is compact and the strong law of large numbers, we have
	\begin{equation}
		\nonumber
		\begin{aligned}
			&\lim_{N\to+\infty}\nabla^2\log\bar{\pi}_t^N(\bm{x})=-\frac{1}{\sigma^2(t)}I\\
			&+\int_{\mathbb{R}^d}(\bm{x}-\mu(t)\bm{x}_0)\otimes(\bm{x}-\mu(t)\bar{\bm{x}}_0)\exp[-\|\bm{x}-\mu(t)\bm{x}_0\|^2/(2\sigma^2(t))]\exp[-\|\bm{x}-\mu(t)\bar{\bm{x}}_0\|^2/(2\sigma^2(t))]\mathrm{d}\pi(\bm{x}_0)\mathrm{d}\pi(\bar{\bm{x}}_0)\\
			&/(\int_{\mathbb{R}^d}\exp[-\|\bm{x}-\mu(t)\bar{\bm{x}}_0\|^2/(2\sigma^2(t))]\mathrm{d}\pi(\bm{x}_0))^2.
		\end{aligned}
	\end{equation}
	Hence, we get that $\lim_{N\to+\infty}\nabla^2\log \pi_t^N(\bm{x})=\nabla^2\log \pi_t$, and 
	\begin{equation}
		\nonumber
		\langle\bm{x},\nabla^2\log \pi_t(\bm{x})\bm{x}\rangle\leq-(1-\mu(t)^2\text{diam}(\mathcal{X})^2/(2\sigma^2(t)))/\sigma^2(t)\|\bm{x}\|^2.
	\end{equation}
	Similarly, we can show
	\begin{equation}
		\nonumber
		\|\nabla^2 \log p_t(\bm{x}_t)\| \leq (1 + \text{diam}(\mathcal{X})^2) / \sigma^4(t).
	\end{equation}
	
	\noindent\textbf{Part (iii): Control on the time derivative.}
	
	Finally, in order to control the local error of the time discretization, we also need to control the time derivative of the gradient in \eqref{eq:lemma_c3}.
	
	Note that for any $\bm{x}\in\mathbb{R}^d,\pi_t^N(\bm{x})=\bar{\pi}_t^N(\bm{x})/(2\pi\sigma^2(t))^{d/2}$ with
	\begin{equation}
		\nonumber
		\bar{\pi}_t^N(\bm{x})=(1/N)\sum_{k=1}^Ne_t^k(\bm{x}),\quad e_t^k(\bm{x})=\exp[-\|\bm{x}-\mu(t)\bm{x}^{(k)}\|^2/(2\sigma^2(t))].
	\end{equation}
	In what follows, we denote $f_t^k=\log e_t^k$ for any $k\in\{1,2,\ldots,N\}$. For any $\bm{x}\in\mathbb{R}^d$, we have
	\begin{equation}
		\nonumber
		\partial_t\log\bar{p}_t^N(\bm{x})=\sum_{k=1}^N\partial_tf_t^k(\bm{x})e_t^k(\bm{x})/\sum_{k=1}^Ne_t^k(\bm{x}).
	\end{equation}
	
	Taking the gradient in the above, we further have
	\begin{equation}
		\label{eq:ap_ptptn}
		\begin{aligned}
			\partial_{t}\nabla\log\bar{p}_{t}^{N}(\bm{x})=&\sum_{k=1}^N\partial_t\partial_t^k(\bm{x})e_t^k(\bm{x})/\sum_{k=1}^Ne_t^k(\bm{x})+\sum_{k=1}^N\partial_tf_t^k(\bm{x})\nabla f_t^k(\bm{x})e_t^k(\bm{x})/\sum_{k=1}^Ne_t^k(\bm{x})\\
			&-\sum_{k,j=1}^N\partial_tf_t^k(\bm{x})\nabla f_t^j(\bm{x})e_t^k(\bm{x})e_t^j(\bm{x})/\sum_{k,j=1}^Ne_t^k(\bm{x})e_t^j(\bm{x})\\
			=&\sum_{k=1}^N\partial_t\nabla f_t^k(\bm{x})e_t^k(\bm{x})/\sum_{k=1}^Ne_t^k(\bm{x})\\
			&+\frac{1}{2}\sum_{k,j=1}^N(\partial_tf_t^k(\bm{x})-\partial_tf_t^j(\bm{x}))(\nabla f_t^k(\bm{x})-\nabla f_t^j(\bm{x}))e_t^k(\bm{x})e_t^j(\bm{x})/\sum_{k,j=1}^Ne_t^k(\bm{x})e_t^j(\bm{x}).
		\end{aligned}
	\end{equation}
	Note that for any $\bm{x}\in\mathbb{R}^d$,
	\begin{equation}
		\nonumber
		\nabla f_t^k(\bm{x})=-(\bm{x}-\mu(t)\bm{x}^{(k)})/\sigma^2(t).
	\end{equation}
	Hence, using that $\mu(t)\leq 1$, we get
	\begin{equation}
		\label{eq:ap_ftk}
		\|\nabla f_t^k(\bm{x})-\nabla f_t^j(\bm{x})\|\leq \mu(t)\text{diam}(\mathcal{X})/\sigma^2(t)\leq\text{diam}(\mathcal{X})/\sigma^2(t).
	\end{equation}
	In addition, we have that, for any $\bm{x}\in\mathbb{R}^d$
	\begin{equation}
		\nonumber
		\partial_tf_t^k(\bm{x})=\partial_t\sigma^2(t)/(2\sigma^4(t))\|\bm{x}-\mu(t)\bm{x}^{(k)}\|^2+\partial_t\mu(t)/\sigma^2(t)\langle \bm{x}^{(k)},\bm{x}-\mu(t)\bm{x}^{(k)}\rangle.
	\end{equation}
	Combining the above result with Assumption \ref{assumption2}, and note that $\partial_t\sigma^2(t)=\mu^2(t)\partial_t\lambda(t)+2\mu(t)\partial_t\mu(t)\lambda(t)$, we get
	\begin{equation}
		\label{eq:ap_tftk}
		\begin{aligned}
			|\partial_{t}f_{t}^{k}(\bm{x})-\partial_{t}f_{t}^{j}(\bm{x})|&\leq \partial_t\sigma^2(t)/(2\sigma^4(t))\left(4\|x\|\text{diam}(\mathcal{X})+2\text{diam}(\mathcal{X})^2\right) + \partial_t\mu(t)/\sigma^2(t)\left(2\text{diam}(\mathcal{X})\|x\|+2\text{diam}(\mathcal{X})^2\right)\\
			&\leq \text{diam}(\mathcal{X})(M_{\lambda}+2M_{\mu})\bar{\lambda}/\sigma^4(t)\left(2\|x\|+\text{diam}(\mathcal{X})\right)+ 2\text{diam}(\mathcal{X})M_{\mu}/\sigma^2(t)\left(\|x\|+2\text{diam}(\mathcal{X})\right)\\
			&\leq(M_{f,t}/\sigma^{4}(t))\text{diam}(\mathcal{X})(\text{diam}(\mathcal{X})+\|x\|),
		\end{aligned}
	\end{equation}
	where $M_{f,t} = 2(M_{\lambda}+4M_{\mu})\bar{\lambda}$. Therefore, combining \eqref{eq:ap_tftk} and the fact that $\mu(t)\leq1$, we get that, for any $x\in\mathbb{R}^{d}$,
	\begin{equation}
		\label{eq:ap_ftk2}
		\|\partial_{t}\nabla f_{t}^{k}(\bm{x})\|\leq(M_{f,t}/\sigma^{4}(t))(\text{diam}(\mathcal{X})+\|x\|).
	\end{equation}
	Substituting \eqref{eq:ap_ftk} and \eqref{eq:ap_ftk2} into \eqref{eq:ap_ptptn}, we get that, for any $x\in\mathbb{R}^{d}$,
	\begin{equation}
		\nonumber
		\begin{aligned}
			\|\partial_{t}\nabla\log\bar{p}_{t}^{N}(\bm{x})\|&\leq(M_{f,t}/\sigma^{4}(t))(\text{diam}(\mathcal{X})+\|x\|)+(M_{f,t}/\sigma^{6}(t))\text{diam}(\mathcal{X})^{2}(\text{diam}(\mathcal{X})+\|x\|)\\
			&\leq(M_{f,t}/\sigma^{6}(t))(\bar{\lambda}+\text{diam}(\mathcal{X})^{2})(\text{diam}(\mathcal{X})+\|x\|).
		\end{aligned}
	\end{equation}
	Using that $\lim_{N\rightarrow+\infty}\partial_{t}\nabla\log \pi_{t}^{N}(x_{t})=\partial_{t}\nabla\log \pi_{t}$, we have
	\begin{equation}
		\nonumber
		\|\partial_{t}\nabla\log \pi_{t}(\bm{x})\|\leq (M_{f,t}/\sigma^{6}(t))(\bar{\lambda}+\text{diam}(\mathcal{X})^{2})(\text{diam}(\mathcal{X})+\|x\|).
	\end{equation}
\end{proof}

This lemma provides crucial quantitative controls on the forward density \(\pi_t\), which is essential for establishing stability of the reverse SDE. The time derivative bound controls how quickly the score changes with time, which is necessary for analyzing the local truncation error of time discretization.

\subsection{Properties of the Reverse Process}
\label{ap:sub2}

In order to control the error introduced by the discretization and the score estimation, we firstly obtain a uniform control of the moments of the reverse processes. 

Consider the reverse diffusion process:
\begin{equation}
	\nonumber
	d\bm{x}_t^{\leftarrow} = \left\{a(t)\bm{x}_t^{\leftarrow}-b^2(t)\nabla_{\bm{x}_t^{\leftarrow}} \log \pi_t(\bm{x}_t^{\leftarrow})\right\}dt + b(t)d\bar{\bm{B}}_t .
\end{equation}
Replacing the score function as the Moreau score in \eqref{eq:ap_moreau_score} and using the proximal network $\bm{\phi}_{\bm{\theta}^*}$ to approximate the proximal term $\text{Prox}_{U}^{\sigma^2(t)/\mu^2(t)}(\bm{x}_t/\mu(t))$, we obtain the following SDE:
\begin{equation}
	\label{eq:ap_reversep}
	d\hat{\bm{x}}_t^{\leftarrow} = \left\{\left[a(t)+\frac{b^2(t)}{\sigma^2(t)}\right]\hat{\bm{x}}_t^{\leftarrow}-\frac{b^2(t)\mu(t)}{\sigma^2(t)}\bm{P}_t\right\}dt + b(t)d\bar{\bm{B}}_t,
\end{equation}
where $\lambda(t)=\sigma^2(t)/\mu^2(t)$ and 
\begin{equation}
	\nonumber
	\bm{P}_t=\bm{\phi}_{\bm{\theta}^*}\left(\frac{\bm{x}_{t}}{\mu(t)}-\lambda(t)\beta \nabla f(\frac{\bm{x}_{t}}{\mu(t)}),\lambda(t)\right).
\end{equation}
Denote the approximated score
\begin{equation}
	\label{eq:ap_vt}
	\bm{v}_t = \frac{\mu(t)\bm{\phi}_{\bm{\theta}^*}(\frac{\bm{x}_t}{\mu(t)}-\lambda(t)\beta \nabla_{\bm{x}} f(\frac{\bm{x}_t}{\mu(t)}),\lambda(t))-\bm{x}_t}{\mu^2(t)\lambda(t)}.
\end{equation}
Then we have
\begin{equation}
	\label{eq:ap_reversev}
	d\hat{\bm{x}}_t^{\leftarrow} = \left\{a(t)\hat{\bm{x}}_t^{\leftarrow}-b^2(t)\bm{v}_t\right\}dt + b(t)d\bar{\bm{B}}_t .
\end{equation}

From Proposition \ref{pro:expoint}, the interpolation process on the interval $t\in\left[t_k,t_{k+1}\right]$ is given by
\begin{equation}
	\nonumber
	d\bar{\bm{x}}_t^{EI} = \left\{-\left[a(T-t)+\frac{b^2(T-t)}{\sigma^2(T-t)}\right]\bar{\bm{x}}_t^{EI}+\frac{b^2(T-t)\mu(T-t)}{\sigma^2(T-t)}\bm{P}_k\right\}dt + b(T-t)d\bm{B}_t,
\end{equation}
where 
\begin{equation}
	\nonumber
	\bm{P}_k=\bm{\phi}_{\bm{\theta}^*}\left(\frac{\bar{\bm{x}}_{t_k}^{EI}}{\mu(T-t_k)}-\lambda(T-t_k)\beta \nabla f(\frac{\bar{\bm{x}}_{t_k}^{EI}}{\mu(T-t_k)}),\lambda(T-t_k)\right).
\end{equation}
And the corresponding iteration form is given by
\begin{equation}
	\nonumber
	\bar{\bm{x}}_{t_{k+1}}^{EI} = \alpha_{1,k}\bar{\bm{x}}_{t_k}^{EI}+ \alpha_{2,k}\bm{P}_k + \alpha_{3,k} \bm{\xi}_k,\quad  \bm{\xi}_k\sim\mathcal{N}(\bm{0},I).
\end{equation}
For simplicity, we denote $\bar{\bm{x}}_{t_{k}}^{EI}$ as $\bm{y}_{k}$, and then we have
\begin{equation}
	\label{eq:ap_sampley}
	\bm{y}_{k+1} = \alpha_{1,k}\bm{y}_{k}+ \alpha_{2,k} \bm{P}_k + \alpha_{3,k}\bm{\xi}_k,\quad  \bm{\xi}_k\sim\mathcal{N}(\bm{0},I).
\end{equation}
Denote the distribution of $\bm{y}_{k}$ as $\pi_{\infty}R_k$, where we choose $\bm{y}_{0}\sim\pi_{\infty}$ and $R_k$ the transition kernel associated with $\bm{y}_{k}|\bm{y}_{0}$. The following lemma provides uniform bounds for the second moments of $\hat{\bm{x}}_t^{\leftarrow}$, $\bm{y}_k$ and for the increment of the interpolant, depending on the score estimation error introduced in Theorem~\ref{thm:score_est}.
\begin{lemma}[Moment bounds for the reverse process]
	\normalfont
	\label{lem:reverse}
	Under Assumption \ref{assumption1} and \ref{assumption2}, let $\hat{\bm{x}}_t^{\leftarrow}$ satisfy~\eqref{eq:ap_reversev} and let $\bm{y}_k$ be defined by~\eqref{eq:ap_sampley} with $\bm{y}_0\sim\pi_{\infty}$. Then the following hold
	\begin{itemize}
		\item[1.] (Control of the approximated process) For sufficiently small $M\leq 1/6$, and for any $t \in [0, T]$,
		\begin{equation}
			\label{eq:lemma_D4}
			\mathbb{E}[\|\hat{\bm{x}}_t^{\leftarrow}\|^2]\leq 12M_{\lambda}(M + \text{diam}(\mathcal{X}))^2 + 2 \bar{\lambda}d .
		\end{equation}
		
		\item[2.] (Control of the sampling process) For any $k=0,1,\ldots, K$, we get
		\begin{equation}
			\label{eq:lemma_D5}
			\mathbb{E}[\|\bm{y}_k\|^2] \leq R_k := \bar{\lambda}d + B\left(1/A+\delta\right),
		\end{equation}
		where
		\begin{equation}
			\nonumber
			\begin{aligned}
				A =& \frac{1}{2\ln 2}\left(1+\frac{\gamma_k}{8\ln 2}\right)+ (2M+\eta)(1-\frac{\gamma_k}{4\ln 2})\exp\left\{M_{\mu}\delta\right\}M_{\lambda}+\exp\left\{2M_{\mu}\delta\right\}M_{\lambda}^2\gamma_k \left(2M^2+1+(2M+\eta)\bar{\lambda}\right) \\
				B =& \exp\left\{M_{\mu}\delta\right\}M_{\lambda}\frac{(M + \text{diam}(\mathcal{X}))^2}{\eta} \left(1-\frac{\gamma_k}{4\ln 2}\right) + \exp\left\{M_{\mu}\delta\right\}M_{\lambda}\gamma_k(2\eta+\bar{\lambda}) + \bar{\lambda}M_{\lambda}d.
			\end{aligned}
		\end{equation}
		
		\item[3.] (Control on the interpolation process) For any $k\in {0,\ldots,K-1}$ and $t\in[t_k,t_{k+1}]$, we have
		\begin{equation}
			\label{eq:lemma_D6}
			\mathbb{E}[\|\bar{\bm{x}}_{t}^{EI}-\bar{\bm{x}}_{t_k}^{EI}\|^2] \leq \left(A^{\prime}+B^{\prime}\right)\gamma_k,
		\end{equation}
		where 
		\begin{equation}
			\nonumber
			\begin{aligned}
				&A^{\prime} = \left\{\frac{\gamma_k}{(4\ln 2)^2} -\frac{2M+\eta}{4\ln 2}\exp\left\{M_{\mu}\delta\right\}M_{\lambda} + \exp\left\{2M_{\mu}\delta\right\}M_{\lambda}^2\gamma_k(2M^2+1+(2M+\eta)\bar{\lambda})\right\}R_k\\
				&B^{\prime} = -\frac{\gamma_k}{4\ln 2}\exp\left\{M_{\mu}\delta\right\}M_{\lambda}\frac{(M + \text{diam}(\mathcal{X}))^2}{\eta} + \exp\left\{2M_{\mu}\delta\right\}M_{\lambda}^2\gamma_k \frac{(2\eta+\bar{\lambda})(M + \text{diam}(\mathcal{X}))^2}{\eta} + \bar{\lambda}M_{\lambda}d.
			\end{aligned}
		\end{equation}
	\end{itemize}
\end{lemma}

\begin{proof}
	We prove the three estimates separately.
	
	\noindent\textbf{1. Control of the approximated process $\hat{\bm{x}}_t^{\leftarrow}$.}
	
	Let $\hat{\bm{x}}_t^{\leftarrow}$ satisfy the following reverse-time SDE:
	\begin{equation}
		\nonumber
		d\hat{\bm{x}}_t^{\leftarrow} = \left\{-a(T-t)\hat{\bm{x}}_t^{\leftarrow}+b^2(T-t)\bm{v}_{T-t}\right\}dt + b(T-t)d\bm{B}_t ,
	\end{equation}
	where the approximated score
	\begin{equation}
		\nonumber
		\bm{v}_t = \frac{\mu(t)\bm{\phi}_{\bm{\theta}^*}(\frac{\bm{x}_t}{\mu(t)}-\lambda(t)\beta \nabla_{\bm{x}} f(\frac{\bm{x}_t}{\mu(t)}),\lambda(t))-\bm{x}_t}{\mu^2(t)\lambda(t)}.
	\end{equation}
	
	By the results in (\ref{eq:lemma_c1}), we have that for any $t\in[0,T]$, $\mathbb{E}[\|\hat{\bm{x}}_t^{\leftarrow}\|^{2}]<+\infty$. Hence, apply It\^{o}'s lemma \cite{ito1944109} to $f(\bm{x}) = \frac{1}{2}\|\bm{x}\|^2$, we have 
	\begin{equation}
		\label{eq:ap_itox2}
		\begin{aligned}
			d\Bigl(\frac{1}{2}\|\hat{\bm{x}}_t^{\leftarrow}\|^2\Bigr)
			&= \nabla f(\hat{\bm{x}}_t^{\leftarrow}) \cdot d\hat{\bm{x}}_t^{\leftarrow} + \frac{1}{2} \operatorname{Tr}\bigl( b^2(T-t)\nabla^2 f(\hat{\bm{x}}_t^{\leftarrow}) \bigr) dt= \hat{\bm{x}}_t^{\leftarrow} \cdot d\hat{\bm{x}}_t^{\leftarrow} + \frac{1}{2} \operatorname{Tr}\bigl( b^2(T-t)I \bigr) dt \\
			&= \hat{\bm{x}}_t^{\leftarrow} \cdot \left\{-a(T-t)\hat{\bm{x}}_t^{\leftarrow}+b^2(T-t)\bm{v}_{T-t}\right\}dt + \hat{\bm{x}}_t^{\leftarrow} \cdot b(T-t)d\bm{B}_t + \frac{d}{2}b^2(T-t) dt \\
			&= \left\{-a(T-t)\|\hat{\bm{x}}_t^{\leftarrow}\|^2 +b^2(T-t)\langle \bm{v}_{T-t}, \hat{\bm{x}}_t^{\leftarrow}\rangle + \frac{d}{2}b^2(T-t) \right\} dt + b(T-t) \langle \hat{\bm{x}}_t^{\leftarrow}, d\bm{B}_t \rangle.
		\end{aligned}
	\end{equation}
	Taking expectations on \eqref{eq:ap_itox2}, we define $u_t = \frac{1}{2}\mathbb{E}[\|\hat{\bm{x}}_t^{\leftarrow}\|^2]$. Then $u_0 = \frac{d\bar{\lambda}}{2}$ and 
	\begin{equation}
		\label{eq:ap_itox3}
		\frac{d u_t}{dt} = \left\{-a(T-t)\mathbb{E}[\|\hat{\bm{x}}_t^{\leftarrow}\|^2] +b^2(T-t)\mathbb{E}[\langle \bm{v}_{T-t}, \hat{\bm{x}}_t^{\leftarrow}\rangle ] + \frac{d}{2}b^2(T-t) \right\}.
	\end{equation}
	Now we bound the inner product term in \eqref{eq:ap_itox3}. Denote $\tau  = T-t$, we have
	\begin{equation}
		\nonumber
		\langle \bm{v}_{\tau}, \hat{\bm{x}}_t^{\leftarrow} \rangle = \langle \bm{v}_{\tau} - \nabla \log p_{\tau }(\hat{\bm{x}}_t^{\leftarrow}), \hat{\bm{x}}_t^{\leftarrow} \rangle + \langle \nabla \log p_{\tau }(\hat{\bm{x}}_t^{\leftarrow}), \hat{\bm{x}}_t^{\leftarrow} \rangle.
	\end{equation}
	From the result in Appendix \ref{ap:score} and Theorem \ref{thm:score_est}, we have the score approximation error
	\begin{equation}
		\nonumber
		\|\nabla_{\bm{x}_t} \log p_t(\bm{x}_t)-\bm{v}_t\|\leq  \frac{M\left(1 + \|\bm{x}_t\|\right)}{\sigma^2(t)}.
	\end{equation}
	Then, with Cauchy–Schwarz and results in (\ref{eq:lemma_c1}), we have
	\begin{equation}
		\label{eq:ap_vsum1}
		\begin{aligned}
			\langle \bm{v}_{\tau} - \nabla \log p_{\tau }(\hat{\bm{x}}_t^{\leftarrow}), \hat{\bm{x}}_t^{\leftarrow} \rangle &\leq \| \bm{v}_{\tau} - \nabla \log p_{\tau }(\hat{\bm{x}}_t^{\leftarrow}) \| \cdot \|\hat{\bm{x}}_t^{\leftarrow}\| \\
			&\leq \frac{M(1 + \|\hat{\bm{x}}_t^{\leftarrow}\|)}{\sigma^2(\tau)} \cdot \|\hat{\bm{x}}_t^{\leftarrow}\| \leq \frac{M \|\hat{\bm{x}}_t^{\leftarrow}\|}{\sigma^2(\tau)} + \frac{M \|\hat{\bm{x}}_t^{\leftarrow}\|^2}{\sigma^2(\tau)},
		\end{aligned}
	\end{equation}
	and 
	\begin{equation}
		\label{eq:ap_vsum2}
		\langle \nabla \log p_{\tau }(\hat{\bm{x}}_t^{\leftarrow}), \hat{\bm{x}}_t^{\leftarrow} \rangle \leq - \frac{\|\hat{\bm{x}}_t^{\leftarrow}\|^2}{\sigma^2(\tau)} + \frac{\mu(\tau) \text{diam}(\mathcal{X}) \|\hat{\bm{x}}_t^{\leftarrow}\|}{\sigma^2(\tau)}.
	\end{equation}
	Summing \eqref{eq:ap_vsum1} and \eqref{eq:ap_vsum2}, we have
	\begin{equation}
		\label{eq:ap_vsum3}
		2\langle \bm{v}_{\tau}, \hat{\bm{x}}_t^{\leftarrow} \rangle \leq 2(M-1)\frac{\|\hat{\bm{x}}_t^{\leftarrow}\|^2}{\sigma^2(\tau)} + 2\bigl(M + \mu(\tau) \text{diam}(\mathcal{X})\bigr) \frac{\|\hat{\bm{x}}_t^{\leftarrow}\|}{\sigma^2(\tau)}.
	\end{equation}
	Apply the inequality $2ab \leq a^2\eta  + b^2/\eta$ with $a = \|\hat{\bm{x}}_t^{\leftarrow}\|/\sigma(\tau)$, $b = (M + \mu(\tau) \text{diam}(\mathcal{X}))/\sigma(\tau)$, for any $\eta > 0$, we get
	\begin{equation}
		\label{eq:ap_vsum4}
		2\bigl(M + \mu(\tau) \text{diam}(\mathcal{X})\bigr) \frac{\|\hat{\bm{x}}_t^{\leftarrow}\|}{\sigma^2(\tau)} \leq \eta \frac{\|\hat{\bm{x}}_t^{\leftarrow}\|^2}{\sigma^2(\tau)} + \frac{1}{\eta} \frac{(M + \mu(\tau) \text{diam}(\mathcal{X}))^2}{\sigma^2(\tau)}.
	\end{equation}
	Substituting \eqref{eq:ap_vsum4} into \eqref{eq:ap_vsum3}, we have
	\begin{equation}
		\label{eq:ap_vsum5}
		2\langle \bm{v}_{\tau}, \hat{\bm{x}}_t^{\leftarrow} \rangle \leq (2M-2+\eta)\frac{\|\hat{\bm{x}}_t^{\leftarrow}\|^2}{\sigma^2(\tau)} + \frac{1}{\eta} \frac{(M + \mu(\tau) \text{diam}(\mathcal{X}))^2}{\sigma^2(\tau)}.
	\end{equation}
	Plug into $du/dt$ in \eqref{eq:ap_itox3}, we have
	\begin{equation}
		\begin{aligned}
			\frac{d u_t}{dt} &\leq \left\{-2a(\tau)u_t +(2M-2+\eta)\frac{b^2(\tau)}{\sigma^2(\tau)}u_t +\frac{(M + \mu(\tau) \text{diam}(\mathcal{X}))^2}{2\eta} \frac{b^2(\tau)}{\sigma^2(\tau)} + \frac{d}{2}b^2(\tau) \right\} \\
			&= \left((2M-2+\eta)\frac{b^2(\tau)}{\sigma^2(\tau)}-2a(\tau)\right)u_t + \left(\frac{(M + \mu(\tau) \text{diam}(\mathcal{X}))^2}{2\eta} \frac{b^2(\tau)}{\sigma^2(\tau)} + \frac{d}{2}b^2(\tau)\right).
		\end{aligned}
	\end{equation}
	For simplicity, define
	\begin{equation}
		\nonumber
		A_t = \left((2M-2+\eta)\frac{b^2(\tau)}{\sigma^2(\tau)}-2a(\tau)\right), B_t = \left(\frac{(M + \mu(\tau) \text{diam}(\mathcal{X}))^2}{2\eta} \frac{b^2(\tau)}{\sigma^2(\tau)} + \frac{d}{2}b^2(\tau)\right).
	\end{equation}
	Then 
	\begin{equation}
		\label{eq:ap_dut}
		\frac{d u_t}{dt}\leq A_t u_t + B_t.
	\end{equation}
	Note that 
	\begin{equation}
		\nonumber
		\frac{a(\tau)\sigma^2(\tau)}{b^2(\tau)} = \frac{\partial_t \mu(\tau)}{\mu(\tau)}\frac{\lambda(\tau)}{\partial_t \lambda(\tau)} = \frac{\partial_t \log \mu(\tau)}{\partial_t \log \lambda(\tau)}\geq \frac{-M_\mu}{m_\lambda},
	\end{equation}
	then we have
	\begin{equation}
		\label{eq:ap_at}
		A_t = \left((2M-2+\eta)\frac{b^2(\tau)}{\sigma^2(\tau)}-2a(\tau)\right)\leq (2M-2+\eta)m_\lambda+2M_\mu .
	\end{equation}
	We choose $2M+\eta\leq 1-2M_{\mu}/m_{\lambda}$, then we have $A_t\leq -m_{\lambda}$. Applying Grönwall’s lemma \cite{fontaine2021convergence}, we have
	\begin{equation}
		\nonumber
		u_t\leq u_0 + \frac{M_{\lambda}}{m_{\lambda}} \frac{(M + \text{diam}(\mathcal{X}))^2+\eta \bar{\lambda} d}{2\eta}.
	\end{equation}
	Further, for sufficiently small $M$ such that $M\leq 1/6$, we choose $\eta = 1/6$, then $A_t\leq -1/2$ and
	\begin{equation}
		\label{eq:ap_ut2}
		u_t\leq 6M_{\lambda}(M + \text{diam}(\mathcal{X}))^2 + \bar{\lambda}d .
	\end{equation}
	Substituting $u_t = \frac{1}{2}\mathbb{E}[\|\hat{\bm{x}}_t^{\leftarrow}\|^2]$ into \eqref{eq:ap_ut2} gives
	\begin{equation}
		\nonumber
		\mathbb{E}[\|\hat{\bm{x}}_t^{\leftarrow}\|^2]\leq 12M_{\lambda}(M + \text{diam}(\mathcal{X}))^2 + 2 \bar{\lambda}d .
	\end{equation}
	
	\noindent\textbf{2. Control of the sampling process $\bm{y}_k$.}
	
	Consider the iteration form
	\begin{equation}
		\nonumber
		\bm{y}_{k+1} = \alpha_{1,k}\bm{y}_{k}+ \alpha_{2,k} \bm{P}_k + \alpha_{3,k} \bm{\xi}_k, \quad k=1,2,\ldots,K.
	\end{equation}
	From the results in \eqref{eq:ap_vsum5}, we get 
	\begin{equation}
		\nonumber
		\langle \bm{v}_{\tau}, \bm{x}_t \rangle \leq (2M-2+\eta)\frac{\|\bm{x}_t\|^2}{2\sigma^2(\tau)} + \frac{1}{\eta} \frac{(M + \mu(\tau) \text{diam}(\mathcal{X}))^2}{2\sigma^2(\tau)}.
	\end{equation}
	Recall the definition of $\bm{v}_t$ in \eqref{eq:ap_vt}, we have
	\begin{equation}
		\label{eq:ap_vtxt}
		\langle \bm{v}_{\tau}, \bm{x}_t \rangle = \frac{\mu(\tau)}{\sigma^2(\tau)} \langle \bm{P}_{\tau}, \bm{x}_t \rangle - \frac{1}{\sigma^2(\tau)} \langle \bm{x}_t, \bm{x}_t \rangle.
	\end{equation}
	Substituting \eqref{eq:ap_vtxt} into \eqref{eq:ap_vsum5}, we have
	\begin{equation}
		\label{eq:ap_ptxt}
		\begin{aligned}
			\langle \bm{P}_{\tau}, \bm{x}_t \rangle = \frac{\sigma^2(\tau)}{\mu(\tau)}\langle \bm{v}_{\tau}, \bm{x}_t \rangle + \frac{1}{\mu(\tau)}\langle \bm{x}_t, \bm{x}_t \rangle \leq \frac{2M+\eta}{2\mu(\tau)} \|\bm{x}_t\|^2+ \frac{1}{\eta} \frac{(M + \mu(\tau) \text{diam}(\mathcal{X}))^2}{2\mu(\tau)}.
		\end{aligned}
	\end{equation}
	Next, for the norm of $\bm{v}_{\tau}$, using the result in Theorem \ref{thm:score_est} and Lemma \ref{lem:forward}, we get
	\begin{equation}
		\begin{aligned}
			\|\bm{v}_{\tau}\|^2\leq& \|\bm{v}_{\tau}-\nabla \log p_{\tau}(\bm{x}_t)\|^2 + \|\nabla \log p_{\tau}(\bm{x}_t)\|^2 \leq \frac{M^2(1+\|\bm{x}_t\|)^2}{\sigma^4(\tau)} + \frac{2\|\bm{x}_t\|^2}{\sigma^4(\tau)} + \frac{2\mu^2(\tau) \text{diam}(\mathcal{X})^2}{\sigma^4(\tau)} \\
			\leq& \frac{(2M^2+2)}{\sigma^4(\tau)}\|\bm{x}_t\|^2 + \frac{2M^2+2\mu^2(\tau) \text{diam}(\mathcal{X})^2}{\sigma^4(\tau)}.
		\end{aligned}
	\end{equation}
	And substituting $\bm{v}_t$ as $\bm{P}_t$, we have
	\begin{equation}
		\|\bm{v}_{\tau}\|^2 = \|\frac{\mu(\tau)\bm{P}_{\tau} -\bm{x}_t}{\sigma^2(\tau)}\|^2 = \frac{\mu^2(\tau) }{\sigma^4(\tau)}\|\bm{P}_{\tau}\|^2 - \frac{2\mu(\tau)}{\sigma^2(\tau)} \langle \bm{P}_{\tau}, \bm{x}_t \rangle + \frac{\|\bm{x}_t\|^2}{\sigma^4(\tau)}.
	\end{equation}
	Then we have
	\begin{equation}
		\label{eq:ap_pt2}
		\begin{aligned}
			\|\bm{P}_{\tau}\|^2 =& \frac{\sigma^4(\tau)}{\mu^2(\tau)}\|\bm{v}_{\tau}\|^2 + \frac{2\sigma^2(\tau)}{\mu(\tau)}\langle \bm{P}_{\tau}, \bm{x}_t \rangle - \frac{\|\bm{x}_t\|^2}{\mu^2(\tau)}\\
			\leq & \frac{(2M^2+2)}{\mu^2(\tau)}\|\bm{x}_t\|^2 + \frac{2M^2}{\mu^2(\tau)}+\frac{2 \text{diam}(\mathcal{X})^2}{\mu^2(\tau)} + \frac{(2M+\eta)\sigma^2(\tau)}{\mu^2(\tau)} \|\bm{x}_t\|^2+ \frac{\sigma^2(\tau)(M + \mu(\tau) \text{diam}(\mathcal{X}))^2}{\eta\mu^2(\tau)}- \frac{\|\bm{x}_t\|^2}{\mu^2(\tau)}\\
			=& \frac{(2M^2+1)+(2M+\eta)\sigma^2(\tau)}{\mu^2(\tau)}\|\bm{x}_t\|^2 + \frac{(2M^2 +2\text{diam}(\mathcal{X})^2)}{\mu^2(\tau)} + \frac{\sigma^2(\tau)(M + \mu(\tau) \text{diam}(\mathcal{X}))^2}{\eta\mu^2(\tau)}.
		\end{aligned}
	\end{equation}
	Now, we compute the expectation of $\|\bm{y}_{k+1}\|^2$. Since $\bm{\xi}_k$ is independent of $\bm{y}_k$ and has mean zero, we have
	\begin{equation}
		\label{eq:ap_eyk2}
		\begin{aligned}
			\mathbb{E}[\|\bm{y}_{k+1}\|^2] =& \mathbb{E}\Bigl[ \|\alpha_{1,k}\bm{y}_{k}+ \alpha_{2,k} \bm{P}_k\|^2 \Bigr] + \alpha_{3,k}^2 d \\
			=& \alpha_{1,k}^2 \mathbb{E}[\|\bm{y}_{k}\|^2] + 2 \alpha_{1,k}\alpha_{2,k}\mathbb{E}[ \langle \bm{P}_{k}, \bm{y}_k \rangle] + \alpha_{2,k}^2 \mathbb{E}[\|\bm{P}_{k}\|^2]+ \alpha_{3,k}^2 d.
		\end{aligned}
	\end{equation}
	Substituting \eqref{eq:ap_ptxt} and \eqref{eq:ap_pt2} into \eqref{eq:ap_eyk2}, denote $\mu_k = \mu(T-t_k), \sigma_k = \sigma(T-t_k)$, we have the upper bound
	\begin{equation}
		\label{eq:ap_eyk22}
		\begin{aligned}
			\mathbb{E}[\|\bm{y}_{k+1}\|^2] \leq& \alpha_{1,k}^2 \mathbb{E}[\|\bm{y}_{k}\|^2] + 2 \alpha_{1,k}\alpha_{2,k}\left(\frac{2M+\eta}{2\mu_k} \mathbb{E}[\|\bm{y}_k\|^2]+ \frac{1}{\eta} \frac{(M + \mu_k \text{diam}(\mathcal{X}))^2}{2\mu_k}\right) + \\
			& \alpha_{2,k}^2 \left(\frac{2M^2+1+(2M+\eta)\sigma^2_k}{\mu^2_k}\mathbb{E}[\|\bm{y}_{k}\|^2] + \frac{2M^2 +2\text{diam}(\mathcal{X})^2}{\mu^2_k} + \frac{\sigma^2_k(M + \mu_k \text{diam}(\mathcal{X}))^2}{\eta\mu^2_k}\right)+ \alpha_{3,k}^2 d\\
			\leq& \left\{\alpha_{1,k}^2+ \alpha_{1,k}\alpha_{2,k}\frac{2M+\eta}{\mu_k}+ \alpha_{2,k}^2 \frac{2M^2+1+(2M+\eta)\sigma^2_k}{\mu^2_k} \right\}\mathbb{E}[\|\bm{y}_{k}\|^2]+ \\
			& \alpha_{1,k}\alpha_{2,k}\frac{(M + \text{diam}(\mathcal{X}))^2}{\eta\mu_k} + \alpha_{2,k}^2 \left(\frac{2M^2 +2\text{diam}(\mathcal{X})^2}{\mu^2_k} + \frac{\sigma^2_k(M +  \text{diam}(\mathcal{X}))^2}{\eta\mu^2_k}\right) + \alpha_{3,k}^2 d.
		\end{aligned}
	\end{equation}
	Under Assumption \ref{assumption2}, we have
	\begin{equation}
		\begin{aligned}
			\alpha_{1,k} =& \frac{\lambda_{k+1}\mu_{k+1}}{\lambda_{k}\mu_{k}} \leq \exp\left\{(M_{\mu}-M_{\lambda})\gamma_k\right\} \\
			\alpha_{2,k} =& \mu_k \frac{\mu_{k+1}}{\mu_k}\left(1-\frac{\lambda_{k+1}}{\lambda_{k}}\right) \leq \mu_k \exp\left\{M_{\mu}\gamma_k\right\}\left(1-\exp\left\{-M_{\lambda}\gamma_k\right\}\right) \\
			\alpha_{3,k} =& \mu_{k+1}\sqrt{\lambda_{k+1}}\sqrt{1-\frac{\lambda_{k+1}}{\lambda_{k}}} \leq \sqrt{\lambda_{k+1}}\sqrt{1-\exp\left\{-M_{\lambda}\gamma_k\right\}}.
		\end{aligned}
	\end{equation}
	For $\gamma_k\leq \delta\leq 2\ln 2$, we have
	\begin{equation}
		\label{eq:ap_eyk23}
		\begin{aligned}
			\alpha_{1,k} \leq 1-\frac{1}{4\ln 2}\gamma_k,\quad \alpha_{2,k}/\mu_k \leq \exp\left\{M_{\mu}\delta\right\}M_{\lambda}\gamma_k,\quad \alpha_{3,k}^2 \leq \bar{\lambda}M_{\lambda}\gamma_k.
		\end{aligned}
	\end{equation}
	Substituting \eqref{eq:ap_eyk23} into \eqref{eq:ap_eyk22}, we have
	\begin{equation}
		\nonumber
		\begin{aligned}
			\mathbb{E}[\|\bm{y}_{k+1}\|^2] \leq& \{1-\frac{1}{2\ln 2}\gamma_k+\frac{1}{(4\ln 2)^2}\gamma_k^2+ (2M+\eta)(1-\frac{1}{4\ln 2}\gamma_k)\exp\left\{M_{\mu}\delta\right\}M_{\lambda}\gamma_k+ \\ &\exp\left\{2M_{\mu}\delta\right\}M_{\lambda}^2\gamma_k^2 \left(2M^2+1+(2M+\eta)\bar{\lambda}\right) \}\mathbb{E}[\|\bm{y}_{k}\|^2]+ \\
			& (1-\frac{1}{4\ln 2}\gamma_k)\exp\left\{M_{\mu}\delta\right\}M_{\lambda}\gamma_k\frac{(M + \text{diam}(\mathcal{X}))^2}{\eta} + \\
			& \exp\left\{2M_{\mu}\delta\right\}M_{\lambda}^2\gamma_k^2 \left(2M^2 +2\text{diam}(\mathcal{X})^2 + \frac{\bar{\lambda}(M +  \text{diam}(\mathcal{X}))^2}{\eta}\right) + \bar{\lambda}M_{\lambda}d\gamma_k.
		\end{aligned}
	\end{equation}
	Denote 
	\begin{equation}
		\nonumber
		\begin{aligned}
			A =& \frac{1}{2\ln 2}\left(1+\frac{\gamma_k}{8\ln 2}\right)+ (2M+\eta)(1-\frac{\gamma_k}{4\ln 2})\exp\left\{M_{\mu}\delta\right\}M_{\lambda}+\exp\left\{2M_{\mu}\delta\right\}M_{\lambda}^2\gamma_k \left(2M^2+1+(2M+\eta)\bar{\lambda}\right) \\
			B =& \exp\left\{M_{\mu}\delta\right\}M_{\lambda}\frac{(M + \text{diam}(\mathcal{X}))^2}{\eta} \left(1-\frac{\gamma_k}{4\ln 2}\right) + \exp\left\{M_{\mu}\delta\right\}M_{\lambda}\gamma_k(2\eta+\bar{\lambda}) + \bar{\lambda}M_{\lambda}d.
		\end{aligned}
	\end{equation}
	Then we have 
	\begin{equation}
		\nonumber
		\mathbb{E}[\|\bm{y}_{k+1}\|^2] \leq \left(1-\gamma_k A\right)\mathbb{E}[\|\bm{y}_{k}\|^2] + \gamma_k B.
	\end{equation}
	By recursion, we have that, for any $k=0,1,\ldots,K$,
	\begin{equation}
		\nonumber
		\mathbb{E}[\|\bm{y}_{k+1}\|^2] \leq \bar{\lambda}d + B\left(1/A+\delta\right).
	\end{equation}
	Note that the same result holds for $(\bar{\bm{x}}_t^{EI})_{t\in[0,T]}$.
	
	\noindent\textbf{3. Control on the interpolation process.}
	
	Note that from the definition of $\bar{\bm{x}}_t^{EI}$, we have
	\begin{equation}
		\nonumber
		\bar{\bm{x}}_{t}^{EI} = \alpha_{1,t_k}\bar{\bm{x}}_{t_k}^{EI}+ \alpha_{2,t_k}\bm{P}_{t_k} + \alpha_{3,t_k} \bm{\xi}_k, \quad k=1,2,\ldots,K.
	\end{equation}
	Therefore,
	\begin{equation}
		\label{eq:ap_exttk}
		\begin{aligned}
			\mathbb{E}[\|\bar{\bm{x}}_{t}^{EI}-\bar{\bm{x}}_{t_k}^{EI}\|^2] =& \mathbb{E}\Bigl[ \|(\alpha_{1,t_k}-1)\bar{\bm{x}}_{t_k}^{EI}+ \alpha_{2,t_k}\bm{P}_{t_k}\|^2 \Bigr] + \alpha_{3,t_k}^2 d \\
			=& (\alpha_{1,t_k}-1)^2 \mathbb{E}[\|\bar{\bm{x}}_{t_k}^{EI}\|^2] + 2 (\alpha_{1,t_k}-1)\alpha_{2,t_k}\mathbb{E}[ \langle \bm{P}_{t_k}, \bar{\bm{x}}_{t_k}^{EI} \rangle] + \alpha_{2,t_k}^2 \mathbb{E}[\|\bm{P}_{t_k}\|^2]+ \alpha_{3,t_k}^2 d.
		\end{aligned}
	\end{equation}
	From the former results in \eqref{eq:ap_ptxt} and \eqref{eq:ap_pt2} , we have
	\begin{equation}
		\label{eq:ap_exttk2}
		\begin{aligned}
			&\mathbb{E}[ \langle \bm{P}_{t_k}, \bar{\bm{x}}_{t_k}^{EI} \rangle]\leq \frac{2M+\eta}{2\mu_k} \mathbb{E}[\|\bar{\bm{x}}_{t_k}^{EI}\|^2]+ \frac{1}{\eta} \frac{(M + \text{diam}(\mathcal{X}))^2}{2\mu_k} \\
			&\mathbb{E}[\|\bm{P}_{t_k}\|^2]\leq \frac{(2M^2+1)+(2M+\eta)\sigma^2_k}{\mu^2_k}\mathbb{E}[\|\bar{\bm{x}}_{t_k}^{EI}\|^2] + \frac{(2M^2 +2\text{diam}(\mathcal{X})^2)}{\mu^2_k} + \frac{\bar{\lambda}(M + \text{diam}(\mathcal{X}))^2}{\eta\mu^2_k}.
		\end{aligned}
	\end{equation}
	Substituting \eqref{eq:ap_exttk2} into \eqref{eq:ap_exttk}, we get
	\begin{equation}
		\label{eq:ap_exttk3}
		\begin{aligned}
			\mathbb{E}&[\|\bar{\bm{x}}_{t}^{EI}-\bar{\bm{x}}_{t_k}^{EI}\|^2] \leq \left\{(\alpha_{1,t_k}-1)^2+ (\alpha_{1,t_k}-1)\alpha_{2,t_k}\frac{2M+\eta}{\mu_k}+\alpha_{2,t_k}^2\frac{(2M^2+1)+(2M+\eta)\sigma^2_k}{\mu^2_k}\right\}\mathbb{E}[\|\bar{\bm{x}}_{t_k}^{EI}\|^2] \\
			&+ (\alpha_{1,t_k}-1)\alpha_{2,t_k}\frac{(M + \text{diam}(\mathcal{X}))^2}{\eta\mu_k} + \alpha_{2,t_k}^2 \left(\frac{(2M^2 +2\text{diam}(\mathcal{X})^2)}{\mu^2_k} + \frac{\bar{\lambda}(M + \text{diam}(\mathcal{X}))^2}{\eta\mu^2_k}\right)+ \alpha_{3,t_k}^2 d.
		\end{aligned}
	\end{equation}
	Similarly, denote $\tau = t-t_k\leq \gamma_k\leq \delta$, we have 
	\begin{equation}
		\nonumber
		\begin{aligned}
			\alpha_{1,t_k} =& \frac{\lambda(T-t)\mu(T-t)}{\lambda(T-t_{k})\mu(T-t_{k})} \leq \exp\left\{(M_{\mu}-M_{\lambda})\tau\right\} \\
			\alpha_{2,t_k} =& \mu(T-t)\left(1-\frac{\lambda(T-t)}{\lambda(T-t_{k})}\right) \leq \mu_k \exp\left\{M_{\mu}\tau\right\}\left(1-\exp\left\{-M_{\lambda}\tau\right\}\right) \\
			\alpha_{3,t_k} =& \mu(T-t)\sqrt{\lambda(T-t)}\sqrt{1-\frac{\lambda(T-t)}{\lambda(T-t_{k})}} \leq \sqrt{\lambda_{k+1}}\sqrt{1-\exp\left\{-M_{\lambda}\tau\right\}}.
		\end{aligned}
	\end{equation}
	Further, we get
	\begin{equation}
		\label{eq:ap_exttk4}
		\begin{aligned}
			\alpha_{1,t_k} \leq 1-\frac{1}{4\ln 2}\gamma_k,\quad \alpha_{2,t_k}/\mu_k \leq \exp\left\{M_{\mu}\delta\right\}M_{\lambda}\gamma_k,\quad \alpha_{3,t_k}^2 \leq \bar{\lambda}M_{\lambda}\gamma_k.
		\end{aligned}
	\end{equation}
	Therefore, substituting \eqref{eq:ap_exttk4} and the result in (\ref{eq:lemma_D5}) into \eqref{eq:ap_exttk3}, we have
	\begin{equation}
		\nonumber
		\begin{aligned}
			\mathbb{E}[\|\bar{\bm{x}}_{t}^{EI}-\bar{\bm{x}}_{t_k}^{EI}\|^2] \leq& \left\{\frac{\gamma_k^2}{(4\ln 2)^2} -\frac{2M+\eta}{4\ln 2}\exp\left\{M_{\mu}\delta\right\}M_{\lambda}\gamma_k + \exp\left\{2M_{\mu}\delta\right\}M_{\lambda}^2\gamma_k^2(2M^2+1+(2M+\eta)\bar{\lambda})\right\}R_k \\
			&-\frac{\gamma_k}{4\ln 2}\exp\left\{M_{\mu}\delta\right\}M_{\lambda}\gamma_k\frac{(M + \text{diam}(\mathcal{X}))^2}{\eta} + \exp\left\{2M_{\mu}\delta\right\}M_{\lambda}^2\gamma_k^2 \frac{(2\eta+\bar{\lambda})(M + \text{diam}(\mathcal{X}))^2}{\eta}.
		\end{aligned}
	\end{equation}
	Denote 
	\begin{equation}
		\nonumber
		\begin{aligned}
			&A^{\prime} = \left\{\frac{\gamma_k}{(4\ln 2)^2} -\frac{2M+\eta}{4\ln 2}\exp\left\{M_{\mu}\delta\right\}M_{\lambda} + \exp\left\{2M_{\mu}\delta\right\}M_{\lambda}^2\gamma_k(2M^2+1+(2M+\eta)\bar{\lambda})\right\}R_k\\
			&B^{\prime} = -\frac{\gamma_k}{4\ln 2}\exp\left\{M_{\mu}\delta\right\}M_{\lambda}\frac{(M + \text{diam}(\mathcal{X}))^2}{\eta} + \exp\left\{2M_{\mu}\delta\right\}M_{\lambda}^2\gamma_k \frac{(2\eta+\bar{\lambda})(M + \text{diam}(\mathcal{X}))^2}{\eta} + \bar{\lambda}M_{\lambda}d.
		\end{aligned}
	\end{equation}
	Then we have
	\begin{equation}
		\nonumber
		\mathbb{E}[\|\bar{\bm{x}}_{t}^{EI}-\bar{\bm{x}}_{t_k}^{EI}\|^2] \leq \left(A^{\prime}+B^{\prime}\right)\gamma_k.
	\end{equation}
	
\end{proof}

This lemma establishes uniform second-moment bounds for the reverse diffusion processes, which are crucial for proving stability and convergence of discrete-time samplers. The bounds guarantee that the reverse process does not explode in finite time, even when using an approximated score. Overall, these moment bounds are key technical tools for deriving non-asymptotic convergence for proximal diffusion sampling algorithms.

\subsection{Contraction of the Tangent Process}
\label{ap:sub3}

To control the error introduced by the discretization and score approximation, we analyze the tangent process associated with the reverse SDE. This process quantifies how small perturbations in the initial condition propagate along the reverse-time flow.

For the forward process $d\bm{x}_t^{\rightarrow} = \left\{a(t)\bm{x}_t^{\rightarrow}\right\}dt + b(t)d\bm{B}_t$, denote $\pi_{\infty}=\lim\limits_{t\to\infty} \text{Law}(\bm{x}_t^{\rightarrow})$. For any $\bm{x}\in\mathbb{R}^d$ and $s,t\in\left[0,T\right], s\leq t$, denote the reverse process as
\begin{equation}
	d\bm{y}_{s,t}^{\bm{x}} = \left\{-a(T-t)\bm{y}_{s,t}^{\bm{x}}+b^2(T-t)\nabla \log \pi_{T-t}(\bm{y}_{s,t}^{\bm{x}})\right\}dt + b(T-t)d\bm{B}_t, \quad \bm{y}_{s,s}^{\bm{x}} = \bm{x}.
\end{equation}
The approximated reverse process is defined by
\begin{equation}
	d\hat{\bm{y}}_{s,t}^{\bm{x}} = \left\{-a(T-t)\hat{\bm{y}}_{s,t}^{\bm{x}}+b^2(T-t)\bm{v}_{T-t}\right\}dt + b(T-t)d\bm{B}_t , \quad \hat{\bm{y}}_{s,s}^{\bm{x}} = \bm{x},
\end{equation}
where the approximated score
\begin{equation}
	\nonumber
	\bm{v}_{T-t} = \frac{\mu(T-t)\bm{\phi}_{\bm{\theta}^*}(\frac{\hat{\bm{y}}_{s,t}^{\bm{x}}}{\mu(T-t)}-\lambda(T-t)\beta \nabla f(\frac{\hat{\bm{y}}_{s,t}^{\bm{x}}}{\mu(T-t)}),\lambda(T-t))-\hat{\bm{y}}_{s,t}^{\bm{x}}}{\sigma^2(T-t)}.
\end{equation}
It can also be expressed as
\begin{equation}
	d\hat{\bm{y}}_{s,t}^{\bm{x}} = \left\{-\left[a(T-t)+\frac{b^2(T-t)}{\sigma^2(T-t)}\right]\hat{\bm{y}}_{s,t}^{\bm{x}}+\frac{b^2(T-t)\mu(T-t)}{\sigma^2(T-t)}\bm{P}_{T-t}\right\}dt + b(T-t)d\bm{B}_t,
\end{equation}
where $\lambda(t)=\sigma^2(t)/\mu^2(t)$ and 
\begin{equation}
	\nonumber
	\bm{P}_{T-t}=\bm{\phi}_{\bm{\theta}^*}\left(\frac{\hat{\bm{y}}_{s,t}^{\bm{x}}}{\mu(T-t)}-\lambda(T-t)\beta \nabla f(\frac{\hat{\bm{y}}_{s,t}^{\bm{x}}}{\mu(T-t)}),\lambda(T-t)\right).
\end{equation}

From Proposition \ref{pro:expoint}, the interpolation process on the interval $t\in\left[s_k,t_{k+1}\right]$ is defined by
\begin{equation}
	d\bar{\bm{y}}_{s,t}^{\bm{x}} = \left\{-\left[a(T-t)+\frac{b^2(T-t)}{\sigma^2(T-t)}\right]\bar{\bm{y}}_{s,t}^{\bm{x}}+\frac{b^2(T-t)\mu(T-t)}{\sigma^2(T-t)}\bm{P}_k\right\}dt + b(T-t)d\bm{B}_t, \bar{\bm{y}}_{s,s}^{\bm{x}} = \bm{x},
\end{equation}
where $s_k=\max\left\{s,t_k\right\}$ and
\begin{equation}
	\nonumber
	\bm{P}_k=\bm{\phi}_{\bm{\theta}^*}\left(\frac{\bar{\bm{y}}_{s,t}^{\bm{x}}}{\mu(T-t_k)}-\lambda(T-t_k)\beta \nabla f(\frac{\bar{\bm{y}}_{s,t}^{\bm{x}}}{\mu(T-t_k)}),\lambda(T-t_k)\right).
\end{equation}

We introduce the tangent process
\begin{equation}
	d\nabla\bm{y}_{s,t}^{\bm{x}} = \left\{-a(T-t)I+b^2(T-t)\nabla \log \pi_{T-t}(\bm{y}_{s,t}^{\bm{x}})\right\}\nabla\bm{y}_{s,t}^{\bm{x}}dt,
\end{equation}
where $\nabla\bm{y}_{s,s}^{\bm{x}} = I$. Then the bound on the approximation and discretization error relies on the following lemma.

\begin{lemma}[\cite{del2022backward}]
	\normalfont
	\label{lem:del}
	Under Assumption \ref{assumption1}, for any $\bm{x}\in\mathbb{R}^d$ and $s,t\in\left[0,T\right], s\leq t$, we have
	\begin{equation}
		\bm{y}_{s,t}^{\bm{x}}-\bar{\bm{y}}_{s,t}^{\bm{x}}=\int_{s}^{t} \left(\nabla\bm{y}_{u,t}^{\bar{\bm{y}}_{s,u}^{\bm{x}}}\right)^{\top} \Delta \bm{b}_u((\bar{\bm{y}}_{s,v}^{\bm{x}})_{v\in\left[s,t\right]}) du,
	\end{equation}
	where for any $u\in[s_k,t_{k+1}]$ and $(\bm{w}_{v})_{v\in\left[s,t\right]}$, the drift 
	\begin{equation}
		\begin{aligned}
			\bm{b}_u(\bm{w})=& -a(T-t)\bm{w}+b^2(T-t)\nabla \log \pi_{T-t}(\bm{w}), \\
			\bar{\bm{b}}_u(\bm{w})=&-\left[a(T-t)+\frac{b^2(T-t)}{\sigma^2(T-t)}\right]w+\frac{b^2(T-t)\mu(T-t)}{\sigma^2(T-t)}\bm{P}_k(\bm{w}),
		\end{aligned}
	\end{equation}
	and $\Delta \bm{b}_u(\bm{w})= \bm{b}_u(\bm{w})-\bar{\bm{b}}_u(\bm{w})$.
\end{lemma}

Then our goal is to control $\nabla\bm{y}_{s,t}^{\bm{x}}$ and $\Delta \bm{b}_s((\bar{\bm{y}}_{s,t}^{\bm{x}})_{t\in\left[s,T\right]})$ for any $\bm{x}\in\mathbb{R}^d$ and $s,t\in\left[0,T\right], s\leq t$.

\begin{lemma}
	\normalfont
	\label{lem:tangent}
	Under Assumption \ref{assumption1} and \ref{assumption2}, for any $s,t\in[0,T]$ and $\bm{x},\bm{x}^{\prime}\in\mathbb{R}^d$, denote 
	\begin{equation}
		\begin{aligned}
			&\bar{t}=T-\lambda^{-1}\left(\frac{\operatorname{diam}^2(\mathcal{X})M_{\lambda}}{m_{\lambda}-M_{\mu}}\right),\\
			&G_K = \bar{t}/2 - \left(m_{\lambda}-M_{\mu}\right)(T-\bar{t})\exp\left[M_{\lambda}(T-\bar{t})\right].
		\end{aligned}
	\end{equation}
	Then we have 
	\begin{equation}
		\|\nabla \bm{y}_{0,t}^{\bm{x}}\|^2 \leq
		\begin{cases}
			\exp\left[- t/2\right],\quad &t\leq\bar{t}, \\
			\exp\left[- G_K\right],\quad &t\geq\bar{t},
		\end{cases}
	\end{equation}
	and 
	\begin{equation}
		\mathcal{W}_1(\delta_{\bm{x}} Q_t, \delta_{\bm{x}^{\prime}} Q_t) \leq 
		\begin{cases}
			\exp\left[- t/2\right]\|\bm{x}-\bm{x}^{\prime}\|,\quad &t\leq\bar{t}, \\
			\exp\left[- G_K\right]\|\bm{x}-\bm{x}^{\prime}\|,\quad &t\geq\bar{t}.
		\end{cases}
	\end{equation}
\end{lemma}

\begin{proof}
	Firstly, we consider a simple case where the target distribution $\pi=\delta_0$, and then $\text{diam}(\mathcal{X})=0$. We will show that
	\begin{equation}
		\mathcal{W}_1(\delta_{\bm{x}} Q_t, \delta_{\bm{x}^{\prime}} Q_t) \leq \mathbb{E} \left[\|\bm{y}_{0,t}^{\bm{x}}-\bm{y}_{0,t}^{\bm{x}^{\prime}}\|\right]\leq \exp\left\{- (m_\lambda-M_{\mu}) t\right\}\|\bm{x}-\bm{x}^{\prime}\|.
	\end{equation}
	Especially, when $t = T$, the bound becomes
	\begin{equation}
		\mathcal{W}_1(\delta_{\bm{x}} Q_T, \delta_{\bm{x}^{\prime}} Q_T) \leq 0,
	\end{equation}
	implying finite-time convergence.
	
	Consider the backward process $(\bm{y}_{0,t}^{\bm{x}})$ with the associated semi-group $\left(Q_t\right)_{t\in[0,T]}$ and its associated tangent process
	\begin{equation}
		\nonumber
		d\nabla\bm{y}_{0,t}^{\bm{x}} = \left\{-a(T-t)I+b^2(T-t)\nabla \log \pi_t(\bm{y}_{0,t}^{\bm{x}})\right\}\nabla\bm{y}_{0,t}^{\bm{x}}dt, \quad \nabla\bm{y}_{0,0}^{\bm{x}} = I.
	\end{equation}
	Then, for any $t \in [0,T]$, we have
	\begin{equation}
		\label{eq:ap_yotxy}
		\|\bm{y}_{0,t}^{\bm{x}}-\bm{y}_{0,t}^{\bm{x}^{\prime}}\| \leq \int_0^1 \|\nabla \bm{y}_{0,t}^{\bm{z}_{r}}\| \, dr \cdot \|\bm{x}-\bm{x}^{\prime}\|,
	\end{equation}
	where $\bm{z}_{r} = r \bm{x}+(1-r)\bm{x}^{\prime}$. Using the results in Lemma \ref{lem:forward}, we have that, for any $\bm{z}$, 
	\begin{equation}
		\nonumber
		\langle\bm{z},\nabla^2\log \pi_t(\bm{y}_{0,t}^{\bm{x}})\bm{z}\rangle\leq-\frac{2\sigma^2(t)-\mu^2(t)\text{diam}(\mathcal{X})^2}{2\sigma^4(t)}\|\bm{z}\|^2 = -\frac{\|\bm{z}\|^2}{\sigma^2(t)}.
	\end{equation}
	Then we have
	\begin{equation}
		\nonumber
		\begin{aligned}
			\langle\bm{z},\left(-a(T-t)I+b^2(T-t)\nabla \log \pi_{T-t}(\bm{y}_{s,t}^{\bm{x}})\right)\bm{z}\rangle\leq& -a(T-t)\|\bm{z}\|^2 -\frac{b^2(T-t)}{\sigma^2(T-t)}\|\bm{z}\|^2\\
			=&-\left(a(T-t)+\frac{b^2(T-t)}{\sigma^2(T-t)}\right)\|\bm{z}\|^2.
		\end{aligned}
	\end{equation}
	And the evolution of $\|\nabla\bm{y}_{0,t}^{\bm{x}}\|^2$ can be bounded by
	\begin{equation}
		\label{eq:ap_ddty}
		\frac{d}{dt} \|\nabla \bm{y}_{0,t}^{\bm{x}}\|^2 \leq -2\left(a(T-t)+\frac{b^2(T-t)}{\sigma^2(T-t)}\right) \|\nabla\bm{y}_{0,t}^{\bm{x}}\|^2.
	\end{equation}
	Integrating \eqref{eq:ap_ddty} from $0$ to $t$ and using the initial condition $\nabla\bm{y}_{0,0}^{\bm{x}} = I$, we obtain
	\begin{equation}
		\label{eq:ap_ddty2}
		\|\nabla \bm{y}_{0,t}^{\bm{x}}\|^2 \leq \exp\left[ -2 \int_0^t \left(a(T-s)+\frac{b^2(T-s)}{\sigma^2(T-s)}\right) ds \right].
	\end{equation}
	Denote $\tau = T-t$, and recall that $\mu(\tau)=\exp\left\{\int_{0}^{\tau}a(s)ds\right\}$. We have $\frac{d}{d\tau} \log \mu(\tau) = a(\tau)$, and then
	\begin{equation}
		\nonumber
		\int_{T-t}^{T} a(\tau) d\tau = \log \mu(T)-\log \mu(T-t)=\log \frac{\mu(T)}{\mu(T-t)}.
	\end{equation}
	Further, note that $\sigma^2(\tau)=\mu^2(\tau)\int_{0}^{\tau}(b(s)/\mu(s))^2ds$. Denote $\lambda(\tau)=\int_{0}^{\tau}(b(s)/\mu(s))^2ds = \sigma^2(\tau)/\mu^2(\tau)$. Then we have
	\begin{equation}
		\nonumber
		\frac{b^2(\tau)}{\sigma^2(\tau)}=\frac{b^2(\tau)}{\mu^2(\tau)\lambda(\tau)}=\frac{1}{\lambda(\tau)}\frac{b^2(\tau)}{\mu^2(\tau)}.
	\end{equation}
	Also, we have $\lambda^{\prime}(\tau)=(b(\tau)/\mu(\tau))^2$, and then
	\begin{equation}
		\nonumber
		\frac{b^2(\tau)}{\sigma^2(\tau)}=\frac{1}{\lambda(\tau)}\frac{b^2(\tau)}{\mu^2(\tau)}=\frac{\lambda^{\prime}(\tau)}{\lambda(\tau)}=\frac{d}{d\tau} \log \lambda(\tau).
	\end{equation}
	So we have
	\begin{equation}
		\nonumber
		\int_{T-t}^{T} \frac{b^2(\tau)}{\sigma^2(\tau)} d\tau = \log \lambda(T)-\log \lambda(T-t) = \log \frac{\lambda(T)}{\lambda(T-t)}.
	\end{equation}
	Then with Assumption \ref{assumption2}, we get
	\begin{equation}
		\label{eq:ap_ddty3}
		\exp\left[ -\int_0^t \left(a(T-s)+\frac{b^2(T-s)}{\sigma^2(T-s)}\right) ds \right] = \frac{\lambda(T-t)}{\lambda(T)} \frac{\mu(T-t)}{\mu(T)}\leq\exp\left\{- (m_\lambda-M_{\mu}) t\right\}.
	\end{equation}
	Substituting \eqref{eq:ap_ddty3} into \eqref{eq:ap_ddty2}, we obtain
	\begin{equation}
		\|\nabla \bm{y}_{0,t}^{\bm{x}}\|^2 \leq \exp\left[ -2 \int_0^t \left(a(T-s)+\frac{b^2(T-s)}{\sigma^2(T-s)}\right) ds \right]\leq \exp\left\{- 2(m_\lambda-M_{\mu}) t\right\},
	\end{equation}
	and when $t\to T$, $\|\nabla \bm{y}_{0,t}^{\bm{x}}\|^2\to 0$.
	
	Substituting this into the tangent error \eqref{eq:ap_yotxy}, we obtain
	\begin{equation}
		\nonumber
		\|\bm{y}_{0,t}^{\bm{x}}-\bm{y}_{0,t}^{\bm{x}^{\prime}}\| \leq \int_0^1 \|\nabla \bm{y}_{0,t}^{\bm{z}_{r}}\| \, dr \cdot \|\bm{x}-\bm{x}^{\prime}\|\leq \exp\left\{- (m_\lambda-M_{\mu}) t\right\}\|\bm{x}-\bm{x}^{\prime}\|.
	\end{equation}
	By the definition of the 1-Wasserstein distance, we have
	\begin{equation}
		\nonumber
		\mathcal{W}_1(\delta_{\bm{x}} Q_t, \delta_{\bm{x}^{\prime}} Q_t) \leq \mathbb{E} \left[\|\bm{y}_{0,t}^{\bm{x}}-\bm{y}_{0,t}^{\bm{x}^{\prime}}\|\right]\leq \exp\left\{- (m_\lambda-M_{\mu}) t\right\}\|\bm{x}-\bm{x}^{\prime}\|.
	\end{equation}
	
	Especially, when $t = T$, the bound becomes $\|\nabla \bm{y}_{0,t}^{\bm{x}}\|^2 \leq0$ and then
	\begin{equation}
		\nonumber
		\mathcal{W}_1(\delta_{\bm{x}} Q_T, \delta_{\bm{x}^{\prime}} Q_T) \leq \mathbb{E} \left[\|\bm{y}_{0,t}^{\bm{x}}-\bm{y}_{0,t}^{\bm{x}^{\prime}}\|\right]\leq 0.
	\end{equation}
	Hence, $\mathcal{W}_1(\delta_{\bm{x}} Q_T, \delta_{\bm{x}^{\prime}} Q_T) = 0$ for all $\bm{x},\bm{x}^{\prime} \in \mathbb{R}^d$, implying that $\delta_{\bm{x}} Q_T=\delta_{\bm{x}^{\prime}} Q_T$. In particular, taking $\bm{x}^{\prime}\sim \delta_0$ and noting that $\pi = \delta_0$ is invariant under the forward-backward dynamics, we have
	\begin{equation}
		\nonumber
		\delta_{\bm{x}} Q_T = \pi.
	\end{equation}
	This holds for any $\bm{x} \in \mathbb{R}^d$,implying finite-time convergence.
	
	Next, we will control the growth of the tangent process for a general target $\pi$. For the tangent process
	\begin{equation}
		\nonumber
		d\nabla\bm{y}_{0,t}^{\bm{x}} = \left\{-a(T-t)I+b^2(T-t)\nabla \log \pi_t(\bm{y}_{0,t}^{\bm{x}})\right\}\nabla\bm{y}_{0,t}^{\bm{x}}dt, \quad \nabla\bm{y}_{0,0}^{\bm{x}} = I,
	\end{equation}
	using the results in Lemma \ref{lem:forward}, we have that, for any $\bm{z}\in\mathcal{X}$, 
	\begin{equation}
		\nonumber
		\langle\bm{z},\nabla^2\log \pi_t(\bm{y}_{0,t}^{\bm{x}})\bm{z}\rangle\leq-\frac{2\sigma^2(t)-\mu^2(t)\text{diam}(\mathcal{X})^2}{2\sigma^4(t)}\|\bm{z}\|^2 = -\frac{\|\bm{z}\|^2}{\sigma^2(t)}+\frac{\mu^2(t)\text{diam}(\mathcal{X})^2\|\bm{z}\|^2}{2\sigma^4(t)}.
	\end{equation}
	Then we have
	\begin{equation}
		\nonumber
		\begin{aligned}
			&\langle\bm{z},\left(-a(T-t)I+b^2(T-t)\nabla \log \pi_{T-t}(\bm{y}_{s,t}^{\bm{x}})\right)\bm{z}\rangle \\
			\leq& -a(T-t)\|\bm{z}\|^2 -\frac{b^2(T-t)}{\sigma^2(T-t)}\|\bm{z}\|^2 + \frac{b^2(T-t)\mu^2(T-t)\text{diam}(\mathcal{X})^2\|\bm{z}\|^2}{2\sigma^4(T-t)}\\
			=&-\left(a(T-t)+\frac{b^2(T-t)}{\sigma^2(T-t)}-\frac{b^2(T-t)\mu^2(T-t)\text{diam}(\mathcal{X})^2}{2\sigma^4(T-t)}\right)\|\bm{z}\|^2.
		\end{aligned}
	\end{equation}
	And the evolution of $\|\nabla\bm{y}_{0,t}^{\bm{x}}\|^2$ can be bounded by
	\begin{equation}
		\label{eq:ap_ddtyot2}
		\frac{d}{dt} \|\nabla \bm{y}_{0,t}^{\bm{x}}\|^2 \leq -2\left(a(T-t)+\frac{b^2(T-t)}{\sigma^2(T-t)}-\frac{b^2(T-t)\mu^2(T-t)\text{diam}(\mathcal{X})^2}{2\sigma^4(T-t)}\right) \|\nabla\bm{y}_{0,t}^{\bm{x}}\|^2.
	\end{equation}
	Integrating \eqref{eq:ap_ddtyot2} from $0$ to $t$ and using the initial condition $\nabla\bm{y}_{0,0}^{\bm{x}} = I$, we obtain
	\begin{equation}
		\label{eq:ap_ddtyot3}
		\|\nabla \bm{y}_{0,t}^{\bm{x}}\|^2 \leq \exp\left[ -2 \int_0^t \left(a(T-s)+\frac{b^2(T-s)}{\sigma^2(T-s)}-\frac{b^2(T-s)\mu^2(T-s)\text{diam}(\mathcal{X})^2}{2\sigma^4(T-s)}\right) ds \right].
	\end{equation}
	Note that
	\begin{equation}
		\nonumber
		\left(\ln \mu(t)\right)^{\prime} = a(t),\quad \left(\ln \lambda(t)\right)^{\prime} = \frac{b^2(t)}{\sigma^2(t)},\quad \lambda^{\prime}(t) = \frac{b^2(t)}{\mu^2(t)},\quad \left(\frac{1}{\lambda(t)}\right)^{\prime}=-\frac{\lambda^{\prime}(t)}{\lambda^2(t)}=-\frac{b^2(t)\mu^2(t)}{\sigma^4(t)},
	\end{equation}
	then we have
	\begin{equation}
		\label{eq:ap_ddtyot4}
		\begin{aligned}
			&\int_0^t \left(a(T-s)+\frac{b^2(T-s)}{\sigma^2(T-s)}-\frac{b^2(T-s)\mu^2(T-s)\text{diam}(\mathcal{X})^2}{2\sigma^4(T-s)}\right) ds \\
			=& \left\{\ln\bigl(\mu(s) \lambda(s)\bigr) + \frac{\operatorname{diam}^2(\mathcal{X})}{2} \frac{1}{\lambda(s)}\right\} \bigg|_{s=T-t}^{T} = \ln \left(\frac{\mu(T)\lambda(T)}{\mu(T-t)\lambda(T-t)}\right) + \frac{\operatorname{diam}^2(\mathcal{X})}{2} \left(\frac{1}{\lambda(T)}-\frac{1}{\lambda(T-t)}\right).
		\end{aligned}
	\end{equation}
	Substituting \ref{eq:ap_ddtyot4} into \eqref{eq:ap_ddtyot3}, we have
	\begin{equation}
		\nonumber
		\begin{aligned}
			\|\nabla \bm{y}_{0,t}^{\bm{x}}\|^2 \leq& \exp\left[ -2 \ln \left(\frac{\mu(T)\lambda(T)}{\mu(T-t)\lambda(T-t)}\right) -\operatorname{diam}^2(\mathcal{X}) \left(\frac{1}{\lambda(T)}-\frac{1}{\lambda(T-t)}\right) \right] \\
			\leq& \left(\frac{\mu(T-t)\lambda(T-t)}{\mu(T)\lambda(T)}\right)^2\exp\left[ \operatorname{diam}^2(\mathcal{X}) \left(\frac{1}{\lambda(T-t)}-\frac{1}{\lambda(T)}\right) \right] \\
			\leq& \exp\left[-2(m_{\lambda}-M_{\mu})t\right]\exp\left[  \frac{\operatorname{diam}^2(\mathcal{X})M_{\lambda}t }{\lambda(T-t)} \right] = \exp\left[-\bar{G}_K t\right],
		\end{aligned}
	\end{equation}
	where $\bar{G}_K = 2(m_{\lambda}-M_{\mu})-\operatorname{diam}^2(\mathcal{X})M_{\lambda}/\lambda(T-t)$. Note that as $t\to T$, $\bar{G}_K$ may be negative. We denote
	\begin{equation}
		\nonumber
		\bar{t}=T-\lambda^{-1}\left(\frac{\operatorname{diam}^2(\mathcal{X})M_{\lambda}}{m_{\lambda}-M_{\mu}}\right),
	\end{equation}
	then we have $\bar{G}_K\geq (m_{\lambda}-M_{\mu})\geq 1/2$ when $t\leq\bar{t}$ and then
	\begin{equation}
		\label{eq:ap_yotx2}
		\|\nabla \bm{y}_{0,t}^{\bm{x}}\|^2 \leq \exp\left[- t/2\right].
	\end{equation}
	When $t\geq\bar{t}$, we have
	\begin{equation}
		\label{eq:ap_yotx3}
		\begin{aligned}
			\|\nabla \bm{y}_{0,t}^{\bm{x}}\|^2 \leq& \exp\left[- \bar{t}/2\right] \exp\left[ -2 \ln \left(\frac{\mu(T-\bar{t})\lambda(T-\bar{t})}{\mu(T-t)\lambda(T-t)}\right) -\operatorname{diam}^2(\mathcal{X}) \left(\frac{1}{\lambda(T-\bar{t})}-\frac{1}{\lambda(T-t)}\right) \right] \\
			\leq& \exp\left[- \bar{t}/2\right]\left(\frac{\mu(T-t)\lambda(T-t)}{\mu(T-\bar{t})\lambda(T-\bar{t})}\right)^2\exp\left[ \operatorname{diam}^2(\mathcal{X}) \left(\frac{1}{\lambda(T-t)}-\frac{1}{\lambda(T-\bar{t})}\right) \right] \\
			\leq& \exp\left[- \bar{t}/2\right]\exp\left[-2(m_{\lambda}-M_{\mu})(t-\bar{t})\right]\exp\left[  \frac{\operatorname{diam}^2(\mathcal{X})M_{\lambda}(t-\bar{t})}{\lambda(T-t)} \right] \\
			\leq& \exp\left[- \bar{t}/2 +   \frac{\operatorname{diam}^2(\mathcal{X})M_{\lambda}(T-\bar{t})}{\lambda(T-\bar{t})}\exp\left[M_{\lambda}(T-\bar{t})\right]\right].
		\end{aligned}
	\end{equation}
	Denote 
	\begin{equation}
		\nonumber
		\begin{aligned}
			&\bar{t}=T-\lambda^{-1}\left(\frac{\operatorname{diam}^2(\mathcal{X})M_{\lambda}}{m_{\lambda}-M_{\mu}}\right),\\
			&G_K = \bar{t}/2 - \left(m_{\lambda}-M_{\mu}\right)(T-\bar{t})\exp\left[M_{\lambda}(T-\bar{t})\right].
		\end{aligned}
	\end{equation}
	Combine \eqref{eq:ap_yotx3} and \eqref{eq:ap_yotx2}, we have 
	\begin{equation}
		\label{eq:ap_yotx4}
		\|\nabla \bm{y}_{0,t}^{\bm{x}}\|^2 \leq
		\begin{cases}
			\exp\left[- t/2\right],\quad &t\leq\bar{t}, \\
			\exp\left[- G_K\right],\quad &t\geq\bar{t}.
		\end{cases}
	\end{equation}
	
	Substituting \eqref{eq:ap_yotx4} into the tangent error, we obtain
	\begin{equation}
		\nonumber
		\|\bm{y}_{0,t}^{\bm{x}}-\bm{y}_{0,t}^{\bm{x}^{\prime}}\| \leq \int_0^1 \|\nabla \bm{y}_{0,t}^{\bm{z}_{r}}\| \, dr \cdot \|\bm{x}-\bm{x}^{\prime}\|\leq
		\begin{cases}
			\exp\left[- t/2\right]\|\bm{x}-\bm{x}^{\prime}\|,\quad &t\leq\bar{t}, \\
			\exp\left[- G_K\right]\|\bm{x}-\bm{x}^{\prime}\|,\quad &t\geq\bar{t}.
		\end{cases}
	\end{equation}
	By the definition of the 1-Wasserstein distance, we have that, for any $t\in[0,T]$
	\begin{equation}
		\nonumber
		\mathcal{W}_1(\delta_{\bm{x}} Q_t, \delta_{\bm{x}^{\prime}} Q_t) \leq \mathbb{E} \left[\|\bm{y}_{0,t}^{\bm{x}}-\bm{y}_{0,t}^{\bm{x}^{\prime}}\|\right]\leq 
		\begin{cases}
			\exp\left[- t/2\right]\|\bm{x}-\bm{x}^{\prime}\|,\quad &t\leq\bar{t}, \\
			\exp\left[- G_K\right]\|\bm{x}-\bm{x}^{\prime}\|,\quad &t\geq\bar{t}.
		\end{cases}
	\end{equation}
	
\end{proof}

Lemma \ref{lem:tangent} quantifies the norm of the tangent process and reveals the contractivity in the 1-Wasserstein distance, which is crucial for proving the convergence of the reverse sampling process.

Our next goal is to control $\Delta \bm{b}_s((\bar{\bm{y}}_{s,t}^{\bm{x}})_{t\in\left[s,T\right]})$ for any $\bm{x}\in\mathbb{R}^d$ and $s,t\in\left[0,T\right], s\leq t$. Recall that for $\bm{w}=\left(\bm{w}_v\right)_{v\in[t_k,t_{k+1}]}$, we have
\begin{equation}
	\begin{aligned}
		\bm{b}_u(\bm{w})=& -a(T-u)\bm{w}_u+b^2(T-u)\nabla \log \pi_{T-u}(\bm{w}_u), \\
		\bar{\bm{b}}_u(\bm{w})=&-\left[a(T-u)+\frac{b^2(T-u)}{\sigma^2(T-u)}\right]\bm{w}_u+\frac{b^2(T-u)\mu(T-u)}{\sigma^2(T-u)}\bm{P}_k(\bm{w}_{t_k}),
	\end{aligned}
\end{equation}
and $\Delta \bm{b}_u(\bm{w})= \bm{b}_u(\bm{w})-\bar{\bm{b}}_u(\bm{w})$. We introduce the intermediate drift functions that
\begin{equation}
	\label{eq:ap_b4}
	\begin{aligned}
		\bm{b}^{(1)}_u(\bm{w}) =& -a(T-u)\bm{w}_u+b^2(T-u)\nabla \log \pi_{T-u}(\bm{w}_u)= \bm{b}_u(\bm{w}) \\ 
		\bm{b}^{(2)}_u(\bm{w}) =& -a(T-t)\bm{w}_u+b^2(T-t)\bm{v}_{T-u}(\bm{w}_u) \\
		=& -\left[a(T-u)+\frac{b^2(T-u)}{\sigma^2(T-u)}\right]\bm{w}_u+ \frac{b^2(T-u)\mu(T-u)}{\sigma^2(T-u)}\bm{P}_{T-u}(\bm{w}_{u}) \\
		\bm{b}^{(3)}_u(\bm{w}) =& -\left[a(T-u)+\frac{b^2(T-u)}{\sigma^2(T-u)}\right]\bm{w}_u+ \frac{b^2(T-u)\mu(T-u)}{\sigma^2(T-u)}\bm{P}_{T-t_k}(\bm{w}_{u}) \\
		\bm{b}^{(4)}_u(\bm{w}) =& -\left[a(T-u)+\frac{b^2(T-u)}{\sigma^2(T-u)}\right]\bm{w}_u+ \frac{b^2(T-u)\mu(T-u)}{\sigma^2(T-u)}\bm{P}_{T-t_k}(\bm{w}_{t_k}) = \bar{\bm{b}}_u(\bm{w}),
	\end{aligned}
\end{equation}
where 
\begin{equation}
	\begin{aligned}
		&\bm{P}_{T-u}(\bm{w}_{u})=\bm{\phi}_{\bm{\theta}^*}\left(\frac{\bm{w}_{u}}{\mu(T-u)}-\lambda(T-u)\beta \nabla f(\frac{\bm{w}_{u}}{\mu(T-u)}),\lambda(T-u)\right), \\
		&\bm{v}_{T-u}(\bm{w}_u) = \frac{\mu(T-u)\bm{P}_{T-u}-\bm{w}_u}{\mu^2(T-u)\lambda(T-u)}.
	\end{aligned}
\end{equation}

The discrepancy between the ideal reverse drift $\bm{b}_u$ and the discretized drift $\bar{\bm{b}}_u$ can be decomposed into three parts, corresponding to score approximation $B_1$, proximal-time discretization $B_2$, and proximal-spatial discretization $B_3$, i.e., 
\begin{equation}
	\|\Delta \bm{b}_u(\bm{w})\|\leq \underbrace{\|\bm{b}^{(1)}_u(\bm{w})-\bm{b}^{(2)}_u(\bm{w})\|}_{B_1}+\underbrace{\|\bm{b}^{(2)}_u(\bm{w})-\bm{b}^{(3)}_u(\bm{w})\|}_{B_2}+\underbrace{\|\bm{b}^{(3)}_u(\bm{w})-\bm{b}^{(4)}_u(\bm{w})\|}_{B_3}.
\end{equation}

The following lemma quantifies each part.
\begin{lemma}
	\normalfont
	\label{lem:tangent_b}
	Under Assumption \ref{assumption1} and \ref{assumption2}, for any $s,u\in[0,T]$ and $\bm{x}\in\mathbb{R}^d$, we have 
	\begin{equation}
		\label{eq:ap_edbu}
		\mathbb{E}\left[\|\Delta \bm{b}_u((\bar{\bm{y}}_{s,v}^{\bm{x}})_{v\in\left[s,T\right]})\|\right] \leq W_{b,1}\gamma_k + W_{b,2} M + W_{b,3}\gamma_k^{1/2},
	\end{equation}
	where
	\begin{equation}
		\begin{aligned}
			W_{b,1}&= \exp\left\{M_{\mu}T\right\}\left[\left(1 + \beta \bar{\lambda} L+m_{\lambda}\right)\sqrt{R_k} +m_{\lambda} R_p\right] + \beta G_f\left( \bar{\lambda} M_{\lambda}+m_{\lambda}\right)\\
			W_{b,2}&=M_{\lambda}(\sqrt{R_k}+1) \\
			W_{b,3}&=\exp\left\{M_{\mu}T\right\}M_{\lambda}(1 + \bar{\lambda}\beta L) \sqrt{A^{\prime}+B^{\prime}}.
		\end{aligned}
	\end{equation}
\end{lemma}

\begin{proof}
	\noindent\textbf{Term $B_1$: score approximation error.}
	
	From the definition in \eqref{eq:ap_b4}, we have
	\begin{equation}
		\label{eq:ap_b12}
		\|\bm{b}^{(1)}_u(\bm{w})-\bm{b}^{(2)}_u(\bm{w})\|=b^2(T-u)\|\nabla \log p_{T-u}(\bm{w}_u)-\bm{v}_{T-u}(\bm{w}_u)\|.
	\end{equation}
	Note that from the result in Appendix \ref{ap:score} and Theorem \ref{thm:score_est}, we have the score approximation error
	\begin{equation}
		\label{eq:ap_b122}
		\|\nabla \log p_{T-u}(\bm{w}_u)-\bm{v}_{T-u}(\bm{w}_u)\|\leq  \frac{M\left(1 + \|\bm{w}_u\|\right)}{\sigma^2(T-u)}.
	\end{equation}
	Substituting \eqref{eq:ap_b122} into the above \eqref{eq:ap_b12}, we have
	\begin{equation}
		\|\bm{b}^{(1)}_u(\bm{w})-\bm{b}^{(2)}_u(\bm{w})\|\leq \frac{M b^2(T-u)}{\sigma^2(T-u)}\left(1 + \|\bm{w}_u\|\right).
	\end{equation}
	
	\noindent\textbf{Term $B_2$: proximal-time discretization error.}
	Next, from the definition in \eqref{eq:ap_b4}, we have
	\begin{equation}
		\label{eq:ap_b23}
		\|\bm{b}^{(2)}_u(\bm{w})-\bm{b}^{(3)}_u(\bm{w})\|= \frac{b^2(T-u)\mu(T-u)}{\sigma^2(T-u)}\|\bm{P}_{T-u}(\bm{w}_{u})-\bm{P}_{T-t_k}(\bm{w}_{u})\|.
	\end{equation}
	Denote $\bm{A}(s,\bm{w}_{u}) = \frac{\bm{w}_{u}}{\mu(T-u)}-\lambda(T-u)\beta \nabla f(\frac{\bm{w}_{u}}{\mu(T-u)})$ where $s=T-u$ and $s_k=T-t_k$. Then we have
	\begin{equation}
		\label{eq:ap_b232}
		\|\bm{P}_{T-u}(\bm{w}_{u})-\bm{P}_{T-t_k}(\bm{w}_{u})\|= \|\text{Prox}_f^{\lambda(s)}(\bm{A}(s,\bm{w}_{u}))-\text{Prox}_f^{\lambda(s_k)}(\bm{A}(s_k,\bm{w}_{u}))\|.
	\end{equation}
	Let $\bm{p}_s = \operatorname{Prox}_{f}^{\lambda(s)}(A(s, \bm{w}_u))$ and $\bm{p}_{s_k} = \operatorname{Prox}_{f}^{\lambda(s_k)}(A(s_k, \bm{w}_u))$.
	By the optimality conditions of the proximal operator, we have
	\begin{equation}
		\nonumber
		\begin{aligned}
			&\bm{p}_s = \bm{A}(s, \bm{w}_u) - \lambda(s) \nabla f(\bm{p}_s),\\
			&\bm{p}_{s_k} = \bm{A}(s_k, \bm{w}_u) - \lambda(s_k) \nabla f(\bm{p}_{s_k}).
		\end{aligned}
	\end{equation}
	Taking difference, we have
	\begin{equation}
		\nonumber
		\bm{p}_s - \bm{p}_{s_k} = \bigl( A(s, \bm{w}_u) - A(s_k, \bm{w}_u) \bigr) - \bigl( \lambda(s) \nabla f(\bm{p}_s) - \lambda(s_k) \nabla f(\bm{p}_{s_k}) \bigr),
	\end{equation}
	and then
	\begin{equation}
		\label{eq:ap_psk}
		\bm{p}_s - \bm{p}_{s_k} + \lambda(s) \bigl( \nabla f(\bm{p}_s) - \nabla f(\bm{p}_{s_k}) \bigr) + \bigl( \lambda(s) - \lambda(s_k) \bigr) \nabla f(\bm{p}_{s_k}) = \bm{A}(s, \bm{w}_u) - \bm{A}(s_k, \bm{w}_u).
	\end{equation}
	Taking the inner product of \eqref{eq:ap_psk} with $\bm{p}_s - \bm{p}_{s_k}$, we have
	\begin{equation}
		\label{eq:ap_ppsk}
		\begin{aligned}
			&\langle \bm{p}_s - \bm{p}_{s_k}, \bm{p}_s - \bm{p}_{s_k} \rangle + \lambda(s) \langle \nabla f(\bm{p}_s) - \nabla f(\bm{p}_{s_k}), \bm{p}_s - \bm{p}_{s_k} \rangle
			+ \bigl( \lambda(s) - \lambda(s_k) \bigr) \langle \nabla f(\bm{p}_{s_k}), \bm{p}_s - \bm{p}_{s_k} \rangle \\
			&= \langle \bm{A}(s, \bm{w}_u) - \bm{A}(s_k, \bm{w}_u), \bm{p}_s - \bm{p}_{s_k} \rangle.
		\end{aligned}
	\end{equation}
	Since $f$ is $m$-strongly convex, we have
	\begin{equation}
		\label{eq:ap_fpsk}
		\langle \nabla f(\bm{p}_s) - \nabla f(\bm{p}_{s_k}), \bm{p}_s - \bm{p}_{s_k} \rangle \geq m \|\bm{p}_s - \bm{p}_{s_k}\|^2.
	\end{equation}
	Applying \eqref{eq:ap_fpsk} and the Cauchy–Schwarz inequality to \eqref{eq:ap_ppsk} yields
	\begin{equation}
		\label{eq:ap_ppsk2}
		\|\bm{p}_s - \bm{p}_{s_k}\| \leq (1 + \lambda(s) m) \|\bm{p}_s - \bm{p}_{s_k}\| \leq \|\bm{A}(s, \bm{w}_u) - \bm{A}(s_k, \bm{w}_u)\| + |\lambda(s) - \lambda(s_k)| \cdot \|\nabla f(\bm{p}_{s_k})\|.
	\end{equation}
	Next we bound $\|\bm{A}(s, \bm{w}_u) - \bm{A}(s_k, \bm{w}_u)\|$. From the definition, we have
	\begin{equation}
		\nonumber
		\begin{aligned}
			&\bm{A}(s, \bm{w}_u) - \bm{A}(s_k, \bm{w}_u) = \bm{w}_u \left( \frac{1}{\mu(s)} - \frac{1}{\mu(s_k)} \right) - \beta \left[ \lambda(s) \nabla f\left( \frac{\bm{w}_u}{\mu(s)} \right) - \lambda(s_k) \nabla f\left( \frac{\bm{w}_u}{\mu(s_k)} \right) \right] \\
			&= \bm{w}_u \left( \frac{1}{\mu(s)} - \frac{1}{\mu(s_k)} \right) - \beta \left[ (\lambda(s) - \lambda(s_k)) \nabla f\left( \frac{\bm{w}_u}{\mu(s)} \right) + \lambda(s_k) \left( \nabla f\left( \frac{\bm{w}_u}{\mu(s)} \right) - \nabla f\left( \frac{\bm{w}_u}{\mu(s_k)} \right) \right) \right].
		\end{aligned}
	\end{equation}
	Taking norms and using the triangle inequality, we have
	\begin{equation}
		\label{eq:ap_aswu}
		\begin{aligned}
			&\|\bm{A}(s, \bm{w}_u) - \bm{A}(s_k, \bm{w}_u)\|\leq \\
			&\|\bm{w}_u\| \cdot \left| \frac{1}{\mu(s)} - \frac{1}{\mu(s_k)} \right|+ \beta |\lambda(s) - \lambda(s_k)| \cdot \left\| \nabla f\left( \frac{\bm{w}_u}{\mu(s)} \right) \right\| + \beta \lambda(s_k) \left\| \nabla f\left( \frac{\bm{w}_u}{\mu(s)} \right) - \nabla f\left( \frac{\bm{w}_u}{\mu(s_k)} \right) \right\|.
		\end{aligned}
	\end{equation}
	Since $f$ is $L$-smooth, we have 
	\begin{equation}
		\label{eq:ap_aswu2}
		\left\| \nabla f\left( \frac{\bm{w}_u}{\mu(s)} \right) - \nabla f\left( \frac{\bm{w}_u}{\mu(s_k)} \right) \right\| \leq L \left\| \frac{\bm{w}_u}{\mu(s)} - \frac{\bm{w}_u}{\mu(s_k)} \right\| \leq L \|\bm{w}_u\| \left(\frac{1}{\mu(s_k)} - \frac{1}{\mu(s)}\right).
	\end{equation}
	Moreover, denote
	\begin{equation}
		\label{eq:ap_aswu3}
		G_f = \max_{\|\bm{x}\| \leq \|\bm{w}_u\|\exp\left\{M_{\mu}T\right\}} \|\nabla f(\bm{x})\|.
	\end{equation}
	Substituting \eqref{eq:ap_aswu2} and \eqref{eq:ap_aswu3} into \eqref{eq:ap_aswu}, then we have
	\begin{equation}
		\label{eq:ap_aswu4}
		\|\bm{A}(s, \bm{w}_u) - \bm{A}(s_k, \bm{w}_u)\| \leq \|\bm{w}_u\| \left( 1 + \beta \bar{\lambda} L \right) \left(\frac{1}{\mu(s_k)} - \frac{1}{\mu(s)}\right) + \beta G_f |\lambda(s) - \lambda(s_k)|.
	\end{equation}
	From the optimality condition, we have $\nabla f(\bm{p}_{s_k}) = \frac{\bm{A}(s_k, \bm{w}_u) - \bm{p}_{s_k}}{\lambda(s_k)}$, and then
	\begin{equation}
		\label{eq:ap_fpsk2}
		\|\nabla f(\bm{p}_{s_k})\| \leq \frac{\|\bm{A}(s_k, \bm{w}_u)\| + \|\bm{p}_{s_k}\|}{\lambda(s_k)}\leq \frac{\|\bm{w}_u\| + \|\bm{p}_{s_k}\|}{\mu(s_k)\lambda(s_k)}+\frac{\beta G_f}{\lambda(s_k)}.
	\end{equation}
	Since $f$ is $m$-strongly convex, there exists a lower bound, and the proximal minimization problem satisfies
	\begin{equation}
		\nonumber
		f(\bm{p}_{s_k}) + \frac{1}{2\lambda(s_k)} \|\bm{p}_{s_k} - \bm{A}(s_k, \bm{w}_u)\|^2 \leq f(\bm{A}(s_k, \bm{w}_u)).
	\end{equation}
	Since $\bm{A}(s_k, \bm{w}_u)$ is bounded, $f(\bm{A}(s_k, \bm{w}_u))$ is bounded. Therefore, $\|\bm{p}_{s_k}\|$ is bounded, i.e., 
	\begin{equation}
		\label{eq:ap_fpsk3}
		\|\bm{p}_{s_k}\|\leq R_p.
	\end{equation}
	Thus, substituting \eqref{eq:ap_fpsk3} into \eqref{eq:ap_fpsk2}, we obtain
	\begin{equation}
		\label{eq:ap_fpsk4}
		\|\nabla f(\bm{p}_{s_k})\| \leq \frac{\|\bm{w}_u\| + R_p}{\mu(s_k)\lambda(s_k)}+\frac{\beta G_f}{\lambda(s_k)}.
	\end{equation}
	Combine \eqref{eq:ap_aswu4} and \eqref{eq:ap_fpsk4} with \eqref{eq:ap_ppsk2} we have
	\begin{equation}
		\nonumber
		\begin{aligned}
			\|\bm{p}_s - \bm{p}_{s_k}\| \leq &\|\bm{w}_u\| \left( 1 + \beta \bar{\lambda} L \right) \left(\frac{1}{\mu(s_k)} - \frac{1}{\mu(s)}\right) + \beta G_f \left(\lambda(s_k) - \lambda(s)\right) + \left(1 - \frac{\lambda(s)}{\lambda(s_k)}\right) \left(\frac{\|\bm{w}_u\| + R_p}{\mu(s_k)}+\beta G_f\right)\\
			\leq & \frac{\|\bm{w}_u\| \left( 1 + \beta \bar{\lambda} L \right)}{\mu(T-t_k)} \gamma_k + \beta G_f \bar{\lambda} M_{\lambda} \gamma_k + \left(\frac{\|\bm{w}_u\| + R_p}{\mu(T-t_k)}+\beta G_f\right)\left(1 - \exp\left\{-\gamma_k m_{\lambda}\right\}\right).
		\end{aligned}
	\end{equation}
	Denote
	\begin{equation}
		\nonumber
		G_{2,3} = \left( \|\bm{w}_u\| \left( 1 +m_{\lambda} + \beta \bar{\lambda} L \right)+m_{\lambda}R_p \right)\exp\left\{M_{\mu}T\right\} + \beta G_f \left(\bar{\lambda} M_{\lambda}+m_{\lambda}\right),
	\end{equation}
	then we have
	\begin{equation}
		\label{eq:ap_pspsk}
		\|\bm{p}_s - \bm{p}_{s_k}\| \leq G_{2,3} \gamma_k.
	\end{equation}
	Substituting \eqref{eq:ap_b232} and \eqref{eq:ap_pspsk} into \eqref{eq:ap_b23}, we obtain
	\begin{equation}
		\nonumber
		\|\bm{b}^{(2)}_u(\bm{w})-\bm{b}^{(3)}_u(\bm{w})\|\leq \frac{b^2(T-u)\mu(T-u)}{\sigma^2(T-u)} G_{2,3} \gamma_k.
	\end{equation}

	\noindent\textbf{Term $B_3$: proximal-spatial discretization error.}
	
	For the last term in \eqref{eq:ap_b4}, we have
	\begin{equation}
		\label{eq:ap_b34}
		\|\bm{b}^{(3)}_u(\bm{w})-\bm{b}^{(4)}_u(\bm{w})\|= \frac{b^2(T-u)\mu(T-u)}{\sigma^2(T-u)}\|\bm{P}_{T-t_k}(\bm{w}_{u})-\bm{P}_{T-t_k}(\bm{w}_{t_k})\|,
	\end{equation}
	and 
	\begin{equation}
		\nonumber
		\|\bm{P}_{T-t_k}(\bm{w}_{u})-\bm{P}_{T-t_k}(\bm{w}_{t_k})\|= \|\text{Prox}_f^{\lambda(s_k)}(\bm{A}(s_k,\bm{w}_{u}))-\text{Prox}_f^{\lambda(s_k)}(\bm{A}(s_k,\bm{w}_{t_k}))\| .
	\end{equation}
	Let $s_k = T - t_k$, we will show that $\bm{P}_{s_k}(\bm{w})$ is Lipschitz continuous with respect to $\bm{w}$. 
	
	Since $f$ is proper, closed, and convex, the proximal operator $\operatorname{Prox}_{f}^{\lambda(s_k)}$ is nonexpansive, i.e.,
	\begin{equation}
		\nonumber
		\|\operatorname{Prox}_{f}^{\lambda(s_k)}(\bm{x}) - \operatorname{Prox}_{f}^{\lambda(s_k)}(\bm{y})\| \leq \|\bm{x} - \bm{y}\|, \quad \forall~\bm{x}, \bm{y}.
	\end{equation}
	Then we have
	\begin{equation}
		\label{eq:ap_askw2}
		\|\bm{P}_{s_k}(\bm{w}_{u})-\bm{P}_{s_k}(\bm{w}_{t_k})\|\leq \|\bm{A}(s_k,\bm{w}_{u})-\bm{A}(s_k,\bm{w}_{t_k})\|.
	\end{equation}
	For any $\bm{w}_1, \bm{w}_2$,
	\begin{equation}
		\label{eq:ap_askw}
		\begin{aligned}
			\|\bm{A}(s_k,\bm{w}_{1}) - \bm{A}(s_k,\bm{w}_{2})\| 
			&= \left\| \frac{\bm{w}_1 - \bm{w}_2}{\mu(s_k)} - \lambda(s_k)\beta \left( \nabla f\left( \frac{\bm{w}_1}{\mu(s_k)} \right) - \nabla f\left( \frac{\bm{w}_2}{\mu(s_k)} \right) \right) \right\| \\
			&\leq \frac{1}{\mu(s_k)} \|\bm{w}_1 - \bm{w}_2\| + \lambda(s_k)\beta \left\| \nabla f\left( \frac{\bm{w}_1}{\mu(s_k)} \right) - \nabla f\left( \frac{\bm{w}_2}{\mu(s_k)} \right) \right\|.
		\end{aligned}
	\end{equation}
	By the $L$-smoothness of $f$, we have
	\begin{equation}
		\label{eq:ap_askw1}
		\left\| \nabla f\left( \frac{\bm{w}_1}{\mu(s_k)} \right) - \nabla f\left( \frac{\bm{w}_2}{\mu(s_k)} \right) \right\| \leq L \left\| \frac{\bm{w}_1 - \bm{w}_2}{\mu(s_k)} \right\| = \frac{L}{\mu(s_k)} \|\bm{w}_1 - \bm{w}_2\|.
	\end{equation}
	Substituting \eqref{eq:ap_askw1} into \eqref{eq:ap_askw}, we have
	\begin{equation}
		\|\bm{A}(s_k,\bm{w}_{1}) - \bm{A}(s_k,\bm{w}_{2})\| \leq \left( \frac{1}{\mu(s_k)} + \lambda(s_k)\beta \frac{L}{\mu(s_k)} \right) \|\bm{w}_1 - \bm{w}_2\| \leq \left(1 + \bar{\lambda}\beta L\right) \exp\left\{M_{\mu}T\right\} \|\bm{w}_1 - \bm{w}_2\|.
	\end{equation}
	Combine \eqref{eq:ap_askw2} and \eqref{eq:ap_askw1} with \eqref{eq:ap_b34}, we have
	\begin{equation}
		\|\bm{b}^{(3)}_u(\bm{w})-\bm{b}^{(4)}_u(\bm{w})\|\leq \frac{b^2(T-u)\mu(T-u)}{\sigma^2(T-u)}\|\bm{A}(s_k,\bm{w}_{u}) - \bm{A}(s_k,\bm{w}_{t_k})\| \leq\frac{b^2(T-u)\mu(T-u)}{\sigma^2(T-u)} G_{3,4}\|\bm{w}_{u} - \bm{w}_{t_k}\|,
	\end{equation}
	where $G_{3,4} = \left(1 + \bar{\lambda}\beta L\right) \exp\left\{M_{\mu}T\right\}$.

	\noindent\textbf{Combining the three terms.}
	
	Summing the three estimates, we obtain
	\begin{equation}
		\label{eq:ap_dbuw}
		\begin{aligned}
			\|\Delta \bm{b}_u(\bm{w})\|\leq& \|\bm{b}^{(1)}_u(\bm{w})-\bm{b}^{(2)}_u(\bm{w})\|+\|\bm{b}^{(2)}_u(\bm{w})-\bm{b}^{(3)}_u(\bm{w})\|+\|\bm{b}^{(3)}_u(\bm{w})-\bm{b}^{(4)}_u(\bm{w})\|\\
			\leq&\frac{M b^2(T-u)}{\sigma^2(T-u)}\left(1 + \|\bm{w}_u\|\right) + \frac{b^2(T-u)\mu(T-u)}{\sigma^2(T-u)}G_{2,3}\gamma_k +\frac{b^2(T-u)\mu(T-u)}{\sigma^2(T-u)} G_{3,4}\|\bm{w}_{u} - \bm{w}_{t_k}\| \\
			\leq&\left(M_{\lambda}M + \left(1 + \beta \bar{\lambda} L+m_{\lambda}\right) \exp\left\{M_{\mu}T\right\}\gamma_k\right) \|\bm{w}_u\| + M_{\lambda}\left(1 + \bar{\lambda}\beta L\right)\exp\left\{M_{\mu}T\right\} \|\bm{w}_{u} - \bm{w}_{t_k}\| + \\
			& \left(\beta G_f\left( \bar{\lambda} M_{\lambda}+m_{\lambda}\right)+m_{\lambda} R_p\exp\left\{M_{\mu}T\right\}\right)\gamma_k + M_{\lambda}M.
		\end{aligned}
	\end{equation}
	
	Recall that from Lemma \ref{lem:reverse} and (\ref{eq:lemma_D5}), (\ref{eq:lemma_D6}), we have
	\begin{equation}
		\label{eq:ap_dbuw2}
		\begin{aligned}
			&\mathbb{E}[\|\bm{w}_{u}\|^2] \leq R_k := \bar{\lambda}d + B\left(1/A+\delta\right), \\
			&\mathbb{E}[\|\bm{w}_{u} - \bm{w}_{t_k}\|^2] \leq \left(A^{\prime}+B^{\prime}\right)\gamma_k,
		\end{aligned}
	\end{equation}
	where
	\begin{equation}
		\nonumber
		\begin{aligned}
			&A = \frac{1}{2\ln 2}\left(1+\frac{\gamma_k}{8\ln 2}\right)+ (2M+\eta)(1-\frac{\gamma_k}{4\ln 2})\exp\left\{M_{\mu}\delta\right\}M_{\lambda}+\exp\left\{2M_{\mu}\delta\right\}M_{\lambda}^2\gamma_k \left(2M^2+1+(2M+\eta)\bar{\lambda}\right) \\
			&B = \exp\left\{M_{\mu}\delta\right\}M_{\lambda}\frac{(M + \text{diam}(\mathcal{X}))^2}{\eta} \left(1-\frac{\gamma_k}{4\ln 2}\right) + \exp\left\{M_{\mu}\delta\right\}M_{\lambda}\gamma_k(2\eta+\bar{\lambda}) + \bar{\lambda}M_{\lambda}d,\\
			&A^{\prime} = \left\{\frac{\gamma_k}{(4\ln 2)^2} -\frac{2M+\eta}{4\ln 2}\exp\left\{M_{\mu}\delta\right\}M_{\lambda} + \exp\left\{2M_{\mu}\delta\right\}M_{\lambda}^2\gamma_k(2M^2+1+(2M+\eta)\bar{\lambda})\right\}R_k\\
			&B^{\prime} = -\frac{\gamma_k}{4\ln 2}\exp\left\{M_{\mu}\delta\right\}M_{\lambda}\frac{(M + \text{diam}(\mathcal{X}))^2}{\eta} + \exp\left\{2M_{\mu}\delta\right\}M_{\lambda}^2\gamma_k \frac{(2\eta+\bar{\lambda})(M + \text{diam}(\mathcal{X}))^2}{\eta} + \bar{\lambda}M_{\lambda}d.
		\end{aligned}
	\end{equation}
	Substituting \eqref{eq:ap_dbuw2} into \eqref{eq:ap_dbuw}, denote
	\begin{equation}
		\nonumber
		\begin{aligned}
			W_{b,1}&= \exp\left\{M_{\mu}T\right\}\left[\left(1 + \beta \bar{\lambda} L+m_{\lambda}\right)\sqrt{R_k} +m_{\lambda} R_p\right] + \beta G_f\left( \bar{\lambda} M_{\lambda}+m_{\lambda}\right)\\
			W_{b,2}&=M_{\lambda}(\sqrt{R_k}+1) \\
			W_{b,3}&=\exp\left\{M_{\mu}T\right\}M_{\lambda}(1 + \bar{\lambda}\beta L) \sqrt{A^{\prime}+B^{\prime}},
		\end{aligned}
	\end{equation}
	then we have
	\begin{equation}
		\nonumber
		\mathbb{E}\left[\|\Delta \bm{b}_u((\bar{\bm{y}}_{s,v}^{\bm{x}})_{v\in\left[s,T\right]})\|\right] \leq W_{b,1}\gamma_k + W_{b,2} M + W_{b,3}\gamma_k^{1/2}.
	\end{equation}

\end{proof}

Lemma \ref{lem:tangent_b} establishes a bound on the expected norm of the drift difference \(\Delta \bm{b}_u\). It quantifies how three types of local approximations contribute to the error in the drift of the interpolated reverse process. The bound reveals that the total local drift error scales as \(\mathcal{O}(M + \delta^{1/2})\) with weights increasing as \(\mathcal{O}(\exp\left\{T\right\})\), which directly informs the choice of stepsize \(\delta\) and score approximation accuracy \(M\) relative to the time length $T$. 

Compared with corresponding results in \cite{de2022convergence}, we obtain a tighter bound in \eqref{eq:ap_edbu} with the properties of proximal operator. This eliminates the $\mathcal{O}(\epsilon^{-2})$ scaling factor associated with the early-stopping time $\epsilon$, building the foundation for the strengthening of the convergence.

\subsection{Proof of Theorem \ref{thm:dis_conv}}
\label{ap:sub4}

This subsection presents the complete proof of Theorem~\ref{thm:dis_conv}, which establishes the non‑asymptotic convergence of the proximal diffusion sampler. 

For simplicity, we denote $\bar{\bm{x}}_{t_{k}}^{EI}$ as $\bm{y}_{k}$, and then we have
\begin{equation}
	\nonumber
	\bm{y}_{k+1} = \alpha_{1,k}\bm{y}_{k}+ \alpha_{2,k} \bm{P}_k + \alpha_{3,k}\bm{\xi}_k,\quad  \bm{\xi}_k\sim\mathcal{N}(\bm{0},I).
\end{equation}
Denote the distribution of $\bm{y}_{k}$ as $\pi_{\infty}R_k$, where we choose $\bm{y}_{0}\sim\pi_{\infty}$ and $R_k$ the transition kernel associated with the conditional distribution of $\bm{y}_{k}|\bm{y}_{0}$. 

Our goal is to bound the Wasserstein‑1 distance between $\pi_\infty R_K$ (the output of the algorithm after $K$ steps) and the target $\pi$. Denote $\left(P_t\right)_{t\in[0,T]}$ as the semi-group associated with $\bm{x}_t^{\rightarrow}$, and $\left(Q_t\right)_{t\in[0,T]}$ the semi-group associated with $\bm{x}_t^{\leftarrow}$. We decompose the error into two parts:
\begin{equation}
	\label{eq:ap_decom}
	\mathcal{W}_1(\pi_\infty R_K,\pi)\leq\underbrace{\mathcal{W}_1(\pi_\infty R_K,\pi_\infty Q_{T})}_{I_1}+\underbrace{\mathcal{W}_1(\pi_\infty Q_{T},\pi)}_{I_2},
\end{equation}
where the first term $I_1$ corresponds to the discretization and score estimation error. The second term $I_2$ corresponds to the convergence of the continuous-time exact reverse process. 

In contrast to existing convergence results in \cite{de2022convergence,li2023towards,li2024accelerating}, our framework eliminates the  early-stopping error in the convergence result, i.e., the third term $I_3$ in the following decomposition:
\begin{equation}
	\nonumber
	\mathcal{W}_1(\pi_\infty R_K,\pi)\leq\underbrace{\mathcal{W}_1(\pi_\infty R_K,\pi_\infty Q_{T-\epsilon})}_{I_1}+\underbrace{\mathcal{W}_1(\pi_\infty Q_{T-\epsilon},\pi P_{\epsilon})}_{I_2}+\underbrace{\mathcal{W}_1(\pi P_{\epsilon},\pi)}_{I_3}.
\end{equation} 
This error typically arises from partitioning the sampling interval into $[0,\epsilon]$ and $[\epsilon,T]$ to address nonsmoothness near $t=0$. As $t$ approaches zero, the true score $\nabla \log \pi_t(\bm{x}_t)$ diverges, creating a singularity that must be avoided. In addition, one cannot simply force $\epsilon\to0$, as the errors $I_1$ and $I_2$ are related with $\epsilon$ and tend to $\infty$ as $\epsilon\to0$.

In our framework, this issue is naturally resolved by replacing the traditional score with the well-defined Moreau score $\nabla \log \pi_t^{\lambda}(\bm{x}_t)$, leading to enhanced error bounds of $I_1$ and $I_2$.

\begin{proposition}
	\normalfont
	\label{pro:I1}
	Under Assumption \ref{assumption1} and \ref{assumption2}, we have
	\begin{equation}\label{eqn_W1_RQ}
		\mathcal{W}_1(\pi_\infty R_K,\pi_\infty Q_{T})\leq\left(1 + \exp\left[-G_K\right]\lambda^{-1}\left(\frac{\operatorname{diam}^2(\mathcal{X})M_{\lambda}}{m_{\lambda}-M_{\mu}}\right)\right) \left(C_{1,1}^{\prime}\gamma_k^{1/2} + C_{1,2}^{\prime} M\right),
	\end{equation}
	where 
	\begin{equation}
		\nonumber
		\begin{aligned}
			C_{1,1}^{\prime}= W_{b,1}\gamma_k^{1/2} + W_{b,3}, \quad C_{1,2}^{\prime}= W_{b,2},
		\end{aligned}
	\end{equation}
	and 
	\begin{equation}
		\nonumber
		\begin{aligned}
			W_{b,1}&= \exp\left\{M_{\mu}T\right\}\left[\left(1 + \beta \bar{\lambda} L+m_{\lambda}\right)\sqrt{R_k} +m_{\lambda} R_p\right] + \beta G_f\left( \bar{\lambda} M_{\lambda}+m_{\lambda}\right)\\
			W_{b,2}&=M_{\lambda}(\sqrt{R_k}+1) \\
			W_{b,3}&=\exp\left\{M_{\mu}T\right\}M_{\lambda}(1 + \bar{\lambda}\beta L) \sqrt{A^{\prime}+B^{\prime}}.
		\end{aligned}
	\end{equation}
\end{proposition}

\begin{proof}
	Consider the reverse diffusion process:
	\begin{equation}
		\nonumber
		d\bm{x}_t^{\leftarrow} = \left\{a(t)\bm{x}_t^{\leftarrow}-b^2(t)\nabla_{\bm{x}_t^{\leftarrow}} \log \pi_t(\bm{x}_t^{\leftarrow})\right\}dt + b(t)d\bar{\bm{B}}_t ,
	\end{equation}
	and the proximal diffusion sampling process:
	\begin{equation}
		\nonumber
		\bar{\bm{x}}_{k+1} = \alpha_{1,k}\bar{\bm{x}}_{k}+ \alpha_{2,k} \bm{P}_k + \alpha_{3,k} \bm{\xi}_k,
	\end{equation}
	where $\bm{\xi}_k\sim\mathcal{N}(\bm{0},I)$, $\tau_k=T-t_k$, $\bm{P}_k$, $\alpha_{1,k}$, $\alpha_{2,k}$, and $\alpha_{3,k}$ are defined in \eqref{eq:Pk} and \eqref{eq:coeff}, respectively.
	
	Recall that the proximal diffusion sampling process corresponds to the interpolation process on the interval $t\in\left[t_k,t_{k+1}\right]$:
	\begin{equation}
		\nonumber
		d\bar{\bm{x}}_t^{EI} = \left\{-\left[a(T-t)+\frac{b^2(T-t)}{\sigma^2(T-t)}\right]\bar{\bm{x}}_t^{EI}+\frac{b^2(T-t)\mu(T-t)}{\sigma^2(T-t)}\bm{P}_k\right\}dt + b(T-t)d\bm{B}_t.
	\end{equation}
	From Lemma \ref{lem:del}, we have
	\begin{equation}
		\nonumber
		\|\bm{x}_T^{\leftarrow}-\bar{\bm{x}}_{K}\|=\|\bm{x}_T^{\leftarrow}-\bar{\bm{x}}_{T}^{EI}\|\leq \int_{0}^{T} \|\nabla\bm{y}_{u,T}^{\bar{\bm{y}}_{0,u}}\| \|\Delta \bm{b}_u((\bar{\bm{y}}_{0,v})_{v\in\left[0,T\right]})\| du.
	\end{equation}
	Combining this result and Lemma \ref{lem:tangent}, we have
	\begin{equation}
		\begin{aligned}
			\|\bm{x}_T^{\leftarrow}-\bar{\bm{x}}_{K}\|\leq& \int_{0}^{\bar{t}} \|\nabla\bm{y}_{u,T}^{\bar{\bm{y}}_{0,u}}\| \|\Delta \bm{b}_u((\bar{\bm{y}}_{0,v})_{v\in\left[0,T\right]})\| du +\int_{\bar{t}}^{T} \|\nabla\bm{y}_{u,T}^{\bar{\bm{y}}_{0,u}}\| \|\Delta \bm{b}_u((\bar{\bm{y}}_{0,v})_{v\in\left[0,T\right]})\| du \\
			\leq& \int_{0}^{\bar{t}} \exp\left[- \bar{t}/2\right] \|\Delta \bm{b}_u((\bar{\bm{y}}_{0,v})_{v\in\left[0,T\right]})\| du +\int_{\bar{t}}^{t_K} \exp\left[- G_K\right] \|\Delta \bm{b}_u((\bar{\bm{y}}_{0,v})_{v\in\left[0,T\right]})\| du.
		\end{aligned}
	\end{equation}
	where
	\begin{equation}
		\nonumber
		\begin{aligned}
			&\bar{t}=T-\lambda^{-1}\left(\frac{\operatorname{diam}^2(\mathcal{X})M_{\lambda}}{m_{\lambda}-M_{\mu}}\right),\\
			&G_K = \bar{t}/2 - \left(m_{\lambda}-M_{\mu}\right)(T-\bar{t})\exp\left[M_{\lambda}(T-\bar{t})\right].
		\end{aligned}
	\end{equation}
	
	From the definition of Wasserstein-1 distance, we have
	\begin{equation}
		\label{eq:ap_ebu}
		\begin{aligned}
			\mathcal{W}_1(\pi_\infty R_K,\pi_\infty Q_{T})\leq& \mathbb{E}\left[\|\bm{x}_T^{\leftarrow}-\bar{\bm{x}}_{K}\|\right] \leq \left(\bar{t}\exp\left[- \bar{t}/2\right] + \exp\left[- G_K\right](t_K-\bar{t})\right) \mathbb{E}\left[\|\Delta \bm{b}_u((\bar{\bm{y}}_{0,v})_{v\in\left[0,T\right]})\|\right] du.
		\end{aligned}
	\end{equation}
	Then with Lemma \ref{lem:tangent_b}, we have
	\begin{equation}
		\label{eq:ap_ebu2}
		\mathbb{E}\left[\|\Delta \bm{b}_u((\bar{\bm{y}}_{s,v}^{\bm{x}})_{v\in\left[s,T\right]})\|\right] \leq W_{b,1}\gamma_k + W_{b,2} M + W_{b,3}\gamma_k^{1/2},
	\end{equation}
	where
	\begin{equation}
		\nonumber
		\begin{aligned}
			W_{b,1}&= \exp\left\{M_{\mu}T\right\}\left[\left(1 + \beta \bar{\lambda} L+m_{\lambda}\right)\sqrt{R_k} +m_{\lambda} R_p\right] + \beta G_f\left( \bar{\lambda} M_{\lambda}+m_{\lambda}\right)\\
			W_{b,2}&=M_{\lambda}(\sqrt{R_k}+1) \\
			W_{b,3}&=\exp\left\{M_{\mu}T\right\}M_{\lambda}(1 + \bar{\lambda}\beta L) \sqrt{A^{\prime}+B^{\prime}}.
		\end{aligned}
	\end{equation}
	Substituting \eqref{eq:ap_ebu2} into \eqref{eq:ap_ebu}, we have
	\begin{equation}
		\nonumber
		\begin{aligned}
			\mathcal{W}_1(\pi_\infty R_K,\pi_\infty Q_{T})\leq& \left(\bar{t}\exp\left[- \bar{t}/2\right] + \exp\left[-G_K\right](T-\bar{t})\right) \mathbb{E}\left[\|\Delta \bm{b}_u((\bar{\bm{y}}_{0,v})_{v\in\left[0,T\right]})\|\right] du \\
			\leq& \left(\bar{t}\exp\left[- \bar{t}/2\right] + \exp\left[-G_K\right](T-\bar{t})\right) \left(C_{1,1}^{\prime}\gamma_k^{1/2} + C_{1,2}^{\prime} M\right) \\
			\leq& \left(1 + \exp\left[-G_K\right]\lambda^{-1}\left(\frac{\operatorname{diam}^2(\mathcal{X})M_{\lambda}}{m_{\lambda}-M_{\mu}}\right)\right) \left(C_{1,1}^{\prime}\gamma_k^{1/2} + C_{1,2}^{\prime} M\right),
		\end{aligned}
	\end{equation}
	where $C_{1,1}^{\prime}$ and $C_{1,2}^{\prime}$ are defined in \eqref{eqn_W1_RQ}.
	
\end{proof}

Next, we control the second term $\mathcal{W}_1(\pi_\infty Q_{T},\pi)$. 
\begin{proposition}
	\normalfont
	\label{pro:I2}
	Under Assumption \ref{assumption1} and \ref{assumption2}, we have
	\begin{equation}
		\nonumber
		\mathcal{W}_1(\pi_\infty Q_{T},\pi)\leq  \exp\left[-G_K\right](\sqrt{d}+\mathrm{diam}(\mathcal{X}))\exp\left[-m_{\mu} T\right],
	\end{equation}
	where
	\begin{equation}
		\nonumber
		G_K = \bar{t}/2 - \left(m_{\lambda}-M_{\mu}\right)(T-\bar{t})\exp\left[M_{\lambda}(T-\bar{t})\right].
	\end{equation}
\end{proposition}

\begin{proof}
	Note that $\pi=\pi P_{T}Q_{T}$, with Lemma \ref{lem:tangent}, we have
	\begin{equation}
		\nonumber
		\mathcal{W}_1(\pi_\infty Q_{T},\pi ) = \mathcal{W}_1(\pi_\infty Q_{T},\pi P_{T}Q_{T})\leq \exp\left[- G_K\right]\mathcal{W}_1(\pi_\infty ,\pi P_{T}),
	\end{equation}
	where $G_K = \bar{t}/2 - \left(m_{\lambda}-M_{\mu}\right)(T-\bar{t})\exp\left[M_{\lambda}(T-\bar{t})\right]$.
	
	To control $\mathcal{W}_1(\pi_\infty ,\pi P_{T})$, we use a synchronous coupling, i.e. we set $(\bm{Y}_t,\bm{Z}_t)_{t\in[0,T]}$ such that
	\begin{equation}
		\nonumber
		\mathrm{d}\bm{Y}_{t}=a(t)\bm{Y}_{t}\mathrm{d}t+b(t)\mathrm{d}\bm{B}_{t},\quad\mathrm{d}\bm{Z}_{t}=a(t)\bm{Z}_{t}\mathrm{d}t+b(t)\mathrm{d}\bm{B}_{t},
	\end{equation}
	where $(\bm{B}_t)_{t\in[0,T]}$ is a $d$-dimensional Brownian motion and $\bm{Y}_0\sim\pi,\bm{Z}_0\sim\pi_\infty$. We have that for any $t\in[0,T]$, $\bm{Z}_t\sim\pi_\infty$. In addition, denoting $u_t=\mathbb{E}[\|\bm{Y}_t-\bm{Z}_t\|]$ for any $t\in[0,T]$, we have
	\begin{equation}
		\nonumber
		u_t\leq u_0\exp\left\{\int_0^t a(s)ds\right\} = u_0 \mu(t).
	\end{equation}
	Therefore, combining this result and above, we get that
	\begin{equation}
		\nonumber
		\mathcal{W}_1(\pi_\infty Q_{t_K},\pi P_{T-t_K})\leq \exp\left[- G_K\right]\mu(t_K)\mathcal{W}_1(\pi,\pi_\infty)\leq \exp\left[- G_K\right](\sqrt{d}+\mathrm{diam}(\mathcal{X}))\exp\left[-m_{\mu} t_K\right] .
	\end{equation}	
\end{proof}

Finally, we give our main convergence result, which corresponds to Theorem \ref{thm:dis_conv}.

\begin{theorem}
	\normalfont 
	Under Assumption \ref{assumption1} and \ref{assumption2}, let $\text{Law}(\bar{\bm{x}}_k) = \pi_{\infty} R_k$ denote the distribution of $\bar{\bm{x}}_k$ in (\ref{eq:sample}), where $\bar{\bm{x}}_0 \sim \pi_{\infty} := \lim_{t\to\infty} \text{Law}(\bm{x}_t^{\rightarrow})$ and $R_k$ is the transition kernel associated with $p(\bar{\bm{x}}_{k}|\bar{\bm{x}}_{0})$. Define the diameter of the constraint set as $\text{diam}(\mathcal{X})=\sup\left\{\|\bm{x}-\bm{x}^{\prime}\|:\bm{x},\bm{x}^{\prime}\in\mathcal{X} \right\}$. Suppose that
	\begin{equation}
		\nonumber
		T\geq\lambda^{-1}\left(\frac{\operatorname{diam}^2(\mathcal{X})M_{\lambda}}{m_{\lambda}-M_{\mu}}\right),\quad \beta\geq 2, \quad M\leq 1/6, \quad \delta \leq 2\ln2.
	\end{equation}
	Then there exist $D_1,D_2,D_3$ such that 
	\begin{equation}
		\nonumber
		\mathcal{W}_1\left(\text{Law}(\bar{\bm{x}}_K),\pi \right)\leq D_1 \exp\left[-m_{\mu} T\right] + D_2  M +D_3 \delta^{1/2},
	\end{equation}
	where 
	\begin{equation}
		\nonumber
		\begin{aligned}
			&D_1 = W_e(\sqrt{d}+\mathrm{diam}(\mathcal{X})), \\
			&D_2 = \left(1 + W_e\lambda^{-1}\left(\frac{\operatorname{diam}^2(\mathcal{X})M_{\lambda}}{m_{\lambda}-M_{\mu}}\right)\right)W_{\lambda}, \\
			&D_3 = \left(1 + W_e\lambda^{-1}\left(\frac{\operatorname{diam}^2(\mathcal{X})M_{\lambda}}{m_{\lambda}-M_{\mu}}\right)\right)\left(W_f + W_{\mu}\right),
		\end{aligned}
	\end{equation}
	and 
	\begin{equation}
		\nonumber
		\begin{aligned}
			&W_e = \exp\left[-\bar{t}/2 + \left(m_{\lambda}-M_{\mu}\right)(T-\bar{t})\exp\left[M_{\lambda}(T-\bar{t})\right]\right], \\
			&W_{\lambda} = \left(\sqrt{\bar{\lambda}d + B\left(1/A+\delta\right)}+1\right) M_{\lambda}, \\
			&W_{\mu} = \left(\left(1 + \beta \bar{\lambda} L+m_{\lambda}\right)\sqrt{R_k} + m_{\lambda} R_p + M_{\lambda}(1 + \bar{\lambda}\beta L)\sqrt{A^{\prime}+B^{\prime}}\right) \delta \exp\left\{M_{\mu}T\right\}, \\
			&W_f = \beta \delta \left( \bar{\lambda} M_{\lambda}+m_{\lambda}\right) \text{max}_{\|\bm{x}\| \leq \|\bm{w}_u\|}\ \|\nabla f(\bm{x})\|.
		\end{aligned}
	\end{equation}
\end{theorem}

\begin{proof}
	From \eqref{eq:ap_decom}, the total sampling error is decomposed into two parts.
	From the result in Proposition \ref{pro:I1}, we have
	\begin{equation}
		\label{eq:ap_decom2}
		I_1=\mathcal{W}_1(\pi_\infty R_K,\pi_\infty Q_{T})\leq\left(1 + \exp\left[-G_K\right]\lambda^{-1}\left(\frac{\operatorname{diam}^2(\mathcal{X})M_{\lambda}}{m_{\lambda}-M_{\mu}}\right)\right) \left(C_{1,1}^{\prime}\gamma_k^{1/2} + C_{1,2}^{\prime} M\right).
	\end{equation}
	
	Further, from Proposition \ref{pro:I2}, we have
	\begin{equation}
		\label{eq:ap_decom3}
		I_2=\mathcal{W}_1(\pi_\infty Q_{T},\pi)\leq  \exp\left[-G_K\right](\sqrt{d}+\mathrm{diam}(\mathcal{X}))\exp\left[-m_{\mu} T\right].
	\end{equation}
	
	Substituting \eqref{eq:ap_decom2} and \eqref{eq:ap_decom3} into \eqref{eq:ap_decom}, we obtain the final result: 
	\begin{equation}
		\nonumber
		\begin{aligned}			\mathcal{W}_1(\pi_\infty R_K,\pi)\leq &\mathcal{W}_1(\pi_\infty R_K,\pi_\infty Q_{T})+\mathcal{W}_1(\pi_\infty Q_{T},\pi ) 
			\leq  D_1 \exp\left[-m_{\mu} T\right] + D_2 M +D_3 \delta^{1/2}.
		\end{aligned}
	\end{equation}
\end{proof}

This theorem establishes the main convergence result of the work, providing a comprehensive error analysis for the proposed proximal diffusion sampling algorithm. It demonstrates that the total sampling error decays exponentially with the diffusion horizon \(T\), linearly with the score approximation error \(M\), and as the square root of the maximum stepsize \(\delta\).

\end{document}